\documentclass[11pt]{amsart}
\usepackage{amsmath, amssymb, amsthm,fullpage,graphicx}
\usepackage[all]{xypic}
\theoremstyle{definition}

\newtheorem{thm}[equation]{Theorem}
\newtheorem{prop}[equation]{Proposition}
\newtheorem{cor}[equation]{Corollary}
\newtheorem{lemma}[equation]{Lemma}
\newtheorem{conj}[equation]{Conjecture}
\newtheorem{ques}[equation]{Question}

\newtheorem{claim}[equation]{Claim}

\newtheorem{defn}[equation]{Definition}

\theoremstyle{remark}

\newtheorem{ntn}[equation]{Notation}

\newtheorem{rem}[equation]{Remark}

\makeatletter
\renewcommand{\subsection}{\@startsection{subsection}{2}{0pt}{-3ex
plus -1ex minus -0.2ex}{-2mm plus -0pt minus
-2pt}{\normalfont\bfseries}} 
\renewcommand{\subsubsection}{\@startsection{subsubsection}{2}{0pt}{-3ex
plus -1ex minus -0.2ex}{-2mm plus -0pt minus
-2pt}{\normalfont\bfseries}} \makeatother

\numberwithin{equation}{section}
\numberwithin{equation}{subsection}

\newdir^{ (}{{}*!/-5pt/@^{(}}

\renewcommand{\SS}{\S$\!$\S}
\newcommand{\Sat}{\operatorname{Sat}}
\newcommand{\Ind}{\operatorname{Ind}}

\newcommand{\fsf}{\mathsf{f}}
\newcommand{\Der}{\operatorname{Der}}

\newcommand{\Tor}{\operatorname{Tor}}

\newcommand{\tra}{\mathop{\rightarrow}}
\newcommand{\traa}{\mathop{\longrightarrow}}
\newcommand{\iso}{{\;\stackrel{_\sim}{\to}\;}}
\newcommand{\liso}{{\;\stackrel{_\sim}{\leftarrow}\;}}

\newcommand{\beq}{\begin{equation}\label}
\newcommand{\eeq}{\end{equation}}

\newcommand{\into}{\hookrightarrow}

\newcommand{\onto}{\twoheadrightarrow}

\newcommand{\pr}{\operatorname{pr}}
\newcommand{\ino}{\mathbf{1}_{\text{in}}}

\newcommand{\oX}{{\overline{X}}}

\newcommand{\ds}{\displaystyle}

\newcommand{\rep}{\operatorname{rep}}

\newcommand{\op}{\text{op}}
\newcommand{\gr}{\operatorname{gr}}
\newcommand{\Sym}{\operatorname{Sym}}

\def\O{\mathcal{O}}

\def\C{\mathbb{C}}
\def\k{\mathbf{k}}

\def\o{\otimes}

\def\br{\mathrm{br}}
\def\ad{\operatorname{ad}}

\def\End{\mathrm{End}}
\def\dq{{\overline{Q}}}
\def\dzq{{\overline{Q^0}}}
\def\dzqo{{\overline{Q^0_1}}}

\def\Id{\mathrm{Id}}
\def\cyc{\mathrm{cyc}}
\def\N{\mathbb{Z}_{\geq 0}}
\def\Q{\mathbb{Q}}
\def\F{\mathbb{F}}

\def\Span{\mathrm{Span}}

\def\dim{\operatorname{dim}}

\def\Z{{\mathbb Z}}

\def\Dis{\text{Dis}}
\def\1{\mathbf{1}}

\def\ldp{\mathopen{(\!(}} \def\rdp{\mathclose{)\!)}}

\def\ldb{\mathopen{\{\!\!\{}} \def\rdb{\mathclose{\}\!\!\}}}

\begin{document}
\title{Zeroth Hochschild homology of preprojective
algebras over the integers} 
\author{Travis Schedler} 
\maketitle

\begin{abstract}
  We determine the $\Z$-module structure of the preprojective algebra
  and its zeroth Hochschild homology, for any non-Dynkin quiver
  (and hence the structure working over any base commutative ring, of
  any characteristic).  This answers (and generalizes) a conjecture of
  Hesselholt and Rains, producing new $p$-torsion classes in degrees
  $2 p^\ell, \ell \geq 1$. We relate these classes by $p$-th power
  maps and interpret them in terms of the kernel of Verschiebung maps
  from noncommutative Witt theory. An important tool is a generalization of
  the Diamond Lemma to modules over commutative rings, which
  we give in the appendix.
\end{abstract}

\tableofcontents
\section{Introduction and Main Results} 
\label{is} This paper concerns the $\Z$-module $A_\cyc := A/[A,A] = A
/ \text{Span}(\{ab-ba \mid a,b \in A\})$ (also known as the zeroth Hochschild
homology, $HH_0(A)$) 
for certain graded algebras $A$ over $\Z$ related to quivers
and noncommutative Witt theory.  Here, the span means the integral
span.  For these algebras, a curious phenomenon emerges:

 (*) \emph{Over a characteristic $p$ field $\k$, the dimension of the
   $2p^\ell$-th graded part of $A_\cyc \o \k$ exceeds by one the
   dimension of the same part over a characteristic zero field, for all
   $\ell \geq 1$; in all other degrees, the two dimensions are the same.}

The simplest such algebra we study is
\begin{equation} \label{onevertex}
\Pi := \Z\langle x_1, x_2, \ldots, x_n, y_1, y_2, \ldots, y_n \rangle / (\sum_{i=1}^n [x_i, y_i]), \quad n \geq 2,
\end{equation}
where $\Z \langle t_1, \ldots, t_m \rangle$ is by definition the free
associative (noncommutative) algebra on $m$ generators $t_1, \ldots,
t_m$.  For this algebra $A = \Pi$, the phenomenon (*) was conjectured
by Hesselholt and Rains, motivated by noncommutative Witt theory.

Moreover, Hesselholt and Rains produced specific homogeneous classes
of $\Pi_\cyc$ that, if nonzero, are $p$-torsion, and conjectured that
they are all nonzero and generate the $p$-torsion.  These classes are
given as follows: One has the commutative diagram of quotients, where
$P := \Z \langle x_1, \ldots, x_n, y_1, \ldots, y_n \rangle$ is the
free algebra used above:
\begin{equation}
  \xymatrix{
    P \ar@{->>}[rr]^{\alpha} \ar@{->>}[d]^{\pi} & & \Pi \ar@{->>}[d]^{\pi} \\
    P_\cyc \ar@{->>}[rr]^{\overline{\alpha}} & & \Pi_\cyc.
  }
\end{equation}
Define
\begin{equation}
r := \sum_{i=1}^n [x_i, y_i].
\end{equation}
Then, the class $\pi(r^p)$ is a multiple of $p$ in
$P_\cyc$. Hence, one may consider the class
$r^{(p)} := \bar \alpha(\frac{1}{p} \pi(r^p)) \in \Pi_\cyc$, which evidently
satisfies $p r^{(p)} = 0$.  Similarly, $\pi(r^{p^\ell}) \in P_\cyc$
is a multiple of $p$ for all $\ell \geq 1$, and one may thus define
$r^{(p^\ell)} := \bar \alpha(\frac{1}{p} \pi(r^{p^\ell})) \in \Pi_\cyc$.
\begin{conj}[Hesselholt-Rains] \label{rh} For all primes $p \geq 2$,
  if $n \geq 2$, the classes $r^{(p^\ell)}$ are nonzero and generate the
  $p$-torsion of $\Pi_\cyc$.  There is no $p^2$-torsion in $\Pi_\cyc$.
\end{conj}
The conjecture implies (*). We give an elementary proof of this conjecture
in \S \ref{hrcsec} below, which is the first main result of this paper.

In terms of noncommutative Witt theory, the conjecture in
particular implies that the noncommutative $p$-adic Verschiebung map
is not, in general, injective, and the ghost components of its kernel
are given by the nonvanishing $p$-torsion classes in $\Pi_\cyc$
stated in the conjecture. We will explain this in more detail in \S
\ref{wittsec}.

\begin{rem}\label{ppwrrem} If we work over $\F_p$, then
  the classes $r^{(p^\ell)} := r^{(p^\ell)} \otimes 1 \in \Pi_\cyc
  \o_\Z \F_p$ are related by $p$-th power maps: generally, for
  every associative $\F_p$-algebra $A$, we have a well defined $p$-th
  power map $A_\cyc \rightarrow A_\cyc$, $[a]_\cyc \mapsto
  [a^p]_\cyc$, which is well-defined since $(a+b)^p - a^p - b^p \in
  [A,A]$ (modulo $p$).\footnote{In fact, as observed by Jacobson in
    the 1940's, $(a+b)^p \equiv a^p + b^p$ modulo the Lie algebra
    generated by $a$ and $b$.}  Then, it is easy to verify explicitly
  that $r^{(p^\ell)} = (r^{(p^{\ell-1})})^p$. Hence, working over
  $\Z$, all the $p$-torsion of $\Pi_\cyc$ is generated from
  $r^{(p)}+p\Pi_\cyc$ by taking $p$-th powers (i.e., replacing a class
  $[f]$ by a class $[f^p]$, which is well defined modulo $p\Pi_\cyc$)
  and sums of resulting classes (since the images of the $p$-torsion
  modulo $p\Pi_\cyc$ are all generated from $r^{(p)}$ by $p$-th power
  maps and sums). Note that $p\Pi_\cyc$ itself contains no $p$-torsion
  since $\Pi_\cyc$ has no $p^2$-torsion.
\end{rem}
\subsection{Quiver generalization}
A \emph{quiver} is an oriented graph whose edges are called arrows. We
will always assume the quiver to be connected, i.e., its underlying
undirected graph is connected.  We are interested in the quiver
generalization of the preceding, in which $P$ is replaced by the
algebra of paths in a quiver, and the $n \geq 2$ condition is replaced
by a certain non-Dynkin condition.
The algebra $\Pi$ is replaced by the \emph{preprojective algebra of
  $Q$}, which was originally defined by Gelfand and Ponomarev
\cite{GP} in the study of quiver representations.


In detail, for any quiver $Q$, let $P_Q$ be the algebra over $\Z$
(i.e., the ring) generated by paths in $Q$, with concatenation as
multiplication, called the path algebra (later on we will also work
over a general commutative ring $\k$, in which case the path algebra
over $\k$ is $P_Q \otimes_{\Z} \k$).  Let $Q_1$ denote the set of
arrows in the quiver $Q$ and $Q_0$ the set of vertices.  We will equip
$P_Q$ with the grading by length of paths, with $(P_Q)_m$ the subspace
spanned by paths of length $m$; thus $Q_0$ forms a basis for $(P_Q)_0$
and $Q_1$ forms a basis for $(P_Q)_1$.

Define the \emph{double quiver}, $\dq$, to be the quiver obtained from
$Q$ by adding a reverse arrow $a^*$ for every $a \in Q_1$, with the
same endpoints but the opposite orientation, and keeping the same set
of vertices.  We replace $P$ above with $P_\dq$, and recover the $P$
of the previous subsection in the special case when $Q$ has only one
vertex and $n$ arrows.

Let us assume that $Q$ is finite (i.e., it has finitely many arrows,
and hence, by connectivity, finitely many vertices). Define
\begin{equation}
r = \sum_{a \in Q_1} (a a^* - a^* a), \quad \Pi_Q := P_{\dq} / (r).
\end{equation}
Our main object of study in this paper is $(\Pi_Q)_\cyc$, and as
such we define notation for it:
\begin{equation}\label{d:lq}
\Lambda_Q := (\Pi_Q)_\cyc.
\end{equation}
We may consider again the commutative diagram
\begin{equation}
\xymatrix{
P_\dq \ar@{->>}[rr]^{\alpha} \ar@{->>}[d]^{\pi} & & \Pi_Q \ar@{->>}[d]^{\pi} \\
(P_\dq)_\cyc \ar@{->>}[rr]^{\overline{\alpha}} & & \Lambda_Q
}
\end{equation}

Then the Hesselholt-Rains conjecture generalizes as follows.  We say
that a quiver $Q$ is (ADE) Dynkin if the underlying undirected graph
is Dynkin of type $A_n, D_n$, or $E_n$ (with $n=|Q_0|$ equal to the
number of vertices); in particular this means there are no loops and
at most one arrow between any pair of vertices.  We say a quiver $Q$
is extended Dynkin if its underlying undirected graph is the extended
Dynkin diagram of a type ADE Dynkin diagram (in particular, this
implies that it obtained from the latter by adding an additional
vertex and one or two arrows). We consider the diagram with one vertex
and one arrow (from the vertex to itself) to be extended Dynkin, and
call it type $\tilde A_0$. Thus, the algebras $P$ and $\Pi$ from the
previous subsection are path and preprojective algebras of an extended
Dynkin quiver of type $\tilde A_0$ in the case $n=1$ and of a Dynkin
quiver of type $A_1$ in the case $n=0$, whereas for $n \geq 2$ the
corresponding quiver is neither Dynkin nor extended Dynkin.
\begin{thm} \label{mt} For all primes $p \geq 2$, if $Q$ is non-Dynkin
  and non-extended Dynkin, the classes $r^{(p^\ell)} :=
  \bar \alpha(\frac{1}{p} \pi(r^{(p^\ell)}))$ are nonzero and generate the
  $p$-torsion of $\Lambda_Q$. There is no $p^2$-torsion.
\end{thm}
We also present a much more general question (Question \ref{rciques}),
based on conversations with P. Etingof, that asks whether, for
finitely-presented graded algebras over $\Z$ (or $\Z^{Q_0}$), for primes
in which they are ``asymptotic representation complete intersections
(RCIs)'' (see \cite{EG}) the new torsion is generated by classes of
the above form, for relations $r$ which lie in the integral span of
commutators modulo $p$.

\begin{rem}As in Remark \ref{ppwrrem}, over $\F_p$ one has the
  statement $(r^{(p^\ell)}) = (r^{(p)})^{p^{\ell-1}}$ using the $p$-th
  power maps on $\Lambda_Q \o_\Z \F_p$, so all the $p$-torsion of
  $\Lambda_Q$ is generated, modulo $p$, by $r^{(p)}$ using $p$-th
  power maps and sums.
\end{rem}
\subsection{Strategy of proof of Theorem \ref{mt}}\label{stratpfsec}
The main strategy for the proof of the theorem is to exploit an
extended Dynkin subquiver $Q^0 \subsetneq Q$, i.e., a quiver such that
$Q^0_0:=(Q^0)_0 \subseteq Q_0$ and $Q^0_1:=(Q^0)_1 \subseteq Q_1$. (It
is well-known that such a quiver exists whenever $Q$ is neither Dynkin
or extended Dynkin, although there may be more than one choice of it.)
We then use and develop facts about extended Dynkin quivers, and the
relationship between $\Pi_Q$ and $\Pi_{Q^0}$.  Let $Q \setminus Q^0$
denote the quiver with the same vertex set as $Q$, and with arrows
$Q_1 \setminus Q^0_1$ (thus, $(Q \setminus Q^0)_0 := Q_0$).

We prove that there is an isomorphism as $\Z$-modules, 
$\Pi_Q \cong \Pi_{Q^0}
*_{\Z^{Q^0_0}} \Pi_{Q \setminus Q^0, Q^0_0}$, where $\Pi_{Q \setminus
  Q^0, Q^0_0}$ is a partial preprojective algebra (defined in
\cite{EE}; see \S \ref{ss:ppa} below), an algebra over $\Z^{Q_0}$
(which in turn we consider as a $\Z^{Q_0^0}$-algebra in the canonical way).
The proof is then divided into three somewhat overlapping cases:
\begin{enumerate}
\item The case of primes $p$ which are good for the extended Dynkin
  quiver $Q^0$ (i.e., not a factor of the size of the corresponding
  finite group $\Gamma$ under the McKay correspondence), given in
  Theorem \ref{mclgam}.  To prove this, we use the well-known Morita
  equivalence $\Pi_{Q^0}\otimes_\Z \k \simeq \k[x,y] \rtimes \Gamma$
  for $\k$ is an algebraically closed field of characteristic $p$.
  This induces a Morita equivalence $\Pi_{Q,Q^0_0} \otimes_\Z \k
  \simeq (\k\langle x,y\rangle \rtimes \Gamma) *_{\k[\Gamma]}
  (\k[\Gamma] \otimes_{\Z^{Q^0_0}} \Pi_{Q \setminus Q^0, Q^0_0})$.  We
  then prove the theorem in essentially the same manner as in the
  one-vertex case (where $\Gamma=\{1\}$). The latter case, which is
  essentially the original Hesselholt-Rains conjecture, is proved in
  \S \ref{hrcsec} below in an elementary fashion using (a mild
  generalization of) the Diamond Lemma.
\item The case when $Q^0$ is of type $\tilde A$ or $\tilde D$, given
  in Theorems \ref{ant} and \ref{dnt}.  To prove these theorems, we
  use explicit integral bases for $\Pi_{Q^0}, \Pi_Q$ and their zeroth
  Hochschild homology modulo torsion. We explicitly present these
  bases and verify that they are bases using the Diamond Lemma for
  modules over a commutative ring (which is discussed in Appendix
  \ref{dla}).
\item The remaining cases where $Q^0$ is of type $\tilde E_n$ (for $n
  \in \{6,7,8\}$) and $p \leq 5$: these are proved in \S
  \ref{wdescptpfs}.  Here we need to prove a more refined statement,
  given in Theorem \ref{wdescpt}, which relies on $p$-th power maps.
  For the cases at hand, the proof follows via Theorem \ref{dedz},
  which computes $\Lambda_Q$ in the type $E$ Dynkin and extended
  Dynkin cases via straightforward computation using Gr\"obner
  generating sets (cf.~Appendix \ref{gbs} and Proposition \ref{gbp}
  therein), and Proposition \ref{ezphp}, which computes the zeroth
  Poisson homology of the necklace Lie algebra structure on
  $\Lambda_Q$ \cite{G,BLB,CBEG} in these cases.
\end{enumerate}
Above and below, we use the term ``Gr\"obner generating set'' since we
work over arbitrary commutative rings (such as $\Z$) and it is
slightly inconvenient for us to define the minimal such sets that,
when working over fields, are customarily called Gr\"obner bases (and
we do not need minimality here).

In the process, we obtain integral, rational, and characteristic $p$
bases for $\Pi_Q$ and $\Lambda_Q$ modulo torsion.  This relies on
the $\Z$-module decomposition $\Pi_Q \cong \Pi_{Q^0} *_{\Z^{Q^0_0}} \Pi_{Q \setminus Q^0,
  Q^0_0}$.
For $\Lambda_Q$, we can then (essentially) write classes as cyclic
words in $\Lambda_{Q^0}$ and $(\Pi_{Q \setminus Q^0, Q^0_0})_{\cyc}$.
For details, see \S \ref{ean} for the case $Q^0 = \tilde A_n$, \S
\ref{edn} for the case $Q^0 = \tilde D_n$, Propositions \ref{mip2} and
\ref{wdescp} for the case where $p$ is good for $Q^0$ as above (i.e.,
not a factor of the size of the group given by the McKay
correspondence); and for the general case see Theorem \ref{wdescpt}
and the preceding (note that only $Q^0 = \tilde E_n$ and $p \leq 5$
are not covered by the preceding cases). As observed above, these cases are overlapping,
but note that the bases we obtain do \emph{not} coincide in
overlapping cases.

The bases should be interesting in their own right. As one simple
application of the bases for $\Pi_Q$, we may deduce that $\Pi_Q$ is
torsion-free (which was proved in \cite{EE} for non-Dynkin quivers
using Gelfand-Kirillov dimension).  (This follows from much less work
than is required for the proof of Theorem \ref{mt} itself, where the
essential difficulty is in finding the $\Z$-module structure of
$\Lambda_Q$, rather than merely $\Pi_Q$.)

In the Dynkin and extended Dynkin cases, we also compute explicit
bases for $\Pi_Q$ and $\Lambda_Q$ modulo torsion, and give an
explicit description of the torsion of $\Lambda_Q$. Here, it turns
out that the torsion is finite, and the nonzero classes
$r^{(p^{\ell})}$ only occur in ``stably bad primes'': none for $\tilde
A_n$; $p = 2$ for $\tilde D_n$, $p \in \{2,3\}$ for $\tilde E_6,
\tilde E_7$, and $p \in \{2,3,5\}$ for $\tilde E_8$. The precise
result is Theorem \ref{dedz} (which refines a result of \cite{MOV},
which established the cases in which $\Lambda_Q \o \F_p$ vanishes).
Note that there is good reason why this torsion must be as described:
indeed, our proof of Theorem \ref{mt} via the method of the proof of
Theorem \ref{wdescpt} in \S \ref{s:mt-fp} could also be used backwards
to deduce, assuming the statement of Theorem \ref{mt} (which does not
single out any primes), the precise torsion structure of
$\Lambda_Q$ without computing it directly.
\begin{rem}
  In fact, there are precise ways in which $\Pi_Q$ 
is well-behaved in
  \emph{all} primes, unlike $\k[x,y] \rtimes \Gamma$ when $\k$
is a field of characteristic dividing $|\Gamma|$: for
  instance, $\Pi_Q$ 
  is a Calabi-Yau algebra over the base ring $\Z^{Q_0}$, 
in the sense that $\Pi_Q$ has a
  self-dual finitely-generated projective bimodule resolution of
  length two (see, e.g., \cite[(9.2.2)]{CBEG}). In particular, this
  implies that $\Pi_Q \otimes_\Z \k$ has global dimension two for
  all fields $\k$, unlike $\k[x,y] \rtimes
  \Gamma$, which has infinite global dimension when $\k$ is a field
  of characteristic dividing $|\Gamma|$.
\end{rem}



\subsection{Outline of the paper}
First, in \S \ref{wittsec}, we explain the motivation and
interpretation using noncommutative Witt theory.

Then, in \S \ref{hrcsec}, we prove Hesselholt and Rains's
conjecture.

In \S \ref{gpsec} we prove its generalization, Theorem \ref{mt},
in the case of good primes (not dividing $|\Gamma|$ where
$\Gamma < SL_2(\C)$ is associated to an extended Dynkin subquiver).

Next, \S \ref{nccisec} recalls the notion of NCCI algebras and
proves a general result we will need about them.

In \S \ref{hsrcisec}, we explain in detail the (suggestive)
Hilbert series formulas resulting from Theorem \ref{mt}. We also
pose a more general question on asymptotic RCI algebras in positive
characteristic (Question \ref{rciques}).

In the next crucial section, \S \ref{refppsec}, we prove one direction
of Theorem \ref{mt}: that the classes $r^{(p^\ell)}$ are nonzero
(Proposition \ref{rpnz}). We then proceed to state a refinement of the
main theorem (Theorem \ref{wdescpt}), using prime powers, which
implies Theorem \ref{mt}. The goal of the
remainder of the paper will be to prove this refinement.

In \S \ref{lqs1}, we study some algebraic structures related to the
preprojective algebra that we will need. In particular, we define and
generalize the Lie bialgebra structure on $\Lambda_Q$, obtained as a
quotient of the necklace Lie algebra. We explain that it is actually a
Poisson algebra in the extended Dynkin case, modulo torsion (by
identifying $\Lambda_Q$ modulo torsion with the center of $\Pi_Q$,
which is a commutative algebra). We explain how $P_\dq$ is a ``free
product'' deformation of $\Pi_Q$, and prove that in the extended
Dynkin case it quantizes the Poisson bracket on the center coming from
the McKay correspondence.

In \SS \ref{ans} and \ref{dns}, we prove the main results in the cases
where there exists a subquiver $Q^0$ which is extended Dynkin of type
$\tilde A_n$ (Theorem \ref{ant}) or $\tilde D_n$ (Theorem \ref{dnt}).
These results imply Theorem \ref{wdescpt} and hence the main Theorem
\ref{mt} in these cases.

In preparation for the type $\tilde E_n$ cases, in \S \ref{ssqsec},
we prove some results on
preprojective algebras of star-shaped quivers.
Then, in \S \ref{ens}, we give explicit bases (via Gr\"obner generating sets)
for $\Pi_Q$ in the case of (extended) Dynkin quivers of
type $E$ (Proposition \ref{enprop}).  We also compute the Lie structure on
$\Lambda_Q$ (Proposition \ref{epbp}).  

Finally, we complete the proof of Theorem \ref{wdescpt} in \S \ref{s:mt-fp}.

In the appendix, we give a generalized version of the Diamond Lemma for modules
over commutative rings (which we use to compute bases).

\begin{rem}In this paper we make use of two structures on
  $(P_{\overline{Q}})_{\cyc} \otimes \F_p$ and $\Lambda_Q \otimes
  \F_p$: the $p$-th power maps $[f] \mapsto [f^p]$ discussed above,
  and the necklace Lie algebra structure of \cite{G,BLB} discussed in
  \S \ref{lqs1} below.  It is natural to ask if these structures are
  compatible in any way.  In particular, one might ask if they form a
  restricted Lie algebra. This is, however, not true, because in the
  latter case an axiom of restricted Lie algebras would be $\ad [f^p]
  = (\ad f)^p$, but the LHS is an operator of degree $p|f|-2$ whereas
  the RHS is an operator of degree $p(|f|-2)$, which are not equal.
  We could not find any compatibility axiom which these structures enjoy.
\end{rem}

\subsection{Notation and Definitions} \label{ns} We will always let
$\k$ denote a base commutative ring and will work with algebras and
their modules over $\k$.  When considering a quotient $A/B$ for $A$ an
algebra, we will usually only require $B$ to be a graded
$\k$-submodule rather than an ideal (so that the quotient is only a
graded $\k$-module).

To avoid confusion with $\k$-submodules whose definition requires
parentheses, the ideal generated by elements will henceforth be
denoted with double parentheses: $\ldp r \rdp=$ the ideal generated by
$r$.

Given a $\k$-module $M$ and a subset $S \subseteq M$, we will let
$\langle S \rangle$ denote the $\k$-linear span of the set $S$.  This
is a completely different use of $\langle - \rangle$ than that for the
free algebra $\k \langle x_1, \ldots, x_n \rangle$ on indeterminates
$x_1, \ldots, x_n$, and the usage will be clear from the context.

Given a set $S$, we let the free $\k$-module generated by $S$ be denoted by
$\k \cdot S$.

We will use square braces to indicate that the expression inside is
taken up to cyclic permutations: i.e., we write cyclic words as $[a_1
a_2 \cdots a_m] = [a_2 a_3 \cdots a_m a_1]$.  We will use this in many
cases where it is not strictly necessary (i.e., where the context already
guarantees that the expression is invariant under cyclic permutations).


As stated earlier, a quiver $Q$ is a finite, directed, connected
graph, allowing loops and multiple edges. The edges are called arrows
and the set of arrows is denoted by $Q_1$, while the set of vertices
is denoted by $Q_0$.  Connected here means the underlying undirected
graph is connected. We will maintain the definition of the path
algebra $P_Q$ and the preprojective algebra $\Pi_Q$ above,
as well as the double quiver $\dq$. Note that, in $P_Q$, the product
$p_1 p_2$ of two paths is zero if $p_1$ does not terminate at the same
vertex at which $p_2$ begins.

For an arrow $a \in Q_1$, the reverse arrow is denoted $a^*$, and we also
use the notation $(a^*)^* := a$.  If an arrow $a$ goes from vertex
$i \in Q_0$ to $j \in Q_0$, we say $a: i \rightarrow j$, and set
$a_s = i; a_t = j$ (``s''=``source'', ``t''=``target'').

Let $\k^{Q_0} = \bigoplus_{i \in Q_0} \k$ denote the ring which, as a
$\k$-module, is the free $\k$-module $\k \cdot Q_0$ with basis $Q_0$, and
which has product $ij = \delta_{ij} i$. Then, we consider $\k \cdot
Q_1$ to be a $\k^{Q_0}$-bimodule with multiplication $iaj = \delta_{i a_s}
\delta_{j a_t} e$. One has $P_Q \otimes_\Z \k = T_{\k^{Q_0}} (\k \cdot Q_1)$.

The ``length'' of a path in a quiver is the number of arrows in the
path.  Similarly, the length of a line segment of arrows is the number
of arrows in the segment. As above, $P_Q$ and $\Pi_Q$ are graded by
path length.

When $A$ is a graded $\k$-module, we let $A_m$ denote the degree-$m$
component of $A$, so that $A = \bigoplus_m A_m$. The gradings we will
need will be by path length (hence nonnegative).

We will let $\Z/n$ denote $\Z/n\Z$ throughout.


\subsection{Acknowledgements}
The author is very grateful to Pavel Etingof for communicating Rains
and Hesselholt's conjecture and for many useful discussions.  The
author also thanks Victor Ginzburg for discussions, Lars Hesselholt
for explaining patiently his work on Witt theory and much of \S
\ref{wittsec} in detail, Eric Rains and Pavel Etingof for assistance
with and access to Magma, and the anonymous referee for many helpful
suggestions and corrections and his/her great patience during
revisions. This work was partially supported by an NSF GRF and an AIM
five-year fellowship.

\section{Relations to Witt theory}\label{wittsec}
In \cite{H} (see also \cite{H2}), Hesselholt defined the abelian group
of $p$-typical Witt vectors $W(A)$ for any noncommutative ring $A$,
and the vectors of a given length $\ell$, $W_\ell(A)$.  In the case
where $A$ is commutative, this reduces to the usual Witt vectors
\cite{W} (which form a commutative ring in this case).  We recall
briefly their definition.

First, one defines the $W_{\ell}(A)$ and the restriction maps $R: W_{\ell+1}(A) \onto W_{\ell}(A)$, and then $W(A)$ will be the inverse limit of the $W_\ell(A)$. 
Each $W_{\ell}(A)$ is defined as a certain quotient of $A^{\ell}$, such that there
exist restriction maps $R$ completing the commutative diagrams
(where the left vertical arrow is the projection to the
first $\ell$ components):
\begin{equation} \label{wittcpts}
\xymatrix{
A^{\ell+1} \ar@{->>}[rr] \ar@{->>}[d] & & W_{\ell+1}(A) \ar@{->>}[d]^R \\
A^{\ell} \ar@{->>}[rr] & & W_\ell(A).
}
\end{equation}
We therefore represent elements of $W_\ell(A)$ (nonuniquely) by
coordinates $(a_0, \ldots, a_{\ell-1})$. 

In terms of these coordinates, the sum and difference operations on
$W_\ell(A)$ are expressed using noncommutative versions of the usual
Witt polynomials for these operations.  In more detail, define first
the noncommutative \emph{ghost map},
$w: A^{\ell} \rightarrow A_\cyc^{\ell}$,
$w(a_0, \ldots, a_{\ell-1}) = ([a_0], [a_0^p + p a_1], [a_0^{p^2} + p a_1^p
+ p^2 a_2], \ldots)$.
There exist (nonunique) noncommutative polynomials
$s_i, d_i, i = 0, 1, \ldots$ in infinitely many variables
$(x_0, x_1, x_2, \ldots)$ such that, for all $\ell$,
$s = (s_0, s_1, \ldots, s_{\ell-1})$ and $d = (d_0, d_1, \ldots,
d_{\ell-1})$
are transported under $w$ to the usual sum and difference in
$A_\cyc^{\ell}$. Then, the $W_\ell(A)$ are the unique quotients (of
sets) of the $A^{\ell}$, compatible with \eqref{wittcpts}, such that
$s$ and $d$ descend to each $W_\ell(A)$, and each composition
$(\{0\}^{\ell-1} \times A) \subset A^{\ell} \onto W_\ell(A)$ has zero
fiber $\{0\}^{\ell-1} \times [A,A]$ (i.e., one may inductively define a
noncanonical bijection of sets $A_\cyc^\ell \cong
W_\ell(A)$).
Hesselholt showed that $W_\ell(A)$ then is an abelian group under $s$
and $d$ and the operations are independent of the choice of
noncommutative polynomials $s_i, d_i$.

In addition to defining $W(A)$ for associative algebras $A$,
Hesselholt generalized the Frobenius ($F$) and Verschiebung ($V$)
operators on $W(A)$ to the noncommutative setting.  These
have the form
$F: W_\ell(A) \rightarrow W_{\ell-1}(A)$ and
$V: W_\ell(A) \rightarrow W_{\ell+1}(A)$.

In terms of coordinates, the Verschiebung operator continues to have
the form
\begin{equation}\label{vefla}
V(a_0, a_1, \ldots, a_{\ell-1}) = (0, a_0, a_1, \ldots, a_{\ell-1}),
\end{equation}
and the Frobenius operator is the unique functorial operation
whose expression in terms of ghost components
$w(a_0, \ldots, a_{\ell-1}) = (w_0,w_1, \ldots, w_{\ell-1}) \in
A_\cyc^{\ell}$
is
\begin{equation}\label{frfla}
w(F(a_0, \ldots, a_{\ell-1})) = (p w_1, pw_2, \ldots, p w_{\ell-2}).
\end{equation}

While, in the commutative case, the Verschiebung operator is always
injective, it turns out this may not be the case in the noncommutative
setting.  Hesselholt defined functorial exact sequences,
\begin{equation}
HH_1(A) \tra^{\delta_\ell} W_\ell(A) \tra^V W_{\ell+1}(A),
\end{equation}
using the polynomials $s_i, d_i$. 
In particular, if
$\sum_{i=1}^n \bar x_i \o \bar y_i \in A \o A$ is a Hochschild
one-cycle (i.e.~$\sum_{i=1}^n [\bar x_i, \bar y_i] = 0$) for some
$\bar x_i, \bar y_i \in A$, one may consider the map
$\phi: \Pi \rightarrow A$ mapping $x_i \mapsto \bar x_i, y_i \mapsto 
\bar y_i$ (for $\Pi$ as in the beginning of \S \ref{is} with parameter $n$); then
one has a commutative diagram with exact rows,
\begin{equation}
\xymatrix{
HH_1(\Pi) \ar[r]^{\delta_\ell} \ar[d]^\phi & W_\ell(\Pi) \ar[r]^V \ar[d]^\phi & 
W_{\ell+1}(\Pi) \ar[d]^\phi \\
HH_1(A) \ar[r]^{\delta_\ell} & W_\ell(A) \ar[r]^V & W_{\ell+1}(A),
}
\end{equation}
so that
$\phi \circ \delta_\ell[\sum_i x_i \o y_i] = \delta_\ell [\sum_i \bar
x_i \o \bar y_i]$.
Hence, the kernel of the Verschiebung map is generated by
 all classes
that can be obtained as the image of the classes
$\delta_\ell[\sum_i x_i \o y_i]$ under algebra maps
$\Pi \rightarrow A$ (for all choices of $n$).

Furthermore, one has $W_1(A) = A/[A,A]$, and Hesselholt computed that
$\delta_1(\sum_i x_i \o y_i) = r^{(p)}$.  More generally,
\begin{equation}
w(\delta_\ell(\sum_i x_i \o y_i)) = (r^{(p)}, r^{(p^2)}, \ldots, r^{(p^\ell)}).
\end{equation}
Hence, Conjecture \ref{rh} says that, not only is the kernel of $V$
nonzero in general, but the universal elements of the kernel,
$\delta_\infty(\sum_{i=1}^n x_i \o y_i) \in W(\Pi)$, have all nonzero
ghost components (for $n \geq 2$).

Note that, by \eqref{vefla},\eqref{frfla}, the ghost components of any
element of $\ker(V)$ are $p$-torsion in $A_\cyc$.\footnote{In fact,
  $F \circ V$ is multiplication by $p$ in $W(A)$, so $\ker(V)$ is
  $p$-torsion in $W(A)$ as well.}
Conjecture \ref{rh} implies that, in fact, the ghost components of the
universal element of $\ker(V)$ above form a $\F_p$-basis for the
$p$-torsion in $\Pi_\cyc$.
\begin{rem}The quiver generalization of the above should be as
  follows: Given any set of vertices $Q_0$, one may consider
  noncommutative algebras over $\Z^{Q_0}$ instead of over $\Z$.  The
  Witt group itself is unchanged, only the category of algebras is
  changed.  Now, the universal elements of $\ker(V)$ should instead be
  written as $\delta_\infty(\sum_{a \in Q_1} a \otimes a^* - a^*
  \otimes a) \in W(\Pi_Q)$, replacing the parameter $n$ by the quiver
  $Q$.  We then see that, for non-Dynkin, non-extended Dynkin quivers,
  not only is the corresponding universal element nonzero, but all its
  ghost components are nonzero, and they form a $\F_p$-basis for the
  $p$-torsion of $\Lambda_Q$.
\end{rem}

\section{Hesselholt and Rains' conjecture}\label{hrcsec}
The purpose of this section is to prove the following main
combinatorial result, which in particular implies Conjecture \ref{rh}.
The proof of this lemma will be generalized later, to compute bases of
$\Pi$ and $\Lambda$ modulo torsion for quivers containing $\tilde A_n$
or $\tilde D_n$:
\begin{defn} \label{d:rep}
For a sequence $(x_1,\ldots,x_k) \in X^k$ for any set $X$, let
 $\text{per}(x_\cdot)$ be the period (the least positive integer such
 that $x_i = x_{i+\text{per}(x_{\cdot})}$, with indices taken modulo
 $k$), and let $\rep(x_\cdot) := k/\text{per}(x_{\cdot})$ be the number
 of cyclic permutations which fix the sequence (the size of the
 stabilizer in $\Z/k$ of the sequence), which we can call the ``number
of times the sequence repeats itself'', hence ``rep''. 
\end{defn}
Let $A:=F/R$ where $F:=\Z\langle x,y,r'\rangle$ is the free algebra
and $R:=\ldp r \rdp$ with $r := xy-yx + r'$. Let $V := A/[Ar'A,A]$.
Let $\hat V$ be the $\Z$-module obtained from $V$ by adjoining $[\frac{1}{p}
(r')^{p^{\ell}}]$ for all $p$ prime and $\ell \geq 1$, i.e., 
the quotient of $V
\oplus \Z \cdot \{t_p\}_{p \text{ prime}, \ell \geq 1}$ by the
relations $p \cdot t_p = [(r')^{p^\ell}]$ for all $p$ prime and $\ell \geq 1$.
Similarly
let $\widehat{F/[F,F]}$
and
$\widehat{A/[A,A]}$
be the modules obtained from $F/[F,F]$ and $A/[A,A]$,
respectively, by adjoining $[\frac{1}{p}
(r')^{p^{\ell}}]$ for all $p$ prime and $\ell \geq 1$.

Henceforth, a ``monomial'' in $X$ refers to a noncommutative monomial,
i.e., an element of the form $x_1 \cdots x_k$ for $x_1,\ldots,x_k \in
X$ (with any of the $x_i$ allowed to be equal to each other), unless
we say ``cyclic monomial.'' A cyclic monomial $[x_1 \cdots x_k]$ is the same except where $[x_1 \cdots x_k] = [x_2 \cdots x_k x_1]$ by definition.
\begin{lemma}\label{mcl}
 \begin{enumerate}
\item \label{bapp} A basis for $A$ is given by monomials in $x, y,$ and $r'$
      such that the maximal submonomials in $x,y$ are of the type
      \begin{equation} \label{zabdefn}
      z_{a,b} := \begin{cases} (x y)^b x^{a-b}, & \text{if } a \geq b, \\
                         (y x)^a y^{b-a}, & \text{if } a < b, 
      \end{cases}
      \end{equation}
      i.e.~of the form $z_{a_1, b_1} r' z_{a_2, b_2} \cdots r' z_{a_{m+1},
      b_{m+1}}$ for $m, a_i, b_i \geq 0$;
\item \label{ampc} A basis for $A/[Ar'A,A]$ is given by
\begin{enumerate}
\item $z_{a,b}$; and
\item classes of the form $[z_{a_1, b_1} r' z_{a_2, b_2} \cdots z_{a_m, b_m}
      r']$ (up to simultaneously cyclically permuting the $a$ and $b$ indices).
\end{enumerate}
\item \label{bacpp} The canonical quotient map
 $V = A/[Ar'A,A] \onto A/[A,A]$
has kernel the free submodule
  $W \subset V$ with
  basis given by the classes (for $a > b \geq 1$):
\begin{gather} \label{wabdefn1}
W_{a,b} := \sum_{a_\cdot,b_\cdot} \frac{\gcd(a,b)}{\rep(a_\cdot,b_\cdot)}
[\prod_{\ell=1}^{|\{a_{\cdot}\}|} (-r')
z_{a_\ell,b_\ell}], \\
W_{b,a} := \sum_{a_\cdot,b_\cdot} \frac{\gcd(a,b)}{\rep(a_\cdot,b_\cdot)} \label{wabdefn2}
[\prod_{\ell=1}^{|\{a_\cdot\}|} r'
z_{b_\ell,a_\ell}], \\ \label{wabdefn3}
W_{a,a} = [(xy+r')^a] - [(xy)^a],
\end{gather}
where the sums $\sum_{a.,b.}$ are over all \textbf{distinct} pairs
$(a.,b.)$ of tuples of the same length, modulo simultaneous cyclic
permutations of the indices, such that $W_{c,d}$ always has bidegree
$(c,d)$ in $x,y$ (with $r'$ having bidegree $(1,1)$, so in
the summation $a = \sum
(a_i+1)$ and similarly for $b$) and such that $a_\ell > b_\ell$ for
all $\ell$.
The coefficient $\rep(i_\cdot,j_\cdot)$ is as defined in \eqref{d:rep}
above, viewing $(i.,j.)$ as a $k$-tuple of elements of $\N \times \N$.
Finally, the product is taken in left-to-right order, i.e., in
\eqref{wabdefn1} it expands as $[(-r')z_{a_1,b_1} (-r')z_{a_2,b_2}
\cdots (-r') z_{a_\ell,b_\ell}]$, and similarly for
\eqref{wabdefn2}. Note that the coefficients of $W_{a,b}, W_{b,a}$,
and $W_{a,a}$ above are integers.

\item \label{bactpp} None of the classes $W_{a,b}$, considered as
  elements of $V$, are multiples of any non-unit in $\Z$.  Under the
  quotient $V \onto V/\langle [(r')^m] \rangle_{m \geq 1}$, the images
  of only the classes $W_{p^{\ell},p^{\ell}}$ are multiples of a
  non-unit in $\Z$, the greatest of which is $p$.

       Then, the classes $\frac{1}{p} W_{p^\ell,p^\ell} \in \hat V$
      and $\frac{1}{p} [r^{p^\ell}] \in \widehat{F/[F,F]}$ have
      the same image in $\widehat{A/[A,A]}$, are nonzero, and
      generate the torsion of $\widehat{A/[A,A]}$ ($\Z/p$ in degrees
      $2 p^\ell$ and $0$ otherwise).
\end{enumerate}
\end{lemma}
We note that abstractly understanding $A/[A,A]$ is easy if we wanted
to use cyclic words in $x$ and $y$, but the point of finding bases of
the above form is to allow one to obtain bases in a further quotient
by a power of $r'$, or in an extension of $A$ (cf.~Corollary
\ref{mclcor}). In particular, to prove Conjecture \ref{rh}, we will
view $F$ as a subalgebra of the free algebra $P=\Z\langle x_1,
\ldots, x_n, y_1, \ldots, y_n \rangle$ by $x=x_1, y=y_1,$ and $r' =
\sum_{i=2}^n [x_i,y_i]$, in which case the image of $(r')^{\ell}$
in $\Lambda$ is zero for all
$\ell \geq 1$.
\begin{proof}
  \eqref{bapp} 
%
We use the Diamond Lemma for modules as formulated in Appendix \ref{dla}.
First, define the disorder $\Dis(M)$ of a
  monomial $M$ in $x,y$ to be the minimal number of swaps of adjacent
  letters in $M$ needed to bring it to the form $z_{a,b}$. For a
  monomial $M = M_1 r' M_2 \cdots r' M_{n+1}$ for $n \geq 1$, define the disorder to be
  $\Dis(M) := \Dis(M_1) + \ldots + \Dis(M_{n+1})$: that is, $\Dis(M)$
  is given by the sum of the above disorder over each maximal monomial
  in $x$'s and $y$'s.

  Let $O(M)$ be the maximal nonnegative integer such that $M \in
  (r')^{O(M)}$.  Then, we define the partial order on monomials 
  such that $M_1 \prec M_2$ if and only if either $O(M_1) > O(M_2)$ or
  $O(M_1) = O(M_2)$ and $\Dis(M_1) < \Dis(M_2)$.  Every relation $f
  (xy-yx+r') g$, where $f$ and $g$ are monomials in $x$, $y$, and
  $r'$, then has leading term equal to either $f xy g$ or $f yx g$
  (because of the $O$ condition), and thus can be viewed as a
  reduction $f yx g \mapsto f xy g + f r' g$ or $f xy g \mapsto f yx g
  - f r' g$ which reduces the disorder of the leading term.

  To prove that the reductions are confluent, we have to show that, if
  a monomial $f$ in $x, y$, and $r'$ can be reduced in two different
  ways, then their difference is integrally spanned by relations
  with leading monomials strictly less than $f$ in the partial
  ordering.  That is, we need to consider $f$ in which there are two
  distinct instances where $x$ and $y$ are adjacent and swapping the
  pair of $x$ and $y$ results in a smaller monomial in the partial
  ordering.

  We claim that these two adjacent instances of $x$ and $y$ do not
  overlap.  If not, then there must be either an instance of $xyx$ in
  $f$ such that the disorder decreases if $xyx$ is replaced by $yxx$
  as well as if $xyx$ is replaced by $xxy$, or alternatively an
  instance of $yxy$ in $f$ such that the disorder decreases if $yxy$
  is replaced by $yyx$ as well as if it is replaced by $xyy$.  But
  this is impossible.  Indeed, in the case that the disorder of a
  monomial $f_1 yxx f_2$ is less than that of $f_1 x y x f_2$, then it
  easily follows that $\Dis(f_1 x y x f_2) = \Dis(f_1 y x x f_2) + 1$
  and also $\Dis(f_1 x x y f_2) = \Dis(f_1 x y x f_2)+1$.

So, the two adjacent instances of $x$ and $y$ in $f$ in question do
not overlap.  Call the two reductions of $f$ obtained by the two
corresponding swaps $f'$ and $f''$.  Then $f'$ and $f''$ both admit a
common reduction, obtained by swapping the pair of $x$ and $y$ which
was not yet swapped (i.e., reducing $f'$ by swapping the pair of $x$
and $y$ which was swapped to obtain $f''$, and vice-versa). Since all
monomials appearing in $f'$ and $f''$ are strictly less than $f$, this
implies that $f'-f''$ is a linear combination of relations (defining
reductions) whose leading monomials are strictly less than $f$, which
yields the confluence property in this case.

With either approach, all reductions yield relations with leading
coefficient $1$, so by Proposition \ref{dl3}, the quotient $F/R$ is a
free bigraded $\Z$-module with the given basis (the elements not
appearing as leading coefficients of relations).

\eqref{ampc} Let $F' \subset F$ be integrally spanned by monomials not
containing $r'$ and let $V' := F' \oplus Fr'F/[Fr'F,F]$.  There is an
obvious surjection $\beta: V' \onto V := A/[Ar'A,A]$.  $V'$ has a basis
consisting of monomials in $x$ and $y$ and cyclic monomials in $x$,
$y$, and $r'$, i.e., monomials in $x$, $y$, and $r'$ up to cyclic
permutation.  On cyclic monomials, let $O(f)$ denote the number of
times that $r'$ occurs in $f$. Let us define a partial order on
monomials in $V'$: $f \prec g$ if either (1) $O(f) > O(g)$; or (2)
$O(f)=O(g)$ and there are fewer swaps $yx \leftrightarrow xy$ needed
to bring $f$ to normal form than for $g$.

Note that, in $V'$, monomials in $x, y$, and $r'$ are invariant under
cyclic permutations, which is why we place them in square braces (in
accordance with \S \ref{ns}).

 Then, it is not difficult to see that the set of elements
 $f(xy-yx)g + [fr'g], [h(xy-yx+r')] \in V'$ (for $f, g, h$ monomials,
 with $h \in \ldp r' \rdp$ and $f,g \notin \ldp r' \rdp$) form a
 confluent set $W'$, so that $V'/W' \cong A/[Ar'A,A]=V$, as desired: the
 normal-form basis of the quotient is given by the classes listed in
 the statement of \eqref{ampc}.  One makes the same arguments as
 before: the main point is that cyclic monomials that contain $r'$
 behave similarly to regular monomials in that one can compute the
 minimal number of swaps $xy \leftrightarrow yx$ needed to reduce to a
 normal-form element, and any two reductions can be performed in
 either order with the same result.

 In detail, to prove confluence, as before, it suffices to show that,
 if $f$ is either a monomial in $x$ and $y$, or a cyclic monomial in
 $x$, $y$, and $r'$, and if there are two adjacent instances of $x$
 and $y$ in $f$ such that swapping either pair decreases the disorder
 of $f$, then the corresponding reductions $f'$ and $f''$ admit a
 common reduction.  In the case of monomials in $x$ and $y$, this is
 exactly the same argument as before.  In the case of cyclic monomials
 in $x$, $y$, and $r'$, then as before, in this case the two adjacent
 instances of $x$ and $y$ must be disjoint, since $\Dis([f_1 yxx f_2])
 < \Dis([f_1 xyx f_2])$ (for $f_1$ and $f_2$ monomials in $x, y$, and
 $r'$ at least one of which contains $r'$) implies $\Dis([f_1 xyx
 f_2]) = \Dis([f_1 yxx f_2]) + 1$ and $\Dis([f_1 xxy f_2]) = \Dis([f_1
 xyx f_2]) + 1$.  Therefore $f'$ and $f''$ admit common reductions,
 obtained by swapping the other instance of $x$ and $y$ not yet
 swapped.  Thus $f'-f''$ is a linear combination of relations with
 smaller leading (cyclic) monomial than that of $f$, proving the
 confluence property.

 Again, the leading coefficients of the relations generating $W'$
 (corresponding to the reductions) all are one, so by Proposition
 \ref{dl3}, the desired set (which consist of the generators of $V'$
 which are not leading coefficients of any relations) forms a basis of
 $V$.

 \eqref{bacpp} We claim that
 $[A,A] = [Ar'A,A] + [\langle z_{a,b-1}\rangle_{a,b \geq 1},y] + [\langle z_{a-1,b}
 \rangle_{a,b \geq 1}, x]$.
 This follows immediately from the fact that the $z_{a,b}$ form a
basis of $A$ (Part \eqref{bapp}).  This means that $[A,A] = V/W$ where $W$ is 
integrally spanned by
 the relations
\begin{gather}
w_{a,b,y} := \eta([z_{a,b-1}, y]), \quad w_{a,b,x} := \eta([z_{a-1,b}, x]
 \in V,
\intertext{where}
 \eta: A \onto A/[Ar'A,A] \cong V \text{ is the quotient.}
\end{gather} 
It remains to show that
\begin{equation}
\langle w_{a,b,x}, w_{a,b,y} \rangle  = \langle W_{a,b}
\rangle.
\end{equation} 
This will complete the proof of \eqref{bacpp}.

We will prove the sharper result that
\begin{equation} \label{xycomms}
\frac{\gcd(a,b)}{b} w_{a,b,x} = -\frac{\gcd(a,b)}{a}
w_{a,b,y} = \pm W_{a,b}, \quad a,b\geq 1,
\end{equation}
using the positive choice of gcd,
where $\pm$ is plus if $a \leq b$ and minus if $a > b$. Here the division by $a$ or $b$ makes sense since $V$ is a free $\Z$-module.

First, we note that $a w_{a,b,x} + b w_{a,b,y} = 0$, since the left-hand
side is equivalent in $A/[Ar'A, A]$ to
\begin{equation}
\sum_{i=1}^{a+b} \eta([h_{i+1} h_{i+2} \cdots h_{a+b} h_1 h_2 \cdots h_{i-1}, h_i]), \quad h_1 \cdots h_{a+b} = z_{a,b}, \quad h_i \in \{x,y\}, \forall i.
\end{equation}
Thus, it remains to compute only one of $w_{a,b,x}$ and $w_{a,b,y}$
for all positive integers $a$ and $b$.  Let us suppose that $a > b$.
Then, it follows that $w_{a,b,x}$ is the result of successively
commuting the $x$ on the right of $(xy)^b x^{a-b}$ all the way to the
left, and subtracting the resulting $x(xy)^b x^{a-b-1}$.  So,
$w_{a,b,x} = \sum_{0 \leq b' \leq b-1} [(xy)^{b'} x r' (xy)^{b-b'-1}
x^{a-b-1}]$.  Each summand can then be reduced to a linear combination
of monomials in $r'$ and the monomials $z_{c,d}$ for all $c, d$.
Specifically, this is the sum of all possible terms of the form $[r'
z_{a_1,b_1} r' \cdots r' z_{a_k, b_k}]$ such that $b_1 \geq b - b'-1$,
and satisfying the conditions of \eqref{wabdefn1}: $a_\ell > b_\ell$
for all $\ell$, and $a_1 + \cdots + a_k + k = a$, $b_1 + \cdots + b_k
+k = b$.  So, when we add up all the contributions to $w_{a,b,x}$, we
get $-((b_1+1) + \cdots + (b_m+1) = b) / rep(a_\cdot, b_\cdot)$ copies
of each $[(-r') z_{a_1,b_1} (-r') \cdots (-r') z_{a_m, b_m}]$, i.e.,
$-\frac{b}{\gcd(a,b)} W_{a,b}$.

For the same reason, when $b > a$ we get
$w_{a,b,y} =-\frac{a}{\gcd(a,b)} W_{a,b}$.  So it remains to consider
$w_{b,b,x}$.  Here, we get $\eta((yx)^b-(xy)^b) =
[(xy+r')^b]-[(xy)^b]$, as desired.

It is clear that the classes $W_{a,b}$ are zero if either of $a$ or
$b$ is zero.  This completes the proof of part \eqref{bacpp}.

\eqref{bactpp} Since $V$ is a free $\Z$-module, and each relation
\eqref{wabdefn1}--\eqref{wabdefn3} lives in a different bigraded
degree, we find that the torsion part of $V/W$ is a direct sum of its
bigraded components, which are all cyclic (or trivial).  In bidegree
$(a,b)$, this module is $\Z/g_{a,b}$, where $g_{a,b}$ is the gcd of
all the coefficients $\ds \frac{\gcd(a,b)}{\rep(a_\cdot, b_\cdot)}$
that appear in \eqref{wabdefn1}--\eqref{wabdefn3}.  To compute this,
we claim first that the numbers $\rep(a_\cdot, b_\cdot)$ range over
all positive factors of $\gcd(a,b)$: it's clear that any cyclic
monomial of the form $[f^\ell]$ with bidegree $(a,b)$ must have $\ell$
be a factor of $a$ and $b$; on the other hand, for any such factor, we
can form the cyclic monomial $[f^\ell]$ where $f = r'
x^{\frac{a}{\ell} - 1} y^{\frac{b}{\ell} - 1}$.  So, the gcd of the
coefficients $\ds \frac{\gcd(a,b)}{\rep(a_\cdot, b_\cdot)}$ is $1$ for
all $a,b$.  When $a \neq b$, the same is true if we restrict to terms
not of the form $[(r')^m]$, since such terms cannot have bidegree
$(a,b)$.  On the other hand, in the case $a=b$, we see that,
if we restrict to the terms other than $[(r')^m]$, then the gcd is
equal to $p$ in the case $m=p^\ell$ where $p$ is prime (but still one
in any other case). To see this, first note that the term $[(r')^m]$
is the unique one that can have coefficient equal to $1 =
\frac{\gcd(a,a)}{a}$, since $\rep(a_\cdot, b_\cdot)=a=b$ if and only
if $a_\cdot=b_\cdot=(0,0,\ldots,0)$ (tuples of length $a=b$ with all
zeros).  Next, the gcd of all positive factors of $m$ other than one
equals $p$ when $m=p^\ell$ and equals one otherwise.  
This proves the claims of the first paragraph.

Finally, we need to show that the classes
$\frac{1}{p} W_{p^\ell,p^\ell} \in \hat V$ and
$\frac{1}{p} [r^{p^\ell}] \in \widehat{F/[F,F]}$ have the same
image in $\widehat{A/[A,A]}$, are nonzero, and generate its torsion.
First, we have already seen that the images of the classes
$\frac{1}{p} W_{p^\ell,p^\ell} \in \hat V$ are nonzero and 
generate the torsion of
$\widehat{A/[A,A]}$, since $\widehat{A/[A,A]}$ is obtained by modding
by the classes $W_{a,b}$, which now have greatest integer factor
$= p$ if $a=b=p^\ell$ and greatest integer factor $=1$ otherwise.

We can write the image of the class $\frac{1}{p} [r^{p^\ell}]$ in $A/[A,A]$ as follows:
\begin{equation}
\frac{1}{p} [r^{p^\ell} - (xy)^{p^\ell} - (-yx)^{p^\ell}] + 
\frac{1}{p} (1 + (-1)^{p^\ell})[(xy)^{p^\ell}],
\end{equation}
and since the first term in square braces is actually $p$ times an integral
combination of cyclic words, we can replace terms $(yx)$ by $(xy+r')$ (without destroying the ability to divide by $p$), and obtain
\begin{equation}
\frac{1}{p} [-(xy)^{p^\ell} - (-xy-r')^{p^\ell}] + \frac{1}{p} (1 + (-1)^{p^\ell})[(xy)^{p^\ell}] \equiv \frac{1}{p}\bigl([(xy+r')^{p^\ell}] - [(xy)^{p^\ell}]\bigr) \pmod p,
\end{equation}
which is the image of $\frac{1}{p} W_{p^\ell,p^\ell}$ under
$\widehat{F/[F,F]} \onto \widehat{A/[A,A]}$, as desired.
\end{proof}
From the proof, we deduce the more general
\begin{cor} \label{mclcor} Let $B = B_+ \oplus \k$ be any (graded)
  augmented $\k$-algebra such that $B_+$ is a free $\k$-module.
Let $r' \in
  B\setminus\k$ (of degree 2). Let $r := xy-yx + r'$, and set $F' =
  \Z \langle x,y\rangle * B$ (with $\Z \langle x,y \rangle$ the free
  algebra) and $A' = F' / \ldp r \rdp$.  Then
\begin{equation}
A'_{\cyc} \liso (\Z[x,y] * B)_{\cyc} / W,
\end{equation}
where $W$ is as defined in Lemma \ref{mcl}.  In particular,
\begin{equation}
  \text{torsion}(A'_{\cyc}) \liso \text{torsion}(B_{\cyc}) \oplus \bigoplus_{p,\ell}
\langle r^{(p^\ell)} \rangle,
\end{equation}
where the sum is over all primes $p$ and all $\ell \geq 1$ such that
$(r')^{p^\ell} \in pB + [B,B]$, and $\langle r^{(p^\ell)} \rangle$
denotes the subgroup of $A'_{\cyc}$ generated by $r^{(p^\ell)}$.  The isomorphisms
are obtained from $F' \onto A'$ by picking a $\Z$-module section
$\Z[x,y] \into F'$ of the submodule $\Z[x,y] \subseteq A'$.  The
classes $r^{(p^\ell)}$ are all nonzero.
\end{cor}
In the corollary, the class $r^{(p^\ell)}$ is defined as follows:
Let $f(x,y,r')$ be a noncommutative polynomial in $x,y,r'$ such that
$f(x,y,r')+(r')^p \equiv (xy-yx+r')^p$ modulo commutators, and such
that all the coefficients of $f(x,y,r')$ are divisible by $p$ (such an
$f$ exists by the construction of $r^{(p^\ell)} \in \Lambda_Q$).  Let
$g \in B$ be an arbitrary element such that $p g - (r')^p \in 
[B,B]$.  Then we can define $r^{(p^\ell)}$ as the image of
$\frac{1}{p} f(x,y,r') + g$ in $A'_{\cyc}$.
\begin{proof}
  We apply the argument of the lemma. Namely, we apply the Diamond
  Lemma for free modules (as found in Appendix \ref{dla}) now to
  $F'/[F',F'B_+F']$ (note that our $F'$ now replaces what was $F$
  before, and there is no notation for what was $F'$ in the proof of
  the lemma).  Since $B_+$ is free as a $\k$-module, we can do this.
  Our partial ordering only concerns, as before, the $x$ and $y$
  variables.  Namely, we can write $F'/[F',F'B_+F']$ as the direct sum
  of $(B_+)_{\cyc}$ and a free $\k$-module whose basis consists of
  noncommutative words in $x$ and $y$ together with cyclic
  noncommutative words in $x$, $y$, and a basis for $B_+$, with no two
  elements of $B_+$ appearing next to each other (a cyclic
  noncommutative word is, as before, a noncommutative word modulo
  cyclic permutations).  Our partial ordering is based as before on
  the degree in $x$ and $y$ and on the disorder function, giving the
  number of swaps of $x$ and $y$ needed to yield an element which is a
  word in the $z_{a,b}$ and the basis of $B$.  Then, the rest of the
  proof goes through in the same manner.
\end{proof}
\begin{rem}
  The corollary immediately implies Conjecture \ref{rh}, setting $B :=
  \Z\langle x_2, y_2, \ldots, x_g, y_g \rangle$ to be the free
  algebra, and defining $r' = \sum_{j=2}^g [x_j, y_j]$.  (In fact, it
  implies Theorem \ref{mt} for arbitrary quivers $Q$ containing a
  loop, i.e., $Q \supsetneq Q^0 \cong \tilde A_0$. See Remark
  \ref{mclmtr}.)
\end{rem}

\begin{rem}\label{xyrem} The proof of parts (1)--(2) of Lemma \ref{mcl}
 and some of (3) is slightly shorter if we used 
  $x^a y^b$ instead of $z_{a,b}$; in this case we have
\begin{equation}
W_{a,b} = \sum_{a_\cdot, b_\cdot} \frac{\gcd(a,b)}{\rep(a_\cdot,
    b_\cdot)}
  [\prod_{\ell=1}^k (r' x^{a_\ell} y^{b_\ell})],
\end{equation}
and the Diamond Lemma argument is a bit simpler.  However, the
disadvantage is that the given formula for $W_{a,a}$ 
does not follow explicitly from the computation; more importantly,
 the chosen
convention resembles more closely the $Q \supsetneq \tilde D_n$ case
(however, the elements $x^a y^b$ do have an analogue for the $Q
\supsetneq \tilde A_{n-1}$ case).

In more detail,  replacing $z_{a,b}$ by $x^ay^b$,
we could avoid using our version of the Diamond Lemma altogether for
part (1) above, using instead the Gr\"obner generating set
$(xy-yx+r')$ with respect to the graded lexicographical ordering with
$|r'|=|x|=|y|=1$ and $r'<x<y$ (see Appendix \ref{gbs} and Proposition
\ref{gbp} therein).  To prove part (3) from parts (1) and (2),
 we could use that the normal form of $z_{a,b}$ with respect to the
  basis of monomials in $x$ and $y$
 has leading term $x^a y^b$.
\end{rem}

\section{Proof of Theorem \ref{mt} for good primes: a $\Gamma$-equivariant version} \label{gpsec}
In this section, we will prove a ``$\Gamma$-equivariant'' version of
Lemma \ref{mcl}, which will allow us to prove the main Theorem
\ref{mt} for good primes.  

\subsection{Partial preprojective algebras}\label{ss:ppa}
We will need the notion of \emph{partial preprojective algebra},
denoted $\Pi_{Q,J}$, which depends on a subset of vertices $J
\subseteq Q_0$, and is defined as a quotient of $P_\dq$ by the
relations $iri$ only at vertices $i \in Q_0 \setminus J$,
called \emph{black} vertices.  In particular, this includes
$P_\dq=\Pi_{Q,Q_0}$ as well as $\Pi_Q = \Pi_{Q,\emptyset}$. We define
this in detail below.

We remark that there is an analogy between $\Pi_Q$ and closed Riemann
surfaces, in which $\Pi_{Q,J}$ is obtained by adding
\emph{punctures} at the points corresponding to $J$: for more details,
see \S \ref{ssqsec}.

\begin{ntn} \label{1i0ntn}
If $Q_0' \subset Q_0$, and $\k$ is any commutative ring,
then let $\1_{Q_0'} \in \k^{Q_0}$ denote the matrix which is
$1$ in entries $(i,i), \forall i \in Q_0'$, and zero elsewhere.
In terms of vertex idempotents,  $\1_{Q_0'} = \sum_{i \in Q_0'} i$.
\end{ntn}
\begin{defn} \cite{EE} Given any quiver $Q$, and any subset $J
  \subset Q_0$ of vertices (called the \textbf{white} vertices, with
  $Q_0 \setminus J$ the set of \textbf{black} vertices),
  define the \textbf{partial preprojective algebra} $\Pi_{Q,J}$
  by:
\begin{equation}
\Pi_{Q,J} := P_\dq / \ldp \1_{Q_0 \setminus J} r \1_{Q_0 \setminus J} \rdp.
\end{equation}
\end{defn}
\begin{defn} $\Lambda_{Q, J} := (\Pi_{Q,J})_{\cyc}$.
\end{defn}

\subsection{Proof of Theorem \ref{mt} for good primes}
Let $Q^0$ be an extended Dynkin quiver, and $\Gamma \subset SL_2(\C)$
the corresponding finite subgroup under the McKay correspondence. 
 Identify $x$ and $y$ with the standard basis of $\C^2$, and hence
let $SL_2(\C)$ and therefore $\Gamma$ act on $\C \cdot \{x,y\}$.
By \cite{CBH}, \S 3, for $\k = \C$, we know that there
are Morita equivalences
\begin{gather} \label{morp}
P_{\dzq} \o \k \simeq \k\langle x,y \rangle \rtimes \Gamma, \\
\Pi_{Q^0} \o \k \simeq \k[x,y] \rtimes \Gamma, \label{morpi}
\end{gather}
where, as before, $\k \langle x,y \rangle$ is the free algebra,
and for any algebra $A$ with an action by a group $\Gamma$, 
$A \rtimes \Gamma$ denotes the skew group algebra. 
These Morita equivalences are given by
\begin{equation} \label{piq0m}
P_{\dzq} \o \k \iso \fsf (\k\langle x,y \rangle \rtimes \Gamma) \fsf
\end{equation}
where $\fsf \in \k[\Gamma]$ is a sum of primitive idempotents, one for
each irreducible representation of $\Gamma$.  In more detail, $\fsf =
\sum_{i \in Q^0_0} \fsf_i$, where the vertices $Q^0_0$ of the extended
Dynkin quiver $Q^0$ also label the irreducible representations $U_i$
of $\Gamma$, and $\fsf_i$ is chosen so that $\k[\Gamma] \fsf_i \cong
U_i$.  Then, one may choose elements $\tilde a \in \k\langle  x,y
\rangle \rtimes \Gamma$ for each arrow $a \in \dq_1$ so that the map $a
\mapsto \fsf \tilde a \fsf, i \mapsto \fsf_i$ gives the isomorphism
\eqref{piq0m}, and furthermore that the element $xy - yx \in \k
\langle x,y \rangle \rtimes \Gamma$ maps to the element $r \in
P_{\dzq}$, yielding \eqref{morpi} (specifically, \eqref{piq0m}
descends to $\Pi_{Q^0} \o \k \iso \fsf (\k[x,y] \rtimes \Gamma)
\fsf$).

Then, the idea of the proof of Theorem \ref{mt} for good primes is to
repeat the arguments of the previous section, generalizing to a
``$\Gamma$-equivariant'' version.  This involves replacing $A =
\k\langle x,y,r' \rangle/\ldp xy-yx+r' \rdp \cong \k \langle x,y
\rangle$ by $A = (\k \langle x,y,r' \rangle/\ldp xy - yx + r' \rdp)
\rtimes \Gamma \cong \k \langle x,y \rangle \rtimes \Gamma$.
Here, we
set the $\Gamma$-action on $r'$ to be trivial.

There are a few problems with this. First, in the previous section, we
worked over $\Z$; we cannot get any torsion information setting $\k =
\C$ as above.  We resolve this problem by restricting to ``good
characteristic'': primes that do not divide $|\Gamma|$, where all of
the above easily generalizes (as is well-known).  As a slight
modification, rather than working over an algebraically closed field
in good characteristic, we will set $\k :=
\Z[\frac{1}{|\Gamma|},e^{\frac{2\pi i}{|\Gamma|}}]$, since this allows
us to see all of the torsion (except in bad primes)
simultaneously. Moreover, for this choice of $\k$, up to conjugation,
we can assume that $\Gamma < SL_2(\k)$.  (Note that
$\Z[\frac{1}{|\Gamma|},e^{\frac{2\pi i}{|\Gamma|}}]$ is not a
principal ideal domain, but this will not cause us any difficulties,
as we did not and will not need to assume that we obtain a direct sum
of cyclic modules; in any case, the torsion information we obtain
could equivalently be obtained by working over finite fields of all
good characteristics, and the modules we are actually interested in are defined
over $\Z$.)

Second, we will need a way to get from $A$ above to actual
preprojective algebras of quivers properly containing $Q^0$. This will
follow from a generalization of Corollary \ref{mclcor} and some
general arguments about (partial) preprojective algebras, further
developing some of the ideas of \cite{EE, EG}.

Third, to understand (a presentation of) $A/[A,A]$, we first need to
understand $F/[F,F]$ where $F = \k\langle x,y,r'\rangle \rtimes
\Gamma$.  For this, we use the general (known) Hochschild theory for
skew group algebras, as follows:\footnote{Thanks to P. Etingof for
  explaining this (known) proposition and proof.}
\begin{prop} \label{hcsgap} Let $A$ be an associative algebra over a
  commutative ring $\k$ and $\Gamma$ a finite group acting on
  $A$. Assume that $\k$ contains $\frac{1}{|\Gamma|}$ (in particular,
  the characteristic of $\k$ does not divide $\Gamma$). For any
  $\gamma \in \Gamma$, let $\Gamma_\gamma := \{\gamma' \mid \gamma'
  \gamma = \gamma \gamma'\}$ denote the centralizer of $\gamma$. Then,
  for any $A \rtimes \Gamma$-bimodule $M$, one has
\begin{gather}
HH^\bullet(A \rtimes \Gamma, M) \cong HH^\bullet(A,M)^\Gamma, \quad
HH^\bullet(A \rtimes \Gamma) \cong \bigoplus_{\gamma}
HH^\bullet(A,A\gamma)^{\Gamma_\gamma}; \\ \label{sgahom}
HH_\bullet(A \rtimes \Gamma, M) \cong HH_\bullet(A, M)_\Gamma, \quad
HH_\bullet(A \rtimes \Gamma) \cong \bigoplus_\gamma
HH_\bullet(A, A\gamma)_{\Gamma_\gamma},
\end{gather}
where we sum over a collection of representatives of the conjugacy
classes of $\Gamma$ (one for each conjugacy class).
\end{prop}
\begin{proof}
We prove the second result (for Hochschild homology) since we will use that
one more heavily; the first follows from the co-version of the proof.  

We write
\begin{equation}
  HH_\bullet(A \rtimes \Gamma, M) = \Tor^{(A \o A^{\op}) \rtimes 
(\Gamma \times \Gamma^{\op})}_\bullet (A \rtimes \Gamma, M).
\end{equation}

For any algebra $B$, let $LH^B$ denote the left-derived functor of $B-Bimod \rightarrow \k-mod$, given by 
$M \mapsto M/\langle mb - bm \rangle_{m \in M, b \in B}$: that is, the derived functor which yields the Hochschild homology of $B$ with coefficients in the
given $B$-bimodule $M$.
Let $\Gamma_\Delta := \{(g,g^{-1}) \in \Gamma \times \Gamma^{\op}\}$.
By Shapiro's lemma, since $\Ind_{\Gamma_\Delta}^{\Gamma \times \Gamma^{\op}}(A) = A \o \k[\Gamma] \cong A \rtimes \Gamma$ as a $\Gamma \times \Gamma^{\op}$-module, where
$A$ (and $M$) are $\Gamma_\Delta \cong \Gamma$-modules by conjugation,
\begin{multline}
\Tor^{(A \o A^{\op}) \rtimes 
(\Gamma \times \Gamma^{\op})}_\bullet (A \rtimes \Gamma, M) =
H_\bullet(LH^{A \o A^{\op}}
(LH^{\k[\Gamma] \o \k[\Gamma]^{\op}}((A \rtimes \Gamma) \o M))) \\\cong
H_\bullet(LH^{A \o A^{\op}}(LH^{\k[\Gamma_\Delta]}(A \o M))) =
H_\bullet(LH^{\k[\Gamma_\Delta]}(LH^{A \o A^{\op}}(A \o M))).
\end{multline}
Now, by our assumption on $\k$, taking $\Gamma$-coinvariants
(or invariants) is exact, so the RHS is 
\begin{equation}
\Tor_\bullet^{A \o A^{\op}}(A \o M)_{\Gamma} = HH_\bullet(A,M)_{\Gamma}.
\end{equation}
Finally, specializing to $M := A \rtimes \Gamma$, we note that
\begin{equation}
\k[\Gamma] = \bigoplus_{C} k[\Gamma,C], \quad
\k[\Gamma,C] := \bigoplus_{\gamma \in C}\langle \gamma \rangle,
\end{equation}
where $C$ ranges over the conjugacy classes of $\Gamma$.
Since $\k[\Gamma,C]$ is stable under the $\k[\Gamma]$-action (given by
the conjugation action of $\Gamma$), we end up with the
second formula in \eqref{sgahom}.
\end{proof}

Let us introduce the notation (for $a,b \in A$ and $\gamma \in \Gamma$)
\begin{equation}
[a,b]_\gamma := a(\gamma \cdot b) - ba,
\end{equation}
where $\cdot$ denotes the action of $\Gamma$ on $A$ (so
$\gamma \cdot b = \gamma b \gamma^{-1} \in A \rtimes \Gamma$). Then,
in the case of degree zero, we may rewrite \eqref{sgahom} as
\begin{equation} \label{hh0sga}
(A \rtimes \Gamma)_{\cyc} \cong \bigoplus_{C} 
(A/[A,A]_{\gamma_C})_{\Gamma_{\gamma_C}},
\end{equation}
where $C$ ranges over the conjugacy classes of $\Gamma$, 
$\gamma_C \in C$ is a fixed choice of representative for each $C$, and
$\Gamma_{\gamma} \subset \Gamma$ denotes the centralizer of $\gamma$ in
$\Gamma$.
Note that this formula
may also be obtained directly without using any homological algebra,
but only the definition $B_{\cyc} := B/[B,B] = HH_0(B)$ and the decomposition of
$\k[\Gamma]$ into a direct sum of conjugacy classes. However, we will
use the above formulas later on, and felt it is better to explain the
general result.

We now state our
\begin{thm}\label{mclgam} 
  Let $\k := \Z[\frac{1}{|\Gamma|},e^{\frac{2\pi i}{|\Gamma|}}]$ and
  $A = \k\langle x,y,r' \rangle / \ldp xy-yx+r' \rdp$, and let $\Gamma
  \subset SL_2(\k)$ be a finite subgroup, which acts on $x$ and $y$ by
  the tautological action on $\k \cdot \{x,y\} \cong \k^2$ (the free
  $\k$-module generated by $x$ and $y$), and fixes $r'$.

Then, the $\k$-module $(A \rtimes \Gamma)_{\cyc}$ has a canonical decomposition along
conjugacy classes $C$ of $\Gamma$,
\begin{gather}
(A \rtimes \Gamma)_{\cyc} = \bigoplus_{C} (A \rtimes \Gamma)_{\cyc,C},
\end{gather}
presented as follows:
\begin{enumerate}
\item[(i)] For $C \neq \{1\}$ and any choice $\gamma_C \in C$,
$(A \rtimes \Gamma)_{\cyc,C}$ is torsion-free (and in fact $\k$-projective).

\item[(ii)] For $C = \{1\}$, there is an isomorphism
 $V_1/W_1 \iso
(A \rtimes \Gamma)_{\cyc,\{1\}}$, where $V_1$ is the direct sum of
$\k[x,y]^\Gamma$ and the $\k$-module of $\Gamma$-invariant classes of
elements of the form 
\begin{equation} \label{cycpolys}
  \sum_{a_\cdot,b_\cdot} \alpha_{a_\cdot,b_\cdot}  z_{a_1, b_1} r' \cdots z_{a_m,b_m} r',
\end{equation}
modulo cyclic permutations, with $\alpha_{a\cdot,b_\cdot} \in \k$, and
$W_1$ is $\k$-linearly spanned by the $\Gamma$-invariant classes in
$W$ of Lemma \ref{mcl}, part (3)
(\eqref{wabdefn1}--\eqref{wabdefn3}). This map is induced by the
obvious morphism $V_1 \to (A \rtimes \Gamma)_{\cyc,\{1\}}$.
\item[(iii)] The projection $W_1 \rightarrow [\ldp r' \rdp]^\Gamma / [\ldp r' \rdp^2]^\Gamma
\cong (\langle r' \rangle \o \k[x,y]^\Gamma)$ 
 is a monomorphism and becomes an isomorphism after tensoring by $\C$.
The 
image is 
\begin{equation} \label{immodr2}
\bigl( \bigoplus_{a,b} \langle \gcd(a,b) r' \o x^{a-1} y^{b-1}\rangle\bigr)^\Gamma.
\end{equation}
\item[(iv)] $(A \rtimes \Gamma)_{\cyc,\{1\}}$ is torsion-free (in fact,
  projective over $\k$), but its quotient by the image of $\langle
  [(r')^\ell] \rangle_{\ell \geq 1}$ has $\k/\ldp p \rdp$-torsion in each degree
  $2p^\ell$ for $p$ prime and $p \nmid |\Gamma|$.
  The torsion then appears in $V_1/(W_1 + \langle [(r')^\ell]
  \rangle_{\ell \geq 1}$ and is $\k$-linearly spanned by the image of
  $r^{(p^\ell)} = \frac{1}{p}[(xy-yx+r')^{p^\ell}]$ under the map
  $(\k \langle x,y,r' \rangle)_{\cyc} \onto (A \rtimes
  \Gamma)_{\cyc}/\langle [(r')^\ell] \rangle_{\ell \geq 1}$.
\end{enumerate}
\end{thm}
The proof of the theorem will be given at the end of the section. First, we give consequences and comments.
\begin{rem}
  In fact, one can deduce from the theorem that $(A \rtimes
  \Gamma)_{\cyc}$ itself is not merely $\k$-projective but free over $\k$.
  Namely,
  with $\fsf$ as in \eqref{piq0m}, we have $\fsf (A \rtimes \Gamma)
  \fsf = A' \otimes_{\Z} \k$, where $A := (P_{\dzq} *_{\Z^{Q^0_0}}
  \Z[r'])/ \ldp r_0 + r' \rdp$, $r_0 \in P_{\dzq}$ is the standard
  element $r_0 = \sum_{a \in Q_1^0} [a,a^*]$, the path algebra is
  defined over the integers, and the $\Z^{Q^0_0}$-bimodule structure
  on $\Z[r']$ is defined by the augmentation $\Z^{Q^0_0} \onto \Z$.
  Therefore $(A \rtimes \Gamma)_{\cyc} = (\fsf (A \rtimes \Gamma)
  \fsf)_{\cyc} = A'_{\cyc} \otimes_\Z \k$.  On the other hand, $A'_{\cyc}
  \otimes_\Z \Z[\frac{1}{|\Gamma|}]$ is torsion-free, and hence free
  (since $\Z[\frac{1}{|\Gamma|}]$ is a principal ideal domain).  Thus
  $(A \rtimes \Gamma)_{\cyc}$ itself is free over $\k$.  In the proof
  below, we will give a simple reason why the summands $(A \rtimes
  \Gamma)_{\cyc,C}$ are projective, which does not immediately explain why
  their direct sum, $(A \rtimes \Gamma)_{\cyc}$, is free.
\end{rem}
The analogue of Corollary \ref{mclcor} is then
\begin{cor} \label{mclgamcor}
Let $B$ be any nonnegatively graded $\k \oplus \k[\Gamma]$-algebra, whose
graded components are finitely-generated free $\k$-modules. Let 
$r' \in B\setminus\k$, of degree two, be such that $(0 \oplus \1_{\k[\Gamma]}) r' (0 \oplus \1_{\k[\Gamma]}) = r'$.
Set $F' = (\k \langle x,y\rangle \rtimes \Gamma) *_{\k[\Gamma]} B$ where
$\k[\Gamma]$ acts on $B$ by inclusion into the right summand (so does not contain the unit in $\k \oplus \k[\Gamma]$).
Let $r := xy-yx+r'$, and set
$A' := F' / \ldp r \rdp$. 
Then one has isomorphisms
\begin{equation}
A'_{\cyc} \liso ((\k[x,y] \rtimes \Gamma) *_{\k[\Gamma]} B)_{\cyc} / W_1,
\end{equation}
where $W_1$ is as described in Theorem \ref{mclgam} (in terms of
$x,y,$ and $r'$).  In particular,
\begin{equation}
  \text{torsion}(A'_{\cyc}) \liso \text{torsion}(B_{\cyc}) \oplus \bigoplus_{p,\ell}
\langle r^{(p^\ell)} \rangle,
\end{equation}
where $p$ ranges over all primes and $\ell \geq 1$ ranges over
positive integers such that $(r')^{p^\ell} \in pB + [B,B]$.  The
isomorphisms are obtained from $F' \onto A'$ by picking a section of
$(\k[x,y] \rtimes \Gamma)$ into $F'$. All of the classes
$r^{(p^\ell)}$ are nonzero.
\end{cor}
(We omit the proof, which is similar to that of Corollary
\ref{mclcor}. Note that $r^{(p^\ell)}$ is also defined in exactly the
same way as is done there.) This corollary immediately gives us
Theorem \ref{mt} for good primes:
\begin{proof}[Proof of Theorem \ref{mt} for good primes] 
  Assume $Q \supsetneq Q^0$.  We will prove the theorem for this
  quiver $Q$.

  We will use the fact that $\Pi_{Q \setminus Q^0,Q^0_0}$ and
  $(\Pi_{Q \setminus Q^0, Q^0_0})_{\cyc}$ are torsion-free (see Proposition
  \ref{bpp}, whose proof does not use any other results from this
  paper aside from the Diamond Lemma as formulated in Appendix
  \ref{dla}).  To prove the theorem, we claim that it suffices to work
  over $\k := \Z[\frac{1}{|\Gamma|},e^{\frac{2 \pi i}{|\Gamma|}}]$,
  i.e., to replace $\Pi_Q$ and $\Pi_{Q\setminus Q^0, Q^0_0}$ by $\Pi_Q
  \o \k$ and $\Pi_{Q \setminus Q^0, Q^0_0} \o \k$.  This is because
  $\Lambda_Q$ is a direct sum of cyclic modules, and the number of
  copies of $\Z/p$, for $p \nmid |\Gamma|$, is still detected after
  tensoring by $\k$ (it becomes the number of copies of $\k / \ldp p
  \rdp$), as is the fact that $r^{(p^\ell)}$ is nonzero.  So, it
  suffices to prove that $\Lambda_Q \otimes \k$ has torsion in
  degrees $2 p^\ell$ for $p \nmid |\Gamma|$ and $\ell \geq 1$, and
  that that torsion is a copy of $\k/\ldp p \rdp$ which is
  $\k$-linearly spanned by $r^{(p^\ell)}$.

  For readability, we will omit the $\o \k$ and work over $\k$ in the
  proof. We need to show that the torsion is $\k$-linearly spanned by
  nonzero classes $r^{(p^\ell)}$ for $p \nmid |\Gamma|$.

  Viewing $\k^{Q^0_0}$ as the center of
  $\k[\Gamma]$, set $B = \k[\Gamma] \otimes_{\k^{Q^0_0}} \Pi_{Q
    \setminus Q^0,Q^0_0}$, which is a $\k[\Gamma] \oplus \k^{Q_0
    \setminus Q^0_0}$-algebra.  Using the map $\k \into \k^{Q_0
    \setminus Q^0_0}$ sending $1$ to $\1_{Q_0\setminus Q^0_0}$, we view
  $B$ as a $\k \oplus \k[\Gamma]$-algebra. 

Let the element $r' \in B$ be given by
$r' := \1_{Q^0_0} r_{Q \setminus Q^0} \1_{Q^0_0}$.  Now,
$F' = (\k\langle x,y \rangle \rtimes \k[\Gamma]) *_{\k[\Gamma]} B$.
By definition and the precise version of Morita equivalence from
\cite{CBH} outlined above, we see that $\fsf F' \fsf \iso
\Pi_{Q,Q^0_0}$,
and that $\fsf r \fsf = \fsf (xy-yx+r') \fsf$ maps to the element $r$
of $\Pi_{Q,Q^0_0}$ under the isomorphism.

Then, by the corollary, $\Lambda_Q$, viewing
$\Pi_Q$ now as $\Pi_{Q,Q^0_0} / \ldp r \rdp$,
has torsion $\k$-linearly spanned by the nonzero classes
$r^{(p^\ell)}$ for $\ell \geq 1$ and primes $p \nmid |\Gamma|$.
\end{proof}
\begin{rem} \label{mclmtr} Using the above argument with
  $\Gamma=\{1\}$, we obtain Theorem \ref{mt} in the case of quivers
  containing a loop, $\tilde A_0$, for which all primes are good. This
  only requires Corollary \ref{mclcor}, and not the $\Gamma$-generalization
  (Corollary \ref{mclgamcor}).
\end{rem}

\begin{proof}[Proof of Theorem \ref{mclgam}]
Specializing \eqref{hh0sga} to our case, 
\begin{equation} \label{hh0a1} (A \rtimes \Gamma)_{\cyc} \cong
  \bigoplus_{C} \bigl(A / ([A,Ar'A]_{\gamma_C} + \langle[z_{a,b},
  x]_{\gamma_C},[z_{a,b},y]_{\gamma_C}\rangle_{a,b \geq
    1})\bigr)_{\Gamma_{\gamma_C}}.
\end{equation}

Next, fix $C$, and let $\lambda_{C}, \lambda_{C}^{-1} \in \{e^{2\pi i
  k/|\Gamma|}\}_{1 \leq k \leq |\Gamma|}$ be the eigenvalues of
$\gamma_C$ (which has determinant one). Let us choose an eigenbasis
$x_C,y_C$ of $\gamma_C$ acting on $\k \cdot \{ x,y \} \cong \k^2$,
which we may obtain using the projections $\frac{1}{\lambda_{C}^{\pm
    1} - \lambda_{C}^{\mp 1}}(\gamma_C - \lambda_{C}^{\mp 1})$, unless
$\gamma_C = \pm 1$ (which is equivalent to $\lambda_C=\pm 1$ since in
the latter case $\lambda_C=\lambda_C^{-1}$).  If $\gamma_C=\pm 1$, we
set $x_C = x, y_C = y$.  Let us assume $\gamma_C \cdot x_C =
\lambda_{C} x_C$ and $\gamma_C \cdot y_C = \lambda_C^{-1} y_C$.  Since
$[x_C,y_C]$ must be a unit multiple of $[x,y] = -r'$, we may further
assume that $[x_C,y_C]=[x,y]$, by rescaling.  Then, for $g \in \langle
x,y \rangle = \langle x_C, y_C \rangle$,
\begin{equation} \label{gxcyc}
 [g, x_C]_{\gamma_C} = [g, x_C] + (\lambda_{C} - 1) g x_C, \quad [g, y_C]_{\gamma_C} = [g, y_C] + (\lambda_{C}^{-1}-1) g y_C.
\end{equation}
Let us also define $z_{a,b}^C$ as in \eqref{zabdefn}, but replacing
$x$ and $y$ by $x_C$ and $y_C$.

There now remain two steps: (1) to understand
$(A/[A,Ar'A]_{\gamma_C})_{\Gamma_{\gamma_C}}$, and (2) to compute the
needed relations analogous to \eqref{wabdefn1}--\eqref{wabdefn3}. Both
use the $C$-versions of $x,y,$ and $z$ and $\gamma_C$-commutators.

We begin with (1), which is the easier step.  
Since $|\Gamma|$ is invertible in $\k$, we may replace coinvariants by
invariants in \eqref{hh0a1}. Then, the RHS of \eqref{hh0a1} is
isomorphic to
\begin{equation} \label{hh0a2}
\bigl(A^{\langle \gamma_C \rangle} / (([A,Ar'A]_{\gamma_C})^{\langle \gamma_C \rangle} + (\langle[z_{a,b}, x]_{\gamma_C},[z_{a,b},y]_{\gamma_C}\rangle_{a,b \geq 1})^{\langle \gamma_C \rangle})\bigr)^{\Gamma_{\gamma_C}},
\end{equation}
where $\langle \gamma_C \rangle < \Gamma_{\gamma_C}$ is the cyclic subgroup 
generated by $\gamma_C$.  Invariants under this subgroup are those
polynomials in $x_C$ and $y_C$ such that the bidegree $(a,b)$ satisfies
$\lambda_C^{a-b} = 1$ (in other words, letting $|\gamma_C|$ denote the order
of $\gamma_C$,  $|\gamma_C| \mid (a-b)$).  

Next, we generalize the argument of Lemma \ref{mcl} to prove the
following claim:
\begin{equation} \label{agcf}
A^{\langle \gamma_C \rangle}/([A,Ar'A]_{\gamma_C})^{\langle \gamma_C \rangle}
\end{equation}
is a free graded $\k$-module with basis the classes
\begin{enumerate}
\item[(a)] $z^C_{a,b}$ for $a,b \geq 0$ and $\lambda_C^{a-b} = 1$,
\item[(b)] $[z^C_{a_1,b_1} r' \cdots z^C_{a_m,b_m} r']$, for $m \geq
  1$ and $a_i,b_i \geq 0$, \textbf{except} for those classes of the
  form
\begin{equation} [f^\ell], \text{ where $f$ has bidegree $(a,b)$ with
    $\lambda_C^{a-b} \neq 1$ (i.e.,
    $\gamma_C \cdot f = \lambda_C^{a-b} f \neq f$)}.
\end{equation}
\end{enumerate} We now prove the claim.  Set $F:=\k\langle x, y,
r'\rangle$. Then $F^{\langle \gamma_C \rangle}$ itself has a basis
consisting of noncommutative monomials in $x, y$, and $r'$ of
bidegrees $(a,b)$ satisfying $\lambda_C^{a-b}$. In particular, part
(i) of Lemma \ref{mcl} generalizes, with the same proof, to yield that
$A^{\langle \gamma_c \rangle}$ itself has a basis consisting of
elements of the form
\[
z_{a_1,b_1}^C r' \cdots z_{a_m,b_m}^C r' z_{a_{m+1},b_{m+1}}^C, \quad
\text{s.t. } \lambda_C^{\sum_i (a_i-b_i)} = 1.
\]
To do so, we apply again the Diamond Lemma with the same partial
ordering as before.

Next, note that $F^{\langle \gamma_C \rangle} / [F, Fr'F]^{\langle
  \gamma_C \rangle}$ has a basis consisting of (1) noncommutative
monomials in $x$ and $y$ of bidegrees $(a,b)$ such that
$\lambda_C^{a-b} = 1$, and (2) cyclic words in $x$, $y$, and $r'$, of
bidegrees $(a,b)$ such that $\lambda_C^{a-b} = 1$, and such that, if
the cyclic word repeats itself $m$ times (i.e., it is periodic of
length $\frac{1}{m}$ times the length of the word), then it remains
true that $\lambda_C^{(a-b)/m} = 1$.  In particular, $F^{\langle
  \gamma_C \rangle} / [F, Fr'F]^{\langle \gamma_C \rangle}$ is a free
$\k$-module.

Now, we can use the same partial ordering defined in the proof of
Lemma \ref{mcl} on $F^{\langle \gamma_C \rangle} / [F, Fr'F]^{\langle
  \gamma_C \rangle}$ to deduce the claim.

(2) Taking invariants under $\Gamma_{\gamma_C}$ preserves the
property of being torsion-free, and in fact yields a projective module
(since it passes to a summand, namely the image of the symmetrizer
idempotent $\frac{1}{|\Gamma_{\gamma_C}|}\sum_{g \in \Gamma_{\gamma_C}}
g$). By \eqref{hh0a2}, it remains only to compute the remaining
relations $(\langle[z_{a,b}, x]_{\gamma_C}$,
$[z_{a,b},y]_{\gamma_C}\rangle_{a,b \geq 1})^{\langle \gamma_C
  \rangle}$, so that modding \eqref{agcf} by these, and taking
$\Gamma_{\gamma_C}$-invariants, yields $(A \rtimes \Gamma)_{\cyc}$.

We claim that
\begin{equation}
\langle [z^C_{a,b}, x_C]_{\gamma_C}, [z^C_{a,b}, y_C]_{\gamma_C} 
\rangle_{a,b \geq 0} = 
\langle [z^C_{a,b}, x_C]_{\gamma_C}, [z^C_{a,b}, x_C y_C]_{\gamma_C},
[y_C^{b}, y_C] \rangle_{a,b \geq 0}.
\end{equation}
This follows from the general formula
$[f,gh]_{\gamma_C} = [f(\gamma_C \cdot g),h]_{\gamma_C} + [hf,
g]_{\gamma_C}$.  

Thus, we will consider commutators of the form
\begin{equation} \label{zgcs}
[z^C_{a-1,b}, x_C]_{\gamma_C} (a \geq 1), \quad [z^C_{a-1,b-1}, x_C y_C]_{\gamma_C} (a,b \geq 1), \quad [y_C^{b-1}, y_C]_{\gamma_C} (b \geq 1).
\end{equation}
We would like to eliminate the second commutator above similarly to
\eqref{xycomms}. There are two complications: the first is that we are
now using $\gamma_C$-commutators; the second is that we are working in
the $\k$-module of $\gamma_C$-twisted cyclic words
($A^{\langle \gamma \rangle}/([A,Ar'A]_{\gamma_C})^{\langle \gamma_C
  \rangle}$).

Let us assume $C \neq \{1\}$, since otherwise all our work is done in
Lemma \ref{mcl}.  For $a > b$, let us consider the equation
\begin{equation} \label{zagb1}
\sum_{i=0}^{a-b-1} [x_C^i (x_C y_C)^b x_C^{a-b-i-1}, x_C]  + \sum_{i=0}^{b-1} [(x_Cy_C)^i
x_C^{a-b} (x_Cy_C)^{b-i-1}, x_Cy_C] = 0.
\end{equation}
To turn this into an identity involving the $\gamma_C$-commutators in
\eqref{zgcs}, we first note the following: if $f \in A$ is an eigenvector
for the action of $\gamma_C$ (e.g., if $f$ is monomial in $x_C,
y_C$)
with eigenvalue $\lambda_C^\ell$, and $g \in A$ is arbitrary, then
\begin{equation}
[f g, f] = f [g,f] \equiv \lambda_C^\ell [g,f] f = \lambda_C^\ell [gf, f] \pmod{([A, Ar'A]_{\gamma_C})^{\langle \gamma_C \rangle}}.
\end{equation}
In particular, for $a > b$ (requiring $\lambda_C^{a-b} = 1$), 
\begin{gather} \label{zagb2}
[x_C^i (x_Cy_C)^b x_C^{a-b-i-1}, x_C] \equiv \lambda_C^{i} [z^C_{a-1,b}, x_C] \pmod{([A,Ar'A]_{\gamma_C})^{\langle \gamma_C \rangle}}, \\ \label{zagb3}
[(x_Cy_C)^i x_C^{a-b} (x_Cy_C)^{b-i-1}, x_Cy_C] \equiv [z^C_{a-1,b-1}, x_Cy_C] \pmod{([A,Ar'A]_{\gamma_C})^{\langle \gamma_C \rangle}}.
\end{gather}
Now, we combine  \eqref{gxcyc} with \eqref{zagb1},\eqref{zagb2}, \eqref{zagb3}
to obtain
\begin{equation}
  \frac{1 - \lambda_C^{a-b}}{1-\lambda_C} ([z^C_{a-1,b}, x_C]_{\gamma_C} + (1 - \lambda_C) z^C_{a,b}) + b [z^C_{a-1,b-1}, x_C y_C] \in ([A, Ar'A]_{\gamma_C})^{\langle \gamma_C \rangle}.
\end{equation}
Since $\lambda_C^{a-b} = 1$ and $([A, Ar'A]_{\gamma_C})^{\langle \gamma_C \rangle}$ is saturated in $A^{\langle \gamma_C \rangle}$, this says that
\begin{equation} [z^C_{a-1,b-1}, x_C y_C]_{\gamma_C} \in ([A,
  Ar'A]_{\gamma_C})^{\langle \gamma_C \rangle}.
\end{equation} 
It should be possible to verify this explicitly without using \eqref{zagb1},
but it seems more difficult.

Also, 
\begin{equation}
[(x_C y_C)^{b-1}, x_C y_C]_{\gamma_C} = 0, \quad [y_C^{b-1}, y_C]_{\gamma_C} = (\lambda_C^{-1}-1) y_C^b.
\end{equation}
Thus, we conclude that, for $C \neq \{1\}$,
\begin{equation}
(\langle[z^C_{a-1,b}, x_C]_{\gamma_C}, [z^C_{a,b-1},y_C]_{\gamma_C}\rangle_{a,b
  \geq 1})^{\langle
  \gamma_C \rangle} = \langle[z^C_{a-1,b}, x_C]_{\gamma_C}, (\lambda_C^{-1} - 1)y_C^b\rangle_{a,b \geq 1; \lambda_C^{a-b}=1}.
\end{equation}
Also, we easily see that
\begin{equation} [z^C_{a-1,b}, x]_{\gamma_C} \equiv (\lambda_C - 1)
  z^C_{a,b} \pmod{\ldp r' \rdp}.
\end{equation}
Finally, note that $\lambda_C-1$ is invertible in $\k$, as it is a
factor of $|\gamma_C|$ (by plugging $t=1$ into
$\frac{1-t^{|\gamma_C|}}{1-t}$), which is a factor of $|\Gamma|$.
Thus, for $C \neq \{1\}$, the image of the classes $(\langle[z_{a,b},
x]_{\gamma_C}, [z_{a,b},y]_{\gamma_C}\rangle_{a,b \geq 1})^{\langle
  \gamma_C \rangle}$ in $A^{\langle \gamma_C
  \rangle}/([A,Ar'A]_{\gamma_C} + \ldp r' \rdp)^{\langle \gamma_C
  \rangle}$ has the same $\k$-linear span as the image of the elements
$\{z^C_{a,b}\}_{a,b \geq 0}$.

We conclude that the summand of \eqref{hh0a2} corresponding to $C$ is
torsion-free.
More precisely, this summand is the $\Gamma_{\gamma_C}$-invariant part
of the free $\k$-module with basis given by the classes (b) from our
list above (in part (1) of the proof), and hence $\mathbf{k}$-projective.

In the case $C = \{1\}$, on the other hand, we are back in the
situation of Lemma \ref{mcl}, except that we have to take invariants
at the end. So we get the summand
$(V/W)^{\Gamma} \cong V^\Gamma / W^\Gamma$ where $V,W$ are given in
Lemma \ref{mcl} (by construction, $W$ is a $\k[\Gamma]$-module in this
case). Furthermore, the projection of $W$ modulo $\ldp r \rdp^2$ to
$\bigoplus_{a,b} \langle [r' z_{a,b}] \rangle$ is an monomorphism of graded free $\k$-modules
and a $\Gamma$-morphism, hence gives an embedding of
$\k[\Gamma]$-modules (which is a $\C$-isomorphism).  
Thus, the quotient-mod-$\ldp r \rdp^2$ map gives a monomorphism
\begin{equation}
W^\Gamma \into (\langle r' \rangle \o \k[x,y]^\Gamma),
\end{equation}
with image equal to \eqref{immodr2}.
Explicitly, $W^\Gamma$ can be obtained by the inverse of \eqref{immodr2} by pulling back
$\sum \alpha_{a,b} x^a y^b \in \k[x,y]^\Gamma$ with
$\gcd(a,b) \mid \alpha_{a,b}$ to
$\sum \frac{\alpha_{a,b}}{\gcd(a,b)} W_{a,b}$, which must be in
$W^\Gamma$ by the isomorphism.

Since $V/W$ is torsion-free, so is $(V/W)^\Gamma$.  If we considered
instead $V/(W + \langle[(r')^\ell]\rangle_{\ell \geq 1})$, then $V/W$
has a single copy of $\k/\ldp p \rdp$ in each degree $2p^\ell$ for $p$
prime and $p \nmid |\Gamma|$, and $\ell \geq 1$.  To conclude that
$(V/W)^\Gamma$ has the same torsion, we need only check that each
summand of $\k/\ldp p \rdp$ is $\Gamma$-invariant.  This follows
because it is generated by the image of $\frac{1}{p}[r^{p^\ell}] \in
F/[F,F]$, which is $\Gamma$-invariant.

The rest of the theorem follows as in Lemma \ref{mcl}. For the
statement about projectivity of $(A \rtimes \Gamma)_{\cyc,\{1\}}$ in (iv), this
follows because it is the $\Gamma$-invariant submodule of the free
$\k$-module $V/W$ constructed in Lemma \ref{mcl}, so it is a direct
summand of a free module (as in step (2) above).
\end{proof}

\section{Background on NCCI algebras}\label{nccisec}
In order to state and prove many of our results, we will make
essential use of the fact that the preprojective algebra 
of a non-Dynkin quiver is a
noncommutative complete intersection (NCCI) (cf.~e.g.,
\cite[Definition 1.1.1]{EG}). This essentially follows from \cite{EE},
although in the course of our proof we will recover this fact without
needing to refer to it (the arguments needed are similar, however).

In this section, we briefly recall the definitions of NCCI and
formulate and prove properties of them that we will need. We begin
with recollections from \cite{EG}, then proceed to give somewhat
different properties and characterizations (e.g., Proposition
\ref{ournccipr}).

\subsection{Recollections}
Let $A$ be a finitely-presented algebra $A = T_R V / J$.  Assume for a
moment $R$ is a field. Then, in \cite{EG}, $A$ is called an NCCI if $J
/ J^2$ is a projective $A$-bimodule, or equivalently, $A$ has a
projective bimodule resolution of length $2$. In the general case of
$R$ a commutative ring, we also require that $A$ is a
projective $R$-module.

We will additionally assume the following:
\begin{itemize}
\item  $R=\k^I$ for some commutative ring $\k$ and some finite set $I$;
\item $V$ is a finitely-generated, 
$\Z_+$-graded (by ``weight'') $R$-bimodule, which is free over $\k$;
\item $J=\ldp L \rdp$
where $L \subset T_R V$ is a finitely-generated, homogeneous (with
respect to total degree), positive-weight $R$-subbimodule;
\item $L$ is a
minimal generating $R$-subbimodule of $\ldp L \rdp$;
\item For all $i, j \in I$ and all $m \geq 0$, 
 $(iAj)_m$ is free over $\k$ (although this
  will be automatic in our usual context of $\k = \Z$ or a field).
\end{itemize}
Recall above that $A_m$ denotes the $m$-th graded part of $A$.
In particular, we can define the graded matrix Hilbert series of $A$,
which is the following power series in $t$ with coefficients in
$I$-by-$I$ matrices with nonnegative entries:
\begin{equation}
  h(A;t) := \sum_{m \geq 0} [\text{rk } (i A j)_m]_{i,j \in I} t^m 
  \subset \N^{I \times I}[[t]];
\end{equation}
similarly, we can define the same for $J$, or any other graded module
$M$ whose components $(i M j)_m$ are free over $\k$.

Recall the following characterizations of NCCI (the first of which is
\eqref{egflagen} discussed earlier).
\begin{prop} \label{nccipr} \cite{An1}, cf.~\cite[Theorem 3.2.4]{EG}
The following are equivalent:
\begin{enumerate}
\item $A$ is an NCCI ($=\ldp L \rdp / \ldp L \rdp^2$ is a projective
  $T_R V$-bimodule);
\item \label{nphs} $h(A;t) = (1 - h(V;t) + h(L;t))^{-1}$;
\item The Koszul complex of $A$-modules,
\begin{equation} \label{kr}
0 \rightarrow L \o_R A \rightarrow V \o_R A \rightarrow A \onto R \rightarrow 0,
\end{equation}
is exact, where the first nontrivial map is given by restriction of the map
\begin{equation}
T_R V \o_R A \rightarrow V \o_R A,  \quad (v_1 \cdots v_n) \o a \mapsto v_1 \o (v_2 \cdots v_n a),
\end{equation}
and the second nontrivial map is given by multiplication;
\item Anick's $A$-bimodule complex \cite{An2},
\begin{equation} \label{ar}
0 \rightarrow A \o_R L \o_R A \rightarrow A \o_R V \o_R A \rightarrow
A \o_R A \onto A \rightarrow 0,
\end{equation}
is exact,
where the first nontrivial map is given by restriction of the map 
\begin{equation}
A \o_R T_R V \o_R A \rightarrow A \o_R V \o_R A, \quad a \o (v_1 \cdots v_n) \o b \mapsto
\sum_{i=1}^n (a v_1 \cdots v_{i-1}) \o v_i \o (v_{i+1} \cdots v_n b), 
\end{equation}
and the second nontrivial map is given by 
\begin{equation}
a \o v \o b \mapsto av \o b - a \o vb.
\end{equation}
\end{enumerate}
\end{prop}
It is easy to prove the equivalence of parts (2), (3), and (4): this
follows from the Euler-Poincar\'e principle (that the alternating sum
of Hilbert series of an exact sequence is zero).

\begin{rem}\label{r:koszul}
When $\k$ is a field and $L \subseteq V \otimes_R V$
is quadratic, it is easy to deduce from the above that, if $A$ is an
NCCI, it is also Koszul; cf.~\cite[Theorem 2.3.4]{EE}.
\end{rem}
In \cite{EE}, it is proved that, in the special case where $A$ is a
preprojective algebra of a non-Dynkin quiver or $A$ is a partial
preprojective algebra with at least one white vertex (and $\k$ is any
field), then $A$ is Koszul; but actually the proof there shows $A$ is
an NCCI (the definition of NCCI came slightly later in
\cite{EG}). Moreover, it is well-known that, when $A$ is the
preprojective algebra of a Dynkin quiver, then $A$ is not NCCI; in
fact, the minimal bimodule resolution of $A$ is periodic (in
particular infinite), as is the Koszul complex.  Thus one concludes
from \cite{EE} the following, over $\Z$ and hence over an arbitrary
commutative ring $\k$:
\begin{prop}\label{p:prep-ncci}
  A (partial) preprojective algebra $\Pi_{Q,Q_0'}$ is an NCCI if and only
  if either $Q$ is not Dynkin or $Q_0'$ is nonempty.
\end{prop}
As we mentioned at the beginning of the section, we will not need to
use this proposition.

\subsection{General results on NCCIs}
We will use the following alternative
characterization of NCCIs:
\begin{prop} \label{ournccipr} The following are equivalent for $A$ as above:
\begin{enumerate}
\item $A$ is an NCCI.
\item \label{onpag} Equip $T_R V$ with the descending filtration by
  powers of
  $\ldp L \rdp$. Then the canonical map $A *_R T_R L \to \gr T_R V$ is
 an isomorphism.
\item For any section $A \into T_R V$ of the quotient map, the induced
  map $A *_R T_R L \to T_R V$ is an isomorphism.
\end{enumerate}
\end{prop}
\begin{proof}
It is clear that parts (2) and (3) are equivalent since $A$ is a 
free $\k$-module.
We show that (2) is equivalent to exactness of \eqref{kr} ((3) of Proposition
\ref{nccipr}). It
is easy to see that exactness is equivalent to injectivity of $L \o_R A \rightarrow  V \o_R A$ (by definition of $A$).  The latter is equivalent to the following
formula in $T_R V$:
\begin{equation} \label{refainj}
(L \cdot T_R V) \cap (V \cdot \ldp L \rdp) = L \cdot \ldp L \rdp.
\end{equation}
This last equation obviously follows from (2) or (3). Conversely, we
may proceed by induction on $L$-degree of the filtered vector space
$T_R V$.  Let $\gr^L_i T_R V := \ldp L \rdp^i / \ldp L \rdp^{i-1}$.
As the base case, it is clear that $A \cong \gr^L_0 T_R V$.
Inductively, assume that $\gr^L_i T_R V = (A*_R T_R L)^L_i$. 
Then, \eqref{refainj} shows that
\[
\gr^L_{i+1}(L T_R V) \cong  L \cdot \ldp L \rdp^{i} / L \cdot \ldp L \rdp^{i+1}.
\]
Since the multiplication
$L \o_R T_RV \rightarrow T_R V$ is injective, the RHS identifies with
$L \o_R \gr^L_i T_R V \cong L \o_R (A *_R T_R L)^L_{i}$. Similarly,
 for all $j \geq 1$,
\[
\gr^L_{i+1} (V^{j} L \ldp L \rdp^i) = A_j \o_R L \o_R (A *_R T_R L)^L_{i}. \qedhere
\]
\end{proof}
Note that, using the formula
(Lemma 2.2.5 of \cite{EE})
\begin{equation} \label{eelem} h(A) = \frac{1}{1-\alpha}, h(B) =
  \frac{1}{1-\beta} \Rightarrow h(A *_R B) = \frac{1}{1-\alpha-\beta},
  \quad \text{if }\alpha, \beta \in t \Z[[t]] \otimes \End(R),
\end{equation}
one obtains an alternative proof of the fact that \eqref{onpag} of
Proposition \ref{ournccipr} is equivalent to \eqref{nphs} of
Proposition \ref{nccipr}. (Note that \eqref{eelem} is $T_R V *_R
T_R W \cong T_R (V \oplus W)$ in the case that $\alpha, \beta$ are
positive, or for arbitrary $\alpha, \beta$ if one allows $V$ and $W$
to be graded super-vector spaces.)  We preferred a proof using tensors
rather than Hilbert series.

We do not actually need to say much about NCCIs in general, but we
mention these to explain the meaning of conditions (2), (3) of
Proposition \ref{ournccipr}, which we will use heavily. 

Proposition \ref{ournccipr} easily implies the following analogue
of Lemma 5.1.1 of \cite{EG} (which was for RCI algebras):
\begin{prop} \label{starnccip} Let $A$ and $B$ be finitely-presented
  as in Proposition \ref{ournccipr}.
\begin{enumerate}
\item[(i)] $A *_R B$ is an NCCI if and only if
$A$ and $B$ are NCCIs. More generally, consider a commutative diagram
\begin{equation}
\xymatrix{
B \ar@{^{ (}->}[r] \ar@{=}[d] & C \ar@{->>}[r]   & A
\ar@{=}[d] \\
B \ar@{^{ (}->}[r] & A *_R B \ar@{->>}[r] \ar[u]_{\sim,R}^{gr} & A
}
\end{equation}
whose rows are graded $R$-algebra morphisms, and whose middle vertical
arrow is an isomorphism of graded $R$-bimodules (but not necessarily
an algebra homomorphism). Suppose that $C$ is generated as a
$B$-algebra by a positively-graded $R$-bimodule section $\tilde A
\subset C$ of $A$ intersecting $B$ only in $\k$, and that the
middle vertical arrow is obtained by taking the associated graded with
respect to the descending filtration generated by the grading on $B$
and by $|\tilde A \setminus \{0\}|=0$.  Then, $C$ is an NCCI if and
only if $A$ and $B$ are.
\item[(ii)] Assume $L \supset L'$ are finitely-generated,
  positively-graded homogeneous bimodules, which minimally generate
  the ideals $\ldp L \rdp, \ldp L' \rdp$.  If $T_R V / \ldp L \rdp$ is
  an NCCI, then $T_R V / \ldp L' \rdp$ is an NCCI.
\end{enumerate}
\end{prop}
\begin{proof}
(i) Set $A = T_R V / \ldp L \rdp$ and $B = T_R W / \ldp L' \rdp$ for minimal, finitely-generated $L, L'$, so that $A *_R B = T_R (V \oplus W) / \ldp L, L' \rdp$, and $C = T_R (V \oplus W) / \ldp L', \tilde L \rdp$ for some $\tilde L$ mapping isomorphically to $L$ under the projection to $T_R V$.
The result then follows from (2) of Proposition \ref{ournccipr}.

(ii)  Assume $T_R V / \ldp L \rdp$ is an NCCI. Then, 
\begin{equation}
T_R V \cong T_R V / \ldp L \rdp *_R T_R L
\end{equation}
as graded $R$-bimodules (considering some section $T_R V/\ldp L \rdp \into T_R V$).
Modding by $\ldp L' \rdp$, we obtain
\begin{equation}
T_R V / \ldp L' \rdp \cong T_R V / \ldp L \rdp *_R (T_R L / \ldp L' \rdp), 
\end{equation}
By part (i), we may conclude the desired result if
$T_R L / \ldp L' \rdp$ is an NCCI.  But, this is obvious since
$L' \subset L$ is a subbimodule.
\end{proof}
\begin{rem}
More generally, part (ii) of Proposition \ref{starnccip} shows that,
for any NCCI $T_R V / \ldp L \rdp$ (with $L$ minimal and
finitely-generated), if $L' \subset T_R L$, then letting $\bar L'$ be
the image of $L'$ in $T_R V$, one has that $T_R V / \ldp \bar L' \rdp$
is an NCCI if and only if $T_R L / \ldp L' \rdp$ is an NCCI.
\end{rem}

Now, we note that all of the above results also hold for $\k=\Z$,
dealing with free $\Z$-modules (since the definition, in Proposition
\ref{nccipr}.(i), of an NCCI includes being a free $\Z$-module). 

We then deduce the following important result (needed for writing
bases), now with $R=\Z^{Q_0}$ for a quiver $Q$:
\begin{cor} \label{nccicor}
\begin{enumerate}
\item[(i)] If $\Pi_{Q, J}$ is an NCCI and $J \subset J' \subseteq
  Q_0$, then $\Pi_{Q, J'}$ is an NCCI, and
\begin{equation} \label{pijpj}
\gr^{\ldp \1_{J' \setminus J} r \1_{J' \setminus J} \rdp} \Pi_{Q,J'} \cong \Pi_{Q, J} *_R T_R \langle \1_{J' \setminus J} r \1_{J' \setminus J} \rangle,
\end{equation}
by taking the associated graded algebra with respect to powers of the ideal $\ldp
\1_{J' \setminus J} r \1_{J' \setminus J} \rdp$. One also has a graded
$R$-bimodule isomorphism $\liso$ by choosing any graded $R$-bimodule
section $\Pi_{Q, J} \into \Pi_{Q, J'}$;
\item[(ii)] If $Q'$ and $Q''$ are subquivers of $Q$ with vertex set
  $Q_0$ such that $Q_1 = Q_1' \sqcup Q_1''$ (all arrows of $Q$ are in
  either $Q'$ or $Q''$ but not both), and
  $I \subset Q_0$ is the set of vertices incident to arrows from both
  $Q_1'$ and $Q_1''$, then if $\Pi_{Q', J}$ is an NCCI, one has
  \begin{equation} \label{piq123} \gr^{\ldp \overline{Q_1''} 
      \rdp}\Pi_{Q,J} \cong \Pi_{Q', J} *_R \Pi_{Q'', J \cup I},
\end{equation}
by taking the associated graded algebra with respect to the filtration by powers
of the ideal $\ldp  \bar Q_1'' \rdp$. Similarly, one can
obtain a graded $R$-bimodule isomorphism
\begin{equation}
\Pi_{Q,J} \liso \Pi_{Q', J} *_R \Pi_{Q'', J \cup I}
\end{equation}
by picking any weight-graded $R$-bimodule section $\Pi_{Q', J} \subset
\Pi_{Q,J}$ of the quotient, and using the composition $\Pi_{Q'', J \cup I}
\into \Pi_{Q, J \cup I} \onto \Pi_{Q,J}$ of obvious maps.
\item[(iii)] In the situation of (ii), $\Pi_{Q,J}$ is an NCCI if and only if
  $\Pi_{Q'', J \cup I}$ is an NCCI.
\end{enumerate}
\end{cor}
\begin{proof}
  (i) Let $V = \langle\, \dq_1 \, \rangle$. Let $\tilde \Pi_{Q,J}
  \subseteq T_R V$ be a graded $R$-bimodule section of $T_R V \onto
  \Pi_{Q,J}$. Then,
\begin{equation} \label{trve} T_R V = \tilde \Pi_{Q,J} *_R T_R \langle
  \1_{I \setminus J} r \1_{I \setminus J} \rangle = (\tilde \Pi_{Q,J}
  *_R T_R \langle \1_{J' \setminus J} r \1_{J' \setminus J} \rangle)
  *_R \langle \1_{I \setminus J'} r \1_{I \setminus J'} \rangle.
\end{equation}
If we mod by $T_R \1_{I \setminus J'} r \1_{I \setminus J'}$, we
obtain \eqref{pijpj}, which together with \eqref{trve} proves the
desired result.

(ii) It is easy to see that
\begin{equation} \label{piq121} \Pi_{Q, J \cup I} \cong \Pi_{Q',
    J \cup I} *_R \Pi_{Q'', J \cup I}.
\end{equation}
Now, let $V_1 = \langle \, \overline{Q'_1} \, \rangle$ and $V_2 =
\langle \, \overline{Q''_1} \, \rangle$.  By our hypotheses and
\eqref{pijpj}, we may rewrite \eqref{piq121} as
\begin{equation} \label{piq122} \gr^{\ldp \1_{I \setminus J} r
    \1_{I \setminus J} \rdp} \Pi_{Q, J \cup I} \cong
  \Pi_{Q',J} *_{R} T_R \langle \1_{I \setminus J} r_{Q'}
  \1_{I \setminus J} \rangle *_R \Pi_{Q'', J \cup I}.
\end{equation}
Now, modding by $\ldp \1_{I \setminus J} r_Q \1_{I \setminus J} \rdp$,
we replace all instances of $\1_{I \setminus J} r_{Q'} \1_{I \setminus
  J}$ with $-\1_{I \setminus J} r_{Q''} \1_{I \setminus J}$, and
obtain \eqref{piq123} (upon further taking the associated graded
algebra with respect to $\ldp \overline{Q''_1} \rdp$).

(iii) This follows from \eqref{piq123} and Proposition
\ref{starnccip}.(i).
\end{proof}
The corollary implies the inductive result used in \cite{EE} to show
that non-Dynkin quivers or partial preprojective algebras with at
least one white vertex are Koszul: that is, it allows one to reduce to
the case of extended Dynkin quivers and star-shaped quivers whose
special vertex is white and whose branches are of unit length.

\section{Hilbert series and a question about RCI algebras}\label{hsrcisec}
In this short section, we explain the consequences of our results for
Hilbert series in positive characteristic, and pose a question which
would generalize our main theorem. 
This section is \emph{not} needed for the proof of our main results.  

Throughout this section, we will use the abusive notation $\Pi_Q :=
\Pi_Q \otimes \k$, where $\k$ will be always a field, since we are
only interested in Hilbert series.

A main result of \cite{EG} is a formula for the Hilbert series
$h(A;t)$ of certain $\N$-graded algebras $A := \bigoplus_{m \geq 0}
A_m$ over semisimple rings $R = \C^I$ which are noncommutative
analogues of complete intersections, and also for $A/[A,A]$. For the
latter, it turns out to be more natural to describe the graded vector
space $\O(A) := \Sym(A/[A,A])_+$, where the $+$ means to pass to the
positively-graded $\C$-linear subspace of $A/[A,A]$, which we need to
do in order for the symmetric algebra to have finite dimension in each
graded component. (The reason why $\O(A)$ is more natural to describe
is because it is closely related to the $\C$-linear subspace of
functions on the representation variety which are invariant under
change of basis.)


Computing $h(\O(A);t)$ is tantamount to computing
$h(A/[A,A];t)$: explicitly, if $h(A/[A,A];t) = \sum_m a_m t^m$,
then
\begin{equation} \label{Sdef}
\ds h(\Sym(A/[A,A])_+;t) 
= \prod_{m \geq 1} \frac{1}{(1 - t^m)^{a_m}}.
\end{equation}

\subsection{The non-Dynkin, non-extended Dynkin and partial preprojective cases}
\label{mtcs}
First, let us work over the field $\k=\C$.
Let $(Q,J)$ be any pair where $Q$ is a quiver, $J \subset Q_0$ is a subset of
vertices, and either $J \neq \emptyset$ or $Q$ is neither Dynkin nor extended Dynkin.
 Then, by \cite{EG}:
 \begin{equation}\label{egfla} h(\Pi_{Q,J};t) = (1 - t \cdot C + t^2
   \cdot \1_{Q_0\setminus J})^{-1}, \quad h(\O(\Pi_{Q,J});t)= \bigl(\frac{1}{1-t^2}\bigr)^{\delta_{J,\emptyset}} \prod_{m \geq 1} \frac{1}{\det(1 -
     t^m \cdot C + t^{2m} \cdot \1_{Q_0 \setminus J}),} 
\end{equation}
where $C$ is the adjacency matrix of $\dq$, and $\delta_{J,\emptyset} = 1$ if $J = \emptyset$ and $0$ otherwise.  

The above formulas have the following representation-theoretic
interpretation (which was used in their proof in \cite{EG}): Let $V =
\langle \dq_1 \rangle$ and let $L = \langle \1_{Q_0 \setminus J} r \1_{Q_0
  \setminus J} \rangle$: thus, $\Pi_Q = T_{\k^{Q_0}} V / \ldp L \rdp$. Let
$L^\circ = L \cap [T_{\k^{Q_0}} V, T_{\k^{Q_0}} V]$. Then, \eqref{egfla}
expresses as the following formula \cite{EG}:
\begin{equation} \label{egflagen}
h(\Pi_{Q,J};t) = (1 - h(V ; t) + h(L ; t))^{-1}, \quad
h(\O(\Pi_{Q,J})) = \frac{1}{1-h(L^\circ;t)} \prod_{m \geq 1} \frac{1}{\det(1 -
h(V;t^m) + h(L;t^m))}.
\end{equation}
Since $L$ is a minimal generating bimodule of $\ldp L \rdp$, the first
formula above (for $h(\Pi;t)$) says that $\Pi$ is a noncommutative
complete intersection (NCCI) (cf.~Proposition \ref{nccipr}, taken from
\cite[Theorem 3.2.4]{EG}, which was itself taken from, e.g.,
\cite{An1}), and both formulas above together say that $\Pi$ is an
asymptotic representation-complete intersection (asymptotic RCI), by
\cite[Theorem 3.7.7 and Proposition 3.7.1]{EG}.  Note also that, in
fact, $\Pi$ is Koszul (by \cite[Theorem 2.3.4]{EE}, any NCCI with
quadratic relations (i.e., $L \subset T^2_{\k^{Q_0}} V$) is Koszul).

Furthermore, the expression
$\prod_{m \geq 1} \frac{1}{1 - h(V;t^m) + h(L;t^m)}$ can be viewed as
an analogue of the zeta function, so (following \cite{EG}) we define
\begin{equation}
\zeta(V,L;t) := \prod_{m \geq 1} \frac{1}{\det(1 - h(V;t^m) + h(L;t^m))}.
\end{equation}

Finally, the expression $1 - t \cdot C + t^2 = t \cdot ((\frac{1}{t} +
t) - C)$ can be viewed as the ``$t$-analogue'' of the Cartan matrix.
Also, $\frac{1}{1-tx+x^2}$ is the generating function for (suitably
normalized) Chebyshev polynomials $\phi$ of the second type, so that
$h(\Pi;t) = \phi(C;t) := 1 + \sum_{m \geq 1} \phi_m(C) t^m$.

Our main result, Theorem \ref{mt}, together with the fact that $\Pi_Q$
is torsion-free (which is a consequence of \cite{EE}, or alternatively of the
bases we construct), generalizes the above formulas to arbitrary characteristic:
\begin{cor}
  Take any pair $(Q,J)$ where $Q$ is a quiver and $J \subset Q_0$ is a
  subset of vertices, such that either $J \neq \emptyset$ or $Q$ is
  neither Dynkin nor extended Dynkin.
  Over any field $\k$ of characteristic $p > 0$, the Hilbert series of
  $\Pi_{Q,J}$ and $\O(\Pi_{Q,J})$ are given as follows:
\begin{equation}
h(\Pi_{Q,J}) = (1 - t \cdot C + t^2
   \cdot \1_{Q_0\setminus J})^{-1}, \quad h(\O(\Pi_{Q,J});t)
= \bigl(\prod_{\ell \geq 0} \frac{1}{1-t^{2 p^\ell}}\bigr)^{\delta_{J,\emptyset}} \prod_{m \geq 1} \frac{1}{\det(1 -
     t^m \cdot C + t^{2m} \cdot \1),}
\end{equation}
\end{cor}

Using the notation $\zeta, V, L, L^{\circ}$ above, the formulas become
\begin{equation} \label{garhsp}
h(\Pi_{Q,J}) = \frac{1}{1 - h(V;t) + 
   h(L;t)}, \quad h(\O(\Pi_{Q,J});t)
= \zeta(V,L;t) \prod_{\ell \geq 0} \frac{1}{1-h(L^{\circ};t^{p^\ell})}.
\end{equation}
\subsection{A question on asymptotic RCI algebras in positive characteristic}
P. Etingof and the author have also done computer tests of some
finitely presented algebras over $\Z$ which are asymptotic RCI over
$\Q$, and in most sufficiently random cases, \eqref{garhsp} has
held. This motivates the following generalization of Theorem \ref{mt}.
Fix a prime $p \in \Z$, and let $\k := \Z_{\ldp p \rdp}$ be the
localization of $\Z$ at the ideal $\ldp p \rdp$.  Let $I$ be a set,
and $R:=\k^I$.
\begin{ques}\label{rciques} (P. Etingof and the author) 
  Let $A = T_R V / \ldp L \rdp$ be a finitely-presented algebra over
  $R$, with $L$ a minimal generating bimodule.  Further suppose that
  $L$ and $L \cap [T_{R} V, T_{R} V]$ are saturated, and $A$ is an
  asymptotic RCI \cite{EG} (roughly, this says that the Koszul complex
  of the representation varieties over $\k$ are asymptotically exact).
  Is it then true that the $p$-torsion of $A_{\cyc}$ is isomorphic to a
  graded $\F_p$-vector space with basis the classes $r_j^{(p^\ell)}$,
  for $\ell \geq 1$? Here, $(r_j)$ is a lift to $L$ of an $\F_p$-basis
  of $(L \cap [T_R V, T_R V]) \o \F_p$, and $r_j^{(p^\ell)}$ is the
  image of $\frac{1}{p} [r_j^{p^\ell}]$.
\end{ques}
As before, the classes $r_j^{(p^\ell)}$ must be $p$-torsion if
they are nonzero.  Beginning with an algebra $A$ over $\Z$, the
question asks whether these classes are always nonzero in primes such
that $A \otimes_\Z \Z_{\ldp p \rdp}$ is an asymptotic RCI, and whether
they generate all the $p$-torsion.

\subsection{The extended Dynkin case}\label{ss:edc}
The formula for $h(\Pi_{Q,J})$ above still
holds when $Q$ is extended Dynkin and $J = \emptyset$, 
but the second must be modified (since $\Pi_Q$ is still an NCCI but no longer
an asymptotic RCI).  

Suppose that $i_0 \in Q_0$ is an extending vertex of $Q$, i.e., removing
$i_0$ and its incident arrows leaves one with the corresponding Dynkin
quiver. Over $\k = \C$, one has the following isomorphisms of
graded vector spaces:
\begin{equation} \label{edisos}
(HH^0(\Pi_Q) \o \k) 
\iso (i_0 \Pi_Q i_0 \o \k) \iso (HH_0(\Pi_Q)_+ \o \k) \oplus \k =
((\Lambda_Q)_+ \o \k) \oplus \k,
\end{equation}
where the final $\k$ is placed in degree zero.  The first map is given
by the projection $x \mapsto i_0 x i_0$, and the second is given by
restriction of the obvious projection $(\Pi_Q)_+ \onto HH_0(\Pi_Q)_+$,
together with $(i_0 (\Pi_Q)_0 i_0 \o \k) \cong \k$. The isomorphisms
of \eqref{edisos} follow from the Morita equivalences recalled in \S
\ref{gpsec} together with the computation of Hochschild (co)homology
there, since $\Pi_Q \cong \fsf (\k[x,y] \rtimes \Gamma) \fsf$ induces
$i_0 \Pi_Q i_0 \cong \k[x,y]^\Gamma = HH^0(\k[x,y] \rtimes \Gamma)$, and
Hochschild (co)homology is invariant under Morita equivalence.


Note that the maps in \eqref{edisos} make sense also over $\Q$, and they
must also be isomorphisms.  Using the second isomorphism of
\eqref{edisos}, we obtain the following formula:
\begin{gather} \label{edfla} h(\O(\Pi_Q) \o \Q;t) = h(\Sym(i_0
  \Pi_Q i_0 \o \Q)_+;t) = \prod_{m \geq 1} \frac{1}{(1-t^{m})^{a_m}},
  \intertext{where} 1 + \sum_{m \geq 1} a_m t^m = h(i_0 \Pi_Q i_0; t) =
  \bigl(\frac{1}{1 - t \cdot C + t^2 \cdot \1}\bigr)_{i_0 i_0} =
  \phi(C;t)_{i_0 i_0}.
\end{gather}
Here, $_{i_0 i_0}$ denotes the entry of the matrix in the $i_0, i_0$
component.  There is a general formula for NCCI algebras over a
characteristic zero field \cite{EG}:
\begin{equation} \label{gnccif}
h(\O(A);t) = h(\Sym HH_2(A);t) \cdot \zeta(V,L;t).
\end{equation} 
Setting this
equal to \eqref{edfla} in our case yields the following curious
identity from \cite{EG}, of which we will provide a more direct proof:
\begin{equation} \label{egid} \prod_{m \geq 1}
  \frac{1}{(1-t^m)^{\varphi_m(C)_{i_0i_0}}} = \frac{1}{1-t^2} \cdot
  \prod_{m \geq 1} \frac{1}{\det (\1 - t^m \cdot C + t^{2m} \cdot
    \1)},
\end{equation}
where $\varphi_m := \phi_{m} - \phi_{m-2}$ is the $m$-th Chebyshev
polynomial of the first type. Specifically, we use $i_0 \Pi_Q i_0 \o
\Q \cong \Lambda_Q \o \Q$, compute the Hilbert series of the former
using bases, and compare it with the determinant of the $t$-analogue
of the Cartan matrix, to obtain the RHS of \eqref{egid}. To compare it
with the LHS, we use that $\Pi_Q$ is an NCCI over $\C$, which could be verified
explicitly using our bases, but it is easier to use the Morita
equivalence.

From Theorem \ref{dedz} below, together with the fact that $\Pi_Q$ is
torsion-free \cite{EE}, we deduce the following generalization to
arbitrary characteristic:
\begin{prop} Let $Q$ be any extended Dynkin quiver, and $\k$ any field of
characteristic $p > 0$. Then one has the following formulas:
\begin{gather}
h(\Pi_{Q}) = (1 - t \cdot C + t^2
   \cdot \1)^{-1} = \phi(C;t), \quad h(\Lambda_{Q};t) = hT(Q) + \phi(C;t)_{i_0 i_0}, \\
hT(Q) = \begin{cases}  \label{htqf}
  \sum_{m=1}^{\lfloor\frac{n-2}{2}\rfloor} t^{4m}, &
    \text{if $p=2$ and $Q = \tilde D_n$,} \\
  t^4, & \text{if $p=2$ and $Q=\tilde E_6$,} \\
  t^4 + t^8 + t^{16}, & \text{if $p=2$ and $Q=\tilde E_7$,} \\
  t^4+t^8+t^{16}+t^{28}, 
      & \text{if $p=2$ and $Q=\tilde E_8$,} \\
  t^6, & \text{if $p=3$ and $Q \in 
          \{\tilde E_6, \tilde E_7\}$,} \\
  t^6 + t^{18}, & \text{if $p=3$ and $Q=\tilde E_8$,} \\
  t^{10}, & \text{if $p=5$ and $Q=\tilde E_8$,} \\
   0, & \text{otherwise.}
\end{cases}                    
\end{gather}
\end{prop}
Here and elsewhere, we abuse notation and say that $Q = $ some Dynkin or
extended Dynkin quiver if, when orientations are discarded, one obtains the
corresponding quiver (in other words, $Q$ is given by choosing an orientation
on each arrow of the Dynkin or extended Dynkin quiver).

Note that the torsion of $\Lambda_Q$ in the extended Dynkin case
only appears in what we call ``stably bad primes'': $2$ for $D_n$, $2$
and $3$ for $E_6$ and $E_7$, and $2,3$, and $5$ for $E_8$ (this was
observed also in \cite{MOV} where they were called merely ``bad
primes'').  We will see that the same is true in the Dynkin
case. These primes are a subset of what we call ``bad primes'': primes
dividing the order of the group $\Gamma \subset SL_2(\C)$ associated
to the quiver under the McKay correspondence.  We use the term
``stably bad'' because, for types $A_n$ and $D_n$, these primes are
bad independently of $n$.

\subsection{The Dynkin case}\label{ss:dynkin}
In the case that $Q$ is Dynkin and $J = \emptyset$, the formula for
$h(\Pi_Q;t)$ is no longer valid: instead, $h(\Pi_Q;t)$ records the
dimensions of irreducible representations of $Q$ (because $\Pi_Q \o
\C$ is a direct sum of one copy of each).  One may easily show that
$\Pi_Q$ is torsion-free using Gr\"obner generating sets (cf.~Appendix
\ref{gbs} and Proposition \ref{gbp} therein).

Furthermore, as is proved in \cite{MOV}, $\Lambda_Q \o \Q = 0$, and
in fact $\Lambda_Q \o \F_p = 0$ if $p$ is not a stably bad prime.
As explained in Theorem \ref{dedz}, $\Lambda_Q$ is isomorphic to
the torsion of $\Lambda_{\widetilde Q}$ (for $\widetilde Q$ the
extended Dynkin quiver which extends $Q$), so so we may compute the
Hilbert series over any field:
\begin{prop}
Let $Q$ be a Dynkin quiver. Then $h(\Lambda_Q;t) = hT(\tilde Q)$, given in \eqref{htqf} (for $\tilde Q=$ the extended Dynkin quiver associated to $Q$).
\end{prop}


\section{Refinement and partial proof of the main theorem 
($r^{(p^\ell)} \neq 0$)}\label{refppsec}
  Recall
that a \emph{forest} is a (directed) graph without (undirected)
cycles.
\begin{prop} \label{bpp} Take any quiver $Q$
  together with a nonempty subset of white vertices, $J \subset Q_0$.
  Let $G \subset \dq_1$ be a forest such that the map $G \rightarrow Q_0$,
  $a \mapsto a_s$ (the source vertex) yields a bijection
  $G \iso Q_0 \setminus J$.
Then, a free $\Z$-basis of $\Pi_{Q, J}$ is given by monomials in
 the arrows $Q$ that do not contain a subword $a a^*$ for
 any $a \in G$.  In particular, $\Pi_{Q,J}$ is an NCCI.

 Furthermore, $\Lambda_{Q, J}$ is a free $\Z$-module with basis given
 by cyclic words not containing $a a^*$ for any $a \in G$.
\end{prop}
\begin{proof}
  A forest satisfying the given condition can be constructed
  inductively as follows: To begin, for every vertex in
  $Q_0 \setminus J$ which is adjacent to $J$, add an arrow to $G$
  with source at that vertex and target in $J$.  Inductively, for
  each vertex of $Q_0 \setminus J$ which is not incident to $G$, but
  is adjacent to a vertex which is incident to $G$, add an arrow with
  source at that vertex and target at a vertex incident to $G$.  When
  the process is completed, one obtains a forest satisfying
  the desired condition.
  
  The first result follows immediately from the Diamond Lemma
  (Propositions \ref{dl1}--\ref{dl3}) if we let the partial order on
  monomials be given by the number of subwords $a a^*$ with $a \in G$
  that appear.  To see that $\Pi_{Q,J}$ is an NCCI, we show (2) of
  Proposition \ref{ournccipr}. To do this, we adjoin generators
  $r_i (= i r i)$ for all $i \in Q_0 \setminus J$, and apply the Diamond
  Lemma to reduce any path to a unique sum of monomials in
  $\overline{Q}_1$ and the $r_i$ not containing $a a^*$ for any
  $a \in G$.

  For the final result, we note that one may still use the Diamond
  Lemma for $\Lambda_{Q, J}$ because the maximum number of swaps
  $a a^* \mapsto a^* a$ for $a \in G$ that may be performed in a
  cyclic word is still finite, since $a \in G$ cannot be a loop. This
  follows reverse-inductively on the distance of an arrow $a$ from $J$,
  using that this distance is bounded.
\end{proof}
The fact that $\Pi_{Q,J}$ is an NCCI was first shown in \cite{EE}
using Hilbert series, and in \cite{EG}, $\Pi_{Q,J}$ was further shown
to be a RCI.  The fact that $\Lambda_{Q,J}$ is torsion-free
over $\Z$ appears to be new.

As an application of the proposition, by comparing the above basis and
the formula \eqref{egflagen} for $h(\O(\Pi_{Q,J});t)$ in the
asymptotic RCI case (a generalization of RCI), one obtains a formula for
computing the number of cyclic words in letters $x_i, y_i, z_j$ not
containing $x_i y_i$ for any $i$.  

To handle the case where there are no white vertices, we need the
\begin{prop}\label{prepncci}\cite{EE} The algebra $\Pi_{Q}$ is an NCCI over $\Z$ if $Q$
is non-Dynkin (in particular, it is torsion-free).
\end{prop}
\begin{proof} By Corollary \ref{nccicor} (and the comments afterward),
  and Proposition \ref{bpp}, one may reduce to the case that $Q$ is
  extended Dynkin. For the extended Dynkin case, the easiest proof is
  to use our bases to show that $\Pi_Q$ is torsion free; then, after
  tensoring with $\C$, one may use the Morita equivalence of $\Pi_Q$
  with $\C[x,y] \rtimes \Gamma$ from \S \ref{gpsec}.
  Alternatively, one could deduce the NCCI property from our
  computation of bases with some effort.
\end{proof}

For any non-Dynkin, non-extended Dynkin quiver $Q$, it is well-known
(and easy to check) that $Q \supsetneq Q^0$ for some extended Dynkin
quiver $Q^0$ with vertex set $Q^0_0$. The results of the previous
section then allow us to write, as graded $R$-bimodules,
\begin{equation} \label{mipe}
\Pi_Q = \tilde \Pi_{Q^0} *_R \Pi_{Q \setminus Q^0, Q^0_0},
\end{equation}
where $\tilde \Pi_{Q^0}$ is an arbitrary graded $R$-bimodule section
of $\Pi_Q \onto \Pi_{Q^0}$, and $\Pi_{Q \setminus Q^0, Q^0_0}$ embeds
canonically into $\Pi_Q$ via the sequence
$\Pi_{Q \setminus Q^0, Q^0_0} \into \Pi_{Q, Q^0_0} \onto \Pi_Q$.

\begin{ntn} \label{tildntn}
In general, if $Q \supsetneq  Q^0$ where $Q^0$ is extended Dynkin, then
we will fix a graded $R$-bimodule section $\tilde \Pi_{Q^0} \subset \Pi_Q$ of $\Pi_Q \onto \Pi_{Q^0}$, which exists because $\Pi_{Q^0}$ is torsion-free.  Then, 
for any subset $U \subset \Pi_{Q^0}$, we denote its image under the section
by $\tilde U$.
\end{ntn}
Now, we proceed to one of our main goals: a description of
$\Lambda_Q$ when $Q$ is non-Dynkin and non-extended Dynkin.  We begin
with 
\begin{prop} \label{mip2} Let $Q \supsetneq Q^0$ where $Q^0$ is
  non-Dynkin.  Let
    $V := \Pi_Q / [\ldp \langle \dq_1 \setminus \overline{Q_1^0} \rangle \rdp,
  \Pi_Q]$.
  Let $B$ be the algebra
  $B := T_R((\Pi_{Q^0})_+ \o_R (\Pi_{Q \setminus Q^0, Q^0_0})_+)$.  We
  have
\begin{enumerate}
\item[(i)]
\begin{equation} \label{vide}
V \liso \Pi_{Q^0} \oplus (B/[B,B])_+ \oplus \Lambda_{Q \setminus Q^0, Q^0_0},
\end{equation}
as $\Z$-modules, where the map is given using \eqref{mipe}. 
\item[(ii)]  $\Lambda_Q \cong V/W$ where $W$ is the image in $V$ of $[\tilde \Pi_{Q^0}, \langle \,\overline{Q^{0}_{1}} \, \rangle]$,
using Notation \ref{tildntn}.
\end{enumerate}
\end{prop}
Note that $(B/[B,B])_+$ in \eqref{vide} has a basis consisting of
alternating cyclic words in $(\Pi_{Q^0})_+$ and $(\Pi_{Q \setminus
  Q^0, Q^0_0})_+$.  Here and below, an ``alternating cyclic word'' in
sets $\mathcal{A}$ and $\mathcal{B}$ means
 a word of the form $a_1 b_1 \cdots a_m b_m$ (for $m \geq 1$)
modulo simultaneous cyclic permutations of the indices, where $a_i \in \mathcal{A}$ 
and $b_i \in \mathcal{B}$.
\begin{proof}[Proof of Proposition \ref{mip2}]
  The first part follows from \eqref{mipe} and its proof, together
  with the Diamond Lemma argument from Lemma \ref{mcl}.\eqref{ampc}.
  The second part follows from the observation (cf.~Lemma
  \ref{mcl}.\eqref{bacpp}) that $[\Pi_Q, \Pi_Q] = [\Pi_Q \langle \dq_1
  \setminus \dzqo \rangle \Pi_Q, \Pi_Q] + [\tilde \Pi_{Q^0}, \langle \overline{Q^0_1}
  \rangle]$. 
\end{proof}
From now on we will consider the case where $Q^0$ is extended
Dynkin. Let $i_0 \in Q^0_0$ be a fixed choice of extending vertex. Fix
any quiver $Q \supsetneq Q^0$. Let $\Gamma \subseteq SL_2(\C)$ be the
group corresponding to $Q^0$.

First, we describe $W$ (defined in Proposition \ref{mip2}.(ii) above) away
from bad primes, i.e., we describe $W \otimes \Z[\frac{1}{|\Gamma|}]$.
  For simplicity, when we say ``we work over $S$'', for a commutative
ring $S$, we
mean that all $\Z$-modules should be tensored by $S$, e.g., $W$ denotes $W \otimes_\Z S$,
and we will omit the tensor product for ease of notation.
\begin{prop} \label{wdescp}
  We work over $\Z[\frac{1}{|\Gamma|}]$, with $W$ the
  $\Z[\frac{1}{|\Gamma|}]$-module obtained by tensoring the one in
  Proposition \ref{mip2}.(ii) by $\Z[\frac{1}{|\Gamma|}]$.  Let $W_0
  \subset W$ be the image of $[\Pi_{Q^0}, \Pi_{Q^0}]$ under
  \eqref{vide}.  Let $V' := V/W_0$ and $W' := W/W_0$. Then, these are
  graded modules with finitely-generated free homogeneous components,
  and
\begin{enumerate}
\item[(i)] $h(W';t) = t^2 \cdot h(HH^0(\Pi_{Q^0})_+;t) = t^2 \cdot h((i_0 \Pi_{Q^0} i_0)_+; t)$.
\item[(ii)] The composition $W' \into V' \onto V'/[\ldp r' \rdp^2]$ is
  injective, giving an isomorphism \\
  $W' \iso  [r'
  \widetilde{HH^0(\Pi_{Q^0})_+}] $, using Notation \ref{tildntn}.
\end{enumerate}
\end{prop}
\begin{proof}
  Over $\k := \Z[\frac{1}{|\Gamma|}, e^{\frac{2 \pi i}{|\Gamma|}}]$,
  this follows from Theorem \ref{mclgam}.(iii) and the partial proof
  of Theorem \ref{mt} contained there, using the projection $\k[x,y]
  \rtimes \Gamma \onto \fsf (\k[x,y] \rtimes \Gamma) \fsf \cong
  \Pi_{Q^0}$.  Since $\Z[\frac{1}{|\Gamma}]$ is a principal ideal
  domain, and the modules $V'$ and $W'$ are graded with
  finitely-generated homogeneous components, each homogeneous
  component of $V'$ and $W'$ is a direct sum of finitely many cyclic
  modules.  On the other hand, the functor $M \mapsto M
  \otimes_{\Z[\frac{1}{|\Gamma|}]} \k$ does not annihilate any cyclic
  modules.  Thus, since the target of this functor is free in each
  graded component, all of the cyclic modules occurring in $V'$ and
  $W'$ must be $\Z[\frac{1}{|\Gamma|}]$, i.e., $V'$ and $W'$ are free
  in each homogeneous component.  Then, the result on Hilbert series
  follows because the aforementioned functor preserves Hilbert series
  of graded modules with finitely-generated free homogeneous
  components.
  \end{proof}
  As a corollary, we can deduce in full generality the easier
  direction of Theorem \ref{mt}, that the classes $r^{(p^\ell)}$ are
  nonzero. We will use this in the proof of the other direction:
\begin{prop}\label{rpnz}
For 
every prime $p > 0$ and any $\ell \geq 1$, the class $r^{(p^\ell)} \in \Lambda_Q$
is nonzero.
\end{prop}
\begin{proof}
  For any quiver $Q^0$, we may perform the same procedure as in Proposition
  \ref{mip2} to
  obtain a basis of
  $A' := P_{\dzq}\langle r'_i \rangle_{i \in Q^0_0} / \ldp i r i + r'_i
  \rdp_{i \in Q^0_0} \cong P_{\dzq}$. 
  First, \eqref{mipe} becomes $A' \cong \Pi_{Q^0} *_{\Z^{Q^0_0}}
  \langle r'_i \rangle$ where we view $\langle r'_i \rangle$ as
  $\Z^{Q^0_0}$-modules by $j r'_i j' = \delta_{ij} \delta_{ij'}
  r'_i$. Let $r' := \sum_i r'_i$.  Then, Proposition \ref{mip2}
  presents $V_{A'} := A'/[A',A'r'A']$ as the direct sum of $\tilde
  \Pi_{Q^0}$ and the free $\Z$-module with basis given by alternating
  cyclic words in a basis of $\Pi_{Q^0}$ and $(r')^{\ell}, \ell \geq
  1$.  We may also compute the relations $W_{A'}$ as in the
  proposition.  For any quiver $Q \supsetneq Q^0$, we have a canonical
  map $A' \rightarrow \Pi_{Q}$ which induces maps $V_{A'} \rightarrow
  V_{Q}$ and $W_{A'} \rightarrow W_Q$.  It is easy to see that $W_{A'}
  \onto W_Q$ is a surjection, since the relations are integrally
  spanned by commutators $[\tilde \Pi_{Q^0}, \tilde \Pi_{Q^0}]$, which
  can be taken in $A'$.

  Now, assume $Q^0$ is extended Dynkin. By Proposition \ref{wdescp},
  the rank of $W_Q$ (which is free) does not depend on the choice of
  $Q$, but only on $Q^0$ (provided $Q \supsetneq Q^0$).  Also, if
  $Q' \supsetneq Q^0$ is the quiver obtained from $Q^0$ by adjoining a
  loop to each vertex in $Q^0_0$, then the map
  $V_{A'} \rightarrow V_{Q'}$ has kernel equal to
  $\langle [r'] \rangle$. By Proposition
  \ref{wdescp}, this shows that the map $W_{A'} \rightarrow W_{Q'}$
  also has kernel equal to $\langle [r'] \rangle$, using
  Theorem \ref{mclgam}.(iii).  So, we obtain an isomorphism
  $W_{A'}/\langle [r'] \rangle \iso W_{Q'}$.  Since $W_Q$ and $W_{Q'}$
  have the same Hilbert series, one also must obtain an isomorphism
  $W_{A'}/\langle [r'] \rangle \iso W_Q$.  Finally, the kernel of
  $V_{A'} \rightarrow V_Q$ contains the kernel of $V_{A'} \rightarrow
  V_{Q'}$, since $[r']$ is also zero in $V_Q$.

  Hence, in each graded degree $m$, one obtains an isomorphism
  $(W_{Q'})_m \iso (W_Q)_m$, and a monomorphism of their saturations,
  $\Sat((W_{Q'})_m) \into \Sat((W_Q)_m)$. Here, the \emph{saturation} of
  a $\Z$-submodule $M \subseteq V$ is the module $\Sat(M) := \{x \in V
  \mid \exists n \geq 1 \text{ s.t. } n \cdot x \in M\}$. We are
  taking the saturation of $(W_{Q'})_m$ inside $V_{Q'}[m]$ and of
  $(W_Q)_m$ inside $(V_Q)_m$.

  In particular, if we lift the class $r^{(p^\ell)}$ in any way to
  $V_{Q'}$, it lies in $\Sat((W_{Q'})_m) \setminus (W_{Q'})_m$ by
  Conjecture \ref{rh} (i.e., the $\tilde A_0$ case of Theorem
  \ref{mt}, proved in \S \ref{hrcsec}), and hence any lift to $V_{Q}$
  also has this property.
\end{proof}

To prove the other direction of Theorem \ref{mt} in full generality,
we generalize the theorem by an analysis in each prime $p$ using
$p$-th powers, as follows.

Let us define $W' := W \cap [\ldp r' \rdp]$ (similarly when we work
over more general commutative rings than $\Z$). This coincides with $W
\cap [\ldp \langle \dq_1 \setminus \dzqo \rangle \rdp]$, since $W/W'$ is
integrally spanned by $[\tilde \Pi_{Q^0}]$.  Let $W'_p \subset V \o
\F_p$ be the image of the map $W' \o \F_p \rightarrow V \o \F_p$
induced by inclusion. Then, for any $[w] \in W'_p$, we may consider
$[w]^p = [w^p] \in (V \cap [\ldp r' \rdp]) \o \F_p$. Note that $[w]^p
\mapsto 0 \in \Lambda_Q$, since the same is true for $[w]$.  Hence,
$[w]^p \in W'_p$ as well.  This observation allows us to state the
following theorem, which will be proved in \S \ref{wdescptpfs}, and
refines the main Theorem \ref{mt}.  We will use the canonical
projection $V \onto \Pi_{Q^0}$ whose kernel is the image of $\ldp r'
\rdp$.  We will work below over $\Z_{\ldp p \rdp}$, and note that
tensoring by $\F_p$ still makes sense, and in particular $W'_p = (W'
\otimes \Z_{\ldp p \rdp}) \otimes_{Z \ldp p \rdp} \F_p$.
\begin{thm} \label{wdescpt} We work over $\Z_{\ldp p \rdp}$ for any prime
  $p$. 
  Then, $W$ has the form
\begin{equation}
W = W_0 \oplus W',
\end{equation}
where $W' = W \cap [\ldp r' \rdp]$ and $W_0$ is a $\Z_{\ldp p
  \rdp}$-submodule such that the composition $W_0 \into V \onto
\Pi_{Q^0}$ is a monomorphism with image $[\Pi_{Q^0}, \Pi_{Q^0}]$.
Moreover, these satisfy the following conditions:
\begin{enumerate}
\item[(i)] $W_0$ is saturated except in the cases that $(Q^0, p) \in
  \{(\tilde D_n, 2), (\tilde E_n, 2), (\tilde E_n, 3), (\tilde E_8,
  5)\}$, when $W_0 = W_{0,s} \oplus W_{0,r}$ with $W_{0,s}$ saturated,
  and $W_{0,r}$ has finite rank and will be described in (iii).
\item[(ii)] $W' = W'_s \oplus W'_r$, where $W'_s$ 
is saturated (see (iv)) and $W'_r$ will be described in (iii).
\item[(iii)] $W_r := W_{0,r} \oplus W'_r$ has a basis of classes $\{f_\ell\}, \ell \geq 1$, with $|f_\ell| = 2p^\ell$,
satisfying
\begin{equation}
\text{ord}_p(f_\ell) = 1, \quad \frac{1}{p} f_{\ell+1} \equiv (\frac{1}{p} f_{\ell})^p \pmod p,
\end{equation}
where $\text{ord}_p(f)$ denotes the greatest nonnegative integer $m$
such that $f$ is a multiple of $p^m$.  One has $h(W_{0,r};t) \leq
hT(Q^0)$ from \eqref{htqf} (with the same $p$).  The image of
$\frac{1}{p} f_\ell$ in $\Lambda_Q$ is $r^{(p^{\ell})}$.
\item[(iv)] 
 $W'_s$ has a basis of classes $\{g_{i,\ell}\}$ as follows:
\begin{equation}
g_{i,\ell+1} \equiv g_{i,\ell}^p \pmod p.
\end{equation}
Here, the $g_{i,0}$ $\Z_{\ldp p \rdp}$-linearly span a submodule which
projects isomorphically mod $[\ldp \langle \dq_1 \setminus \dzqo \rangle
\rdp^3] + \ldp p \rdp$ to $[r' U_p]$, and $U_p \subset HH^0(\Pi_{Q^0})
\o \F_p$ is a certain $\F_p$-vector subspace with Hilbert series
$h(HH^0(\Pi_{Q^0})_+;t) - t^{2p-2}h(HH^0(\Pi_{Q^0});t^p) +
C_p(t)$. Here, $C_p(t) = 0$ unless $\Lambda_{Q^0}$ has $p$-torsion,
in which case, letting $m$ be the smallest positive integer such that
$r^{(p^m)}$ is zero in $\Lambda_{Q^0}$,
\begin{equation}
t^2 C_p(t) = t^{2p} - t^{2p^m} + \begin{cases} \sum_{\ell=1}^{\lfloor \frac{n}{4} \rfloor-1} t^{4(2\ell+1)} - t^{2^{\lfloor \log_2 \frac{2(n-2)}{2\ell+1} \rfloor + 1} \cdot (2\ell+1)}, & \text{if $p=2$ and $Q^0 = \tilde D_n$}, \\
t^{28} - t^{56}, & \text{if $p=2$ and $Q^0 = \tilde E_8$}, \\
0, & \text{otherwise}.
\end{cases}
\end{equation} 
$i_0 U_p i_0$ contains the image of the Poisson bracket 
$\{,\}: (i_0 \Pi_{Q^0} i_0 \o \F_p)^{\o 2} \rightarrow (i_0 \Pi_{Q^0} i_0) \o \F_p$,
to be defined in \SS \ref{lans}, \ref{dans}, and \ref{eans}.
\end{enumerate}
\end{thm}
Here, the Poisson algebra $i_0 \Pi_{Q^0} i_0 \o \F_p$ 
is an analogue of $\C[x,y]^\Gamma$ (and, if
$p$ is a good prime, then $i_0 \Pi_{Q^0} i_0 \o \bar \F_p \cong \bar
\F_p[x,y]^\Gamma$ by the McKay correspondence \cite{CBH}).  For good
primes, the above theorem (except for the final statement about Poisson
bracket) is not difficult to deduce from Theorem \ref{mclgam}.

Theorem \ref{mt} follows from Theorem \ref{wdescpt}, since we deduce
that the torsion of $\Lambda_Q \otimes \Z_{\ldp p \rdp}$ for $Q
\supsetneq Q^0$ is at most $\Z/p$ in each degree $2p^\ell$, and this
is $\Z_{\ldp p\rdp}$-linearly spanned by the nonzero classes
$r^{(p^\ell)}$.

Proposition \ref{wdescp} (together with Propositions \ref{mip2} and
\ref{bpp}) gives us explicit $\Q$-bases of $\Lambda_Q$.  Fix a basis
$\mathcal{P}$ of $\Pi_{Q^0}$ by paths (which can be explicitly done
using Theorems \ref{ant}.(iii) and \ref{dnt}.(i), and Proposition
\eqref{enprop}.(ii)), such that, for every vertex $i \in Q^0_0$, the
basis $\mathcal{P}$ includes a subset $\mathcal{P}_i$ of paths beginning (and
ending) at $i$ which maps to a
basis of $\Lambda_{Q^0} \otimes \Q$: this is possible since
$\Lambda_{Q^0 \setminus i} \otimes \Q = 0$ for all $i$, with $Q^0
\setminus i$ the quiver obtained by deleting $i$ and all arrows
incident to $i$.  Fix also the basis $\mathcal{B}$ of $(\Pi_{Q \setminus
  Q^0, Q^0_0})_+$ given by Proposition \ref{bpp}.
%
\begin{prop} \label{ratbasp} For every non-Dynkin, non-extended Dynkin
  quiver $Q$, a $\Q$-basis for $\Lambda_Q \o \Q$
  is given as follows, depending on a choice of extended Dynkin
  subquiver $Q^0 \subsetneq Q$ on vertices $Q^0_0 \subsetneq Q_0$, and a
  \textbf{fixed} vertex $i' \in Q^0_0$ incident to an arrow of $Q
  \setminus Q^0$:
\begin{enumerate}
\item The basis $\mathcal{P}$ above;
\item 
Alternating cyclic words in $\mathcal{P}$ and $\mathcal{B}$
such that
 the source and target vertices of the basis elements match up (to give
a nonzero product),
\textbf{excluding} cyclic words of the form $[i' a^* a \pi]$ for $a \in G$
and $\pi \in \mathcal{P}_{i'}$,  
where
$G \subset \dq_1 \setminus \dzqo$ is the forest chosen in Proposition
\ref{bpp} (so $a$ must be the unique arrow of $G$ incident to $i'$);
\item A basis of $\Lambda_{Q \setminus Q^0, Q^0_0} \o \Q$, given in Proposition \ref{bpp}.
\end{enumerate}
\end{prop}

Theorem \ref{wdescpt} similarly gives explicit $\F_p$-bases of
$\Lambda_Q \o \F_p$. One way to express this is as above, but
replacing $a^* a \pi$  by elements $(a^* a \pi)^{p^\ell}$ where
$\pi$ is a leading path of a basis element of $\tilde U_p$; we
also have to add in basis elements $r^{(p^\ell)}$ which have nonzero image
in $\Lambda_{Q^0}$. The details are omitted.

\section{The necklace Lie structure on $\Lambda_Q$ and
generalizations}\label{lqs1}
In this section, we will recall the necklace Lie structure on
$(P_\dq)_{\cyc}$ and its quotient $\Lambda_Q$.  In this section only, for
convenience, we will let $P := P_\dq$, $L := (P_\dq)_{\cyc}$, and
$\Lambda := \Lambda_Q$.  We will also let $Q$ be an arbitrary quiver
until \S \ref{fpds}, where we will have it contain an extended Dynkin
quiver once again.

First, we will recall the fundamental necklace Lie bracket on $L$
\cite{BLB,G} as well as the fact that it descends to $\Lambda_Q$ (by
the comment after Lemma 8.3.2 in \cite{CBEG}). Then we will recall the
related ``double Poisson'' bracket on $P$ defined by \cite{VdB}.  We
will need to generalize these structures a bit, and in particular
prove a noncommutative Leibniz identity \eqref{ncleibeqn}.  Then, we
will deduce a fundamental formula relating commutators with the
necklace bracket in the extended Dynkin case (Proposition
\ref{ubrp}). As explained in \S \ref{fpds}, we interpret the latter as
expressing $P=P_\dq$ as a ``free-product'' deformation of $\Pi_Q$ in
the sense of \cite{GScyc}.

At the same time, we recall the necklace Lie bialgebra structure on
$L$ \cite{S} and prove that it also descends to $(\Lambda_Q)_+$, and
derive some results such as a Batalin-Vilkovisky style identity
\eqref{bve}, since they arise naturally in the above context and may
be of independent interest.  However, we will not require anything
about the Lie cobracket in this paper, so the reader not interested in
this can skip all results involving the Lie cobracket.

We will also prove in this section that the classes $r^{(p^\ell)}$
are in the kernel of the Lie bracket and cobracket (Proposition
\ref{rkp}); in fact, as we mention in Remark \ref{r:rkp} below, this
is true for all of the torsion of $\Lambda_Q$.  However, we will not
need these results in this paper, so the reader may skip this as well
if desired.

\subsection{The necklace Lie and double Poisson brackets}
\label{ips}

We first recall the necklace Lie bracket \cite{BLB,G}.  Let
$L := (P_\dq)_{\cyc} = P_\dq/[P_\dq,P_\dq]$.  
Let $\omega$ be the natural symplectic
form on the free $\Z$-module $\Z \cdot \dq_1$, of rank
$|\dq_1| = 2 |Q_1|$, with symplectic basis $(a,a^*)_{a \in Q_1}$.  Also,
for any arrow $a: i \rightarrow j$ in $\dq$, let $a_s := i$ and
$a_t := j$. Then one defines the bracket $\{\,,\}$
by
\begin{gather}
\label{laf}
\{[a_1 a_2 \cdots a_m], [b_1 b_2 \cdots b_n]\} = \sum_{i,j} \omega(a_i,
b_j) [(a_i)_t a_{i+1}
\cdots a_{i-1} b_{j+1} \cdots b_{j-1}], \\
\delta([a_1 a_2 \cdots a_m]) = \sum_{i < j} \omega(a_i,a_j) [(a_j)_t
a_{j+1} \cdots a_{i-1}] \wedge [(a_i)_t a_{i+1} \cdots a_{j-1}].
\label{cofla}
\end{gather}
Here, we need the terms $(a_i)_t, (a_j)_t$
in case one has
$n=m=1$ in the first line, and in case $j=i+1$ or $j=m,i=1$ in the
second line.  
\begin{prop}\cite{S} The above defines the structure of an involutive
  Lie bialgebra on $L$: that is, $L$ is a Lie bialgebra satisfying
  $\br \circ \delta = 0$, where $\br: L \o L \rightarrow L$ is the
  bracket.
\end{prop}

To define these more suggestively,
we may define \cite{S} operators
$\partial_a: L \rightarrow P, D_a: P \rightarrow P \o P$ by the
formulas
\begin{gather}\label{ped}
  \partial_a([a_1 \cdots a_m]) := \sum_{i: \, a_i = a}  (a_i)_t a_{i+1} \cdots a_{i-1}, \\
  \label{ded} D_a(a_1 \cdots a_m) := \sum_{i: \, a_i = a}  a_1 \cdots
  a_{i-1} (a_i)_s \o (a_i)_t a_{i+1} \cdots a_m,
\end{gather}
so that $\{\,,\}: L \o L \rightarrow L, \delta: L \rightarrow L \o L$
are given by (letting $m : P \o P \rightarrow P$ be the
multiplication)
\begin{equation}
\{\,,\} = \sum_{a \in Q_1} \pr \circ m \circ 
(\partial_a \o \partial_{a^*} - \partial_{a^*} \o \partial_a), \quad
\delta = \sum_{a \in Q_1} (\pr \o \pr) (D_a \partial_{a^*} - D_{a^*} \partial_a).
\end{equation}
(Note that $\partial_a$ was notated $\frac{\partial}{\partial a}$ in
\cite{S} and $D_a$ was notated $\tilde D_a$.)

Our first result is then
\begin{prop} \label{lbp} The submodule $\Z \cdot Q_0 \oplus [Pr] \subset
  P/[P,P] = L$ is a Lie bi-ideal, where $\Z \cdot Q_0 =L_0$ is
  the $\Z$-linear span of vertices.  Hence, $\Lambda_+ :=
  \Lambda/\Lambda_0$ is a Lie bialgebra.  Furthermore, $[Pr] \subset
  L$ is a Lie ideal, so that $\Lambda$ itself is a Lie algebra.
\end{prop} 
(As mentioned above, the Lie part was proved in \cite{CBEG} in greater
generality: see Proposition 4.4.3.(ii) and the comment after Lemma 8.3.2.)
\begin{rem}
It is an interesting question to explicitly quantize
$\Lambda$ in the sense of Drinfeld (cf.~\cite{EK}): by \cite{EK} there
exists a functorial quantization of Lie bialgebras; however, we were
unable to find any natural formulas or even an explicit algorithm for
computing formulas which produces a quantization of $\Lambda$. In more
detail, in \cite{S} we found an explicit quantization of $L$ itself,
but we were unable to find a Hopf ideal of that quantization
corresponding to the Lie bi-ideal $\langle Q_0 \rangle \oplus [Pr]$.
\end{rem}

We can also answer the natural question of how the classes
$r^{(p^{\ell})}$ behave under the Lie structure:
\begin{prop} \label{rkp} In $\Lambda_+$, $r^{(p^\ell)}$
  is in the kernel of both the bracket and cobracket. More generally,
  $[i r^m] \in L$ is in the kernel of the bracket for any
  $i \in Q_0, m \geq 1$, and is in the kernel of the cobracket on
  $L_+ := L/L_0$.
\end{prop}
\begin{rem}\label{r:rkp}
The complete kernel of the Lie bracket can be found in
\cite[Theorem 9.2.2]{Sarx}, and in
\cite[Proposition 9.2.1]{Sarx}
 the first statement is also generalized to all of
the torsion of $\Lambda$.
\end{rem}

We will prove Propositions \ref{lbp} and \ref{rkp} in \S
\ref{lbprkppfs}, after developing the noncommutative BV structure on
$L$.

\subsection{Lifting brackets to $P$}\label{lbs}
In \cite{VdB}, lifts of the Lie bracket on $L$ to $P$ were defined as
follows.  Let $\tau$ denote the operator which permute components of
tensor products. Namely, if $\sigma \in \Sigma_{n}$, then
$\tau_\sigma: V_1 \o V_2 \o \cdots \o V_n \rightarrow
V_{\sigma^{-1}(1)} \o V_{\sigma^{-1}(2)} \o \cdots \o
V_{\sigma^{-1}(n)}$ is the result of applying the permutation
$\sigma$.

Then, one has the double Poisson bracket
$\ldb , \rdb: P \otimes P \rightarrow P \o P$:
\begin{equation}
\ldb a_1 \cdots a_m, b_1 \cdots b_n \rdb := \sum_i \omega(a_i,b_j) b_1 \cdots b_{j-1} (a_i)_t a_{i+1} \cdots a_m \o a_1 \cdots a_{i-1} (a_i)_s b_{j+1} \cdots b_n.
\end{equation}
In terms of the operator $D_e$ of \eqref{ded}, one has
\begin{equation}
\ldb , \rdb := \sum_{a \in Q_1} (m \o m) \circ \tau_{(13)} \circ (D_a \o D_{a^*} - D_{a^*} \o D_a).
\end{equation}

As mentioned in \cite{VdB}, the formula
$m \circ \ldb\,,\rdb: P \o P \rightarrow P$ defines a Loday bracket.
Since this kills $[P,P] \o P$, for our purposes it will be more
convenient to work with the induced map
$\{\,,\}: L \o P \rightarrow P$, which we henceforth call the Loday
bracket:
\begin{equation} \label{lodb}
\{[a_1 \cdots a_m], b_1 \cdots b_n\} := \sum_i \omega(a_i, b_j)
b_1 \cdots b_{j-1} (a_i)_t a_{i+1} \cdots a_{i-1} b_{j+1} \cdots b_m.
\end{equation}
In terms of \eqref{ped}, \eqref{ded},
\begin{equation}
\{\,,\}\Bigl|_{L \o P} := 
\sum_{a \in Q_1} m \circ (m \o 1) \circ \tau_{(12)} \circ (\partial_a \o D_{a^*} - \partial_{a^*} \o D_a).
\end{equation}
Combining this with the Lie bracket $\{\,,\}: L \o L \rightarrow L$, one
can consider $\{\,,\}$ to be a Lie bracket on 
$\tilde P := L \oplus P$, by defining
\begin{equation}
\{\,,\} \Bigr|_{P \o L} := - \{\,,\} \Bigr|_{L \o P} \circ \tau_{(12)}, \quad \{\,,\}\Bigr|_{P \o P} = 0,
\end{equation}
without losing any information: in fact, this encodes the fact that
the bracket $\{\,,\}$ is skew on $L$.

We may now compare the bracket $\{\,,\}$ with the algebra multiplication on $P$, and it
turns out that this satisfies a kind of Leibniz rule:
\begin{prop} \label{poisprop}
Defining multiplication $m \bigr|_{(L \o P) \oplus (P \o L) \oplus (L \o L)} = 0$ by $L$
to be zero,  $(\tilde P,\{\,,\},m)$ satisfies 
\begin{equation}\label{ncleibeqn}
\{a,bc\} = b \{a,c\} + \{a,b\} c.
\end{equation}
\end{prop}
\begin{proof}
  The fact that $\tilde P$ is Lie follows from a straightforward
  computation, and \eqref{ncleibeqn} follows immediately from the
  definitions.  Alternatively, one may derive this from the double
  Poisson axioms that $\ldb, \rdb$ satisfies (see \cite{VdB}).
\end{proof}
Also, one has the rather obvious identity:
\begin{equation} \label{brpr}
\pr \circ \{\,,\}\Bigl|_{P \o L} = \{\,,\} \circ(\pr \o 1), \text{ and similarly} 
 \pr \circ \{\,,\}\Bigl|_{L \o P} = \{\,,\} \circ (1 \o \pr).
\end{equation}

It turns out that the cobracket lifts to $P$ as well:
\begin{gather}
\nonumber 
\label{dld} \delta_\ell: P \rightarrow L \o P, \delta: \tilde P \rightarrow \tilde P \o \tilde P, \\
\delta_\ell(a_1 \cdots a_n) :=
\sum_{i < j} -\omega(a_i, a_j) [(a_i)_t a_{i+1} \cdots a_{j-1}] \o a_1 \cdots a_{i-1} (a_{i})_s a_{j+1} \cdots a_n, \\\delta_\ell = \sum_{a \in Q_1} (\pr \o m) \tau_{(12)} \bigl((1 \o D_a) D_{a^*} - (1 \o D_{a^*}) D_a\bigr), \\
\delta\Bigl|_P = \delta_\ell - \tau_{(12)} \delta_\ell,
\end{gather}
with $\delta$ defined on $L$ as usual. One then has
\begin{prop}
The structure $(\tilde P, \{\,,\}, \delta)$ is an involutive Lie bialgebra.
\end{prop}
\begin{proof}
  This follows from the same proof as \cite[\S 2]{S} (one simply
  needs to remember where the start and end of pieces that come from
  $P$ are).
\end{proof}
As before, one may view $\delta$ on $P$ as a lift of $\delta$ on $L$: If we
define $\pr \Bigl|_{L} = \text{id}$, one has
\begin{equation} \label{copr}
(\pr \o \pr) \delta = \delta \circ \pr.
\end{equation}

\subsection{Noncommutative BV structure}
One notices a BV-style connection encompassing the Poisson and Lie bialgebra
structures, whose proof is straightforward and omitted:
\begin{prop}\label{bvp}
  The following BV-style identity is satisfied by $\tilde P$: For any
  $a,b \in P$, one has
\begin{equation} \label{bve}
\delta_\ell(ab) = \delta_\ell(a) (1 \o b) + (1 \o a) \delta_\ell(b) + (\pr \o 1) \ldb a , b \rdb.
\end{equation}
To get an equation involving $\delta$, one can apply $(1-\tau_{(12)})$ to each side
(which originally lives in $L \o P$).
\end{prop}
What this says is that $\delta$ can be defined as the unique cobracket
satisfying \eqref{bve} which vanishes on $P_0 \oplus P_1$.  Since
$\delta$ has total degree $-2$, it must be the unique \textbf{homogeneous}
cobracket (or, indeed, homogeneous linear map $P \rightarrow L \o P$) of
degree $-2$ 
satisfying \eqref{bve}.

\subsubsection{More general cobrackets} The results of this subsection will not
be needed elsewhere in the paper.

One could define more general inhomogeneous cobrackets which do not vanish
on $P_1$. For example, one could start with
$\delta'_\ell(a) = F_a \o a$, for elements $F_a \in L$ assigned to
arrows $a \in \dq_1$, satisfying: $\delta(F_a) = 0$ (using the old
$\delta$ from \eqref{cofla}), $F_a = - F_{a^*}$, and
$\{b b^*, F_a\} = 0, \forall a,b \in \dq_1$. This extends to a unique
$\delta'_\ell$ on $P$ satisfying \eqref{bve}, which induces a Lie coalgebra
structure on $L$.

However, this does not necessarily yield an involutive Lie bialgebra: to
obtain involutivity and the one-cocycle condition, we can set
$\delta_\ell = \sum_{a \in Q_1} -F_a \o \ad [a a^*] + [a a^*] \o \ad
F_a$.

More generally, any $\delta_\ell'$ satisfying \eqref{bve} must be of the form
\begin{equation}
\delta_\ell' = \delta_\ell + \sum_i F_i \otimes \theta_i,
\end{equation}
where $F_i \in L$ and $\theta_i \in \Der(P)$.  The condition that $\delta_\ell'$ induce a cobracket on $L$ (the co-Jacobi condition) is then
\begin{equation} \label{delcurv}
Skew \circ \bigl( \delta(\sum_i F_i \o \theta_i) + \frac{1}{2} \{\sum_i F_i \o \theta_i, \sum_j F_j \o \theta_j\} \bigr) = 0, 
\end{equation}
where
\begin{gather}
\delta(F \o \theta) := \delta(F) \o \theta - F \o [\delta_\ell, \theta], \\ 
[\delta_\ell, \theta] := \delta_\ell \circ \theta - (\theta \o 1 - \1 \o \theta) \circ \delta_\ell, \\
\{F \o \theta, F' \o \theta'\} := (F \wedge \theta(F')) \o \theta' + (F' \wedge \theta'(F)) \o \theta - (F \wedge F') \o [\theta, \theta'], \text{ and} \\
Skew = \sum_{\sigma \in S_3} \text{sign}(\sigma) \cdot \tau_{\sigma}.
\end{gather}
In the case that the $\theta_i$ are inner derivations (equivalently, they
kill $r$; in particular, they descend to $\Pi$), then we may
consider maps of the form $\delta' = \delta + \sum (F_i \o \ad G_i - G_i \o \ad F_i)$. Then, one has $[\delta_\ell, \ad F] = (1 \o \ad) \delta(F)$. Thus, \eqref{delcurv} says
\begin{equation} \label{delcurv2}
Skew \bigl( (\delta \o 1) (\sum_i F_i \wedge G_i) + \frac{1}{2} \{ \sum_i F_i \wedge G_i, \sum_j F_j \wedge G_j \} \bigl) = 0,
\end{equation}
where now $\{F \o G, F' \o G\} = F \o \{G,F'\} \o G$, extended linearly.

This is a version of the Maurer-Cartan equation.  The explanation is
that it is the condition $(D')^2 = 0$, where $D'$ is the operator on
the exterior algebra $\Lambda^{\bullet} L$ induced by $\delta' := (\pr
\o \Id) \circ \delta'_\ell$ and the bracket.  Indeed, as observed in
\cite[\S 2.10]{G}, involutive Lie (super)bialgebra structures on
$\mathfrak{g}$ are equivalent to BV structures on the exterior algebra
$\Lambda^\bullet \mathfrak{g}$ with a differential obtained from maps
$\mathfrak{g} \rightarrow \wedge^2 \mathfrak{g}$ and $\wedge^2
\mathfrak{g} \rightarrow \mathfrak{g}$.  As a consequence, solutions
of \eqref{delcurv2} will give solutions of \eqref{bve} which descend
to involutive Lie bialgebra structures (not merely
cobrackets). Furthermore, given that $\delta$ descends to $\Lambda$,
so does the above $\delta'$.

For example, for any $F_i, G_i \in [P_Q]$, one obtains the class
$\delta_\ell' = \delta +\sum_i (F_i \o \ad G_i - G_i \o \ad F_i)$,
which gives an involutive Lie bialgebra satisfying \eqref{bve}. Note
that this still gives a graded Lie bialgebra using the ``geometric''
grading, given by setting $|Q^*| = 1$ and $|Q|=0$: the bracket and
cobracket then both have degree $-1$.  (The total grading we are using
everywhere else in this paper is the analogue of the Bernstein or
additive grading.  In \cite{GSncbv}, we use the geometric grading
since it exists for much more general algebras (replacing $P_\dq =
T_{P_Q} \Der(P_Q, P_Q \o P_Q)$ by $T_A \Der(A, A \o A)$ for more
general $A$.)

For more details and a general construction of noncommutative
BV structures, see \cite{GSncbv}. For our purposes, we will only
need the identity \eqref{bve}.

\subsection{Proof of Propositions \ref{lbp} and \ref{rkp}} \label{lbprkppfs}
We break Proposition \ref{lbp} into two lemmas:
\begin{lemma} \label{lip}
The $\Z$-submodule $[Pr] \subset L = P/[P,P]$ is a Lie ideal.
\end{lemma}
\begin{proof}
  Note that this is actually a special case of Proposition 4.4.3(ii)
  from \cite{CBEG}. We give an elementary proof.  Pick
  $f = [a_1 \cdots a_m] \in L$ and let $g \in P$ be arbitrary. We make
  use of the Poisson bracket $\{\,,\}$ on $\tilde P = L \oplus P$, to
  obtain
\begin{multline} \label{eq01}
\{[a_1 \cdots a_m], g r\} - \{[a_1 \cdots a_m], g\}r = g\{[a_1 \cdots a_m], r\} \\ =
g \sum_{j} \omega(a_j,a_j^*) (-1)^{[a_j^* \in Q_1]} \left( a_{j+1} \cdots a_{j-1} a_j -
a_j a_{j+1} \cdots a_{j-1} \right) \\ = g \sum_j \left( a_{j+1} \cdots a_{j-1} a_j - a_j a_{j+1} \cdots a_{j-1} \right) = 0.
\end{multline}
Here $[\text{statement}]$ is defined to equal one if ``statement'' is
true and to equal zero if ``statement'' is false.  Now, the result follows from
\eqref{brpr}.
\end{proof}
\begin{lemma}
The $\Z$-submodule $[Pr] + L_0 \subset L$ is a Lie bi-ideal.
\end{lemma}
\begin{proof}
  It is obvious that $L_0$ is in the kernel of the Lie bracket, so we
  need only show that $[Pr] \oplus L_0$ is a Lie coideal.  Let
  $f \in P$ be arbitrary; we compute $\delta_\ell(rf)$ by means of the
  BV identity \eqref{bve}:
\begin{equation}
\delta_\ell(rf) = \delta_\ell(r) (1 \o f) + (1 \o r) \delta_\ell(f) +
(\pr \o 1) \ldb r,f \rdb.
\end{equation}
Now, $\delta_\ell(r) = \sum_{a \in Q_1} (a_t \o a_s - a_s \o a_t)$, so that the first
term on the RHS is in $L_0 \o P$.  The second term is obviously in $L \o Pr$. Let
us compute the last term on the RHS.  Let us assume $f = a_1 \cdots a_n$ is a single path:
\begin{multline} \label{eq02}
\ldb r, a_1 \cdots a_n \rdb = \sum_i \omega(a_i^*,a_i) (-1)^{[a_i \in Q_1]}
\bigl(a_1 \cdots a_{i-1} a_i \o
(a_i)_t a_{i+1} \cdots a_n \\ - a_1 \cdots a_{i-1} (a_i)_s \o a_i \cdots a_n \bigr)=  a_1 \cdots a_{n} \o (a_n)_t - (a_1)_s \o a_1 \cdots a_n,
\end{multline}
since $\omega(a_i^*,a_i) (-1)^{[a_i^* \in Q_1]} = 1$ as before.  This is in
$P_0 \o P + P \o P_0$. The lemma now follows
from \eqref{copr}.
\end{proof}


  

\begin{proof}[Proof of Proposition \ref{rkp}]
  We prove the second assertion (which clearly implies the first,
  since $L$ is a free module). To show that $[i r^m]$ is in the kernel
  of the Lie bracket on $L$ for any $i \in Q_0$ and $m \geq 0$, one
  simply applies $\eqref{eq01}$ multiple times, replacing $r$ with $i
  r i$ (and hence limiting the sum to those arrows $a_j$ which are
  adjacent to $i$).

Showing that $[i r^m]$ is in the kernel of the cobracket is a bit
more difficult.  We show more generally that $i r^m$ is in the kernel
of $\delta_\ell' := (q \o q') \circ \delta_\ell: P \rightarrow L_+ \o P_+$, where
$q: L \onto L/L_0 = L_+, q': P \onto P/P_0=P_+$ are the projections.  
Inductively, we need to
show that
\begin{equation}
0 = \delta_\ell'(i r^{m+1}) = i r \delta_\ell'(i r^m) + \delta_\ell'(i r) r^m
+ (q \circ \pr \o q') \ldb i r, i r^m \rdb = 
(q \circ \pr \o q')\ldb i r, i r^m \rdb.
\end{equation}
Now, considering \eqref{eq02}, one verifies that most of the terms in
$\ldb i r, i r^m \rdb$ cancel, leaving $i r^m \o i - i \o i r^m$.
This is killed by $q \circ \pr \o q'$, verifying the desired result.
\end{proof}

\subsection{Commutators and Poisson brackets in the extended Dynkin case}
\label{fpds} 
Throughout this subsection, let $Q^0$ be a (fixed) extended Dynkin quiver. Let $\Gamma < SL(2,\C)$
be the corresponding finite group under the McKay correspondence.

\subsubsection{Preliminaries on $i_0 \Pi_{Q^0} i_0$ and
$\Lambda_{Q^0}$}
It will be useful to consider the following sequence of natural maps
(cf.~\eqref{edisos}):
\begin{equation}
HH^0(\Pi_{Q^0}) \into \Pi_{Q^0} \onto HH_0(\Pi_{Q^0}) = \Lambda_{Q^0}.
\end{equation}
Tensoring over $\k := \Z[\frac{1}{|\Gamma|},e^{\frac{2 \pi
    i}{|\Gamma|}}]$, this composition is an isomorphism onto
$((\Lambda_{Q^0})_+ \otimes \k) \oplus \k$ (where $1 \in \k$ is the
class of the unit in $P_{\overline{Q^0}}$).  To see this, we can use
the graded Morita equivalence $\Pi_{Q^0} \o \k \simeq \k[x,y] \rtimes
\Gamma$ (cf.~\S \ref{gpsec}), which induces isomorphisms on $HH^0$ and
$HH_0$. Thus, we can replace $\Pi_{Q^0} \o \k$ by $\k[x,y] \rtimes
\Gamma$, and the result follows from the fact that $HH^0(\k[x,y]
\rtimes \Gamma) = \k[x,y]^\Gamma \to HH_0(\k[x,y] \rtimes \Gamma) =
(\k[x,y]^\Gamma \oplus \k[\Gamma]^\Gamma)/\k$ is an isomorphism in
positive degrees.

Note that, if we work over $\Z$ or a field of characteristic dividing
$|\Gamma|$, then the map $HH^0(\Pi_{Q^0}) \to \Lambda_{Q^0}$ is
\emph{not} an isomorphism in positive degrees, even for the cases of
$\tilde A_{n-1}, \tilde D_n$.  For example, in the case $\tilde
A_{n-1}$, a central element involves a sum over all vertices of cycles
beginning and ending at that vertex; for each vertex, the
corresponding summand has the same image in $\Lambda_{Q^0}$ and
hence the sum must be a multiple of $n$, which does not yield an
isomorphism when we don't invert $|\Gamma|=n$.

Back to the situation above with $\k =
\Z[\frac{1}{|\Gamma|},e^{\frac{2 \pi i}{|\Gamma|}}]$, we may
transplant the commutative multiplication on $HH^0(\Pi_{Q^0} \o \k)
\cong \k[x,y]^\Gamma \cong i_0 \Pi_{Q^0} i_0 \o \k$ to
$((\Lambda_{Q^0})_+ \o \k) \oplus \k$ using the isomorphism
$HH^0(\Pi_{Q^0} \o \k) \iso ((\Lambda_{Q^0})_+ \o \k) \oplus \k$
above.  It is easy to check that this multiplication is compatible
with the necklace bracket, i.e., induces a graded Poisson algebra
structure on $((\Lambda_{Q^0})_+ \o \k) \oplus \k$ (it essentially
follows from Proposition \ref{poisprop}). Since there is a unique
generically nondegenerate Poisson bracket on $HH^0(\Pi_{Q^0} \o \k)
\cong \k[x,y]^\Gamma$ up rescaling, this shows that the isomorphism
$HH^0(\Pi_{Q^0} \o \k) \iso ((\Lambda_{Q^0})_+ \o \k) \oplus \k$
carries the standard Poisson bracket to the necklace bracket, up to
scaling.  This actually works over $\Z[\frac{1}{|\Gamma|}]$ since we
do not need the roots of unity to express the center of
$HH^0(\Pi_{Q^0} \o \Z[\frac{1}{|\Gamma|}])$.

It is useful to have the following result (still with $\k= \Z[\frac{1}{|\Gamma|},e^{\frac{2 \pi i}{|\Gamma|}}]$):
\begin{prop}\label{izprop}
  Let $Q^0$ be any extended Dynkin quiver with extending vertex $i_0
  \in Q^0_0$, and let $i \in Q^0_0$ be any vertex.  Let $z \in
  HH^0(\Pi_{Q^0} \o \k)$ be any central element.  Then, taking image
  in $HH_0(\Pi_{Q^0} \o \k) = \Lambda_{Q^0} \o \k$, we obtain
\begin{equation}
[i z] = \dim(\rho_i) [i_0 z], \quad [z] = |\Gamma| [i_0 z].
\end{equation}
\end{prop}
\begin{proof}
Let us consider the sequence 
\begin{equation}
HH^0(\Pi_{Q^0} \o \k) \into 
\k[x,y] \rtimes \Gamma \onto HH_0(\k[x,y] \rtimes \Gamma) \iso \Lambda_{Q^0} \o \k
\end{equation}
We know that the sequence is injective.  By the analysis in \S
\ref{gpsec}, the image of $[iz]$ in $HH_0(\k[x,y] \rtimes \Gamma)$
consists of projection of $z \fsf_i$ to $z e$ where $e \in \Gamma$ is
the identity.  But this is the projection of the idempotent $\fsf_i$
to $e$, which is taking the trace of the identity element in the
representation $\rho_i$, which is the dimension.
\end{proof}

\subsubsection{$P_{Q^0}$ as a free-product deformation of $\Pi_{Q^0}$}\label{ss:fpdef}
Next, let
$Q \supsetneq Q^0$ be any quiver,
and as in \S \ref{gpsec} let
$r' := \sum_{a \in Q_1 \setminus Q_1^0} \1_{Q^0_0} (a a^* - a^* a)
\1_{Q^0_0}$.
For any
$f \in \Pi_{Q^0}$ and $z \in HH^0(\Pi_{Q^0})$, choose lifts $\tilde f, \tilde z \in P_{\dzq}$.
It follows using the Morita equivalence of \S \ref{gpsec} that
\begin{equation} \label{nrpz}
[\tilde z,\tilde f] \equiv - \mu_{r'} \ldb \tilde z, \tilde f \rdb  \pmod{\ldp r' \rdp^2 + [\ldp r' \rdp, f]},
\end{equation}
where
\begin{equation}
\mu_{r'} (a \o b) := a r' b.
\end{equation}
For example, in the case
$\Gamma=\{1\}$, the above is an expansion of $[x^a y^b, x^c
y^b]$ under the relation $[x,y] = r'$. In this case, $Q^0$ is the
quiver with one vertex and one loop, and
\begin{equation}
[x^n, y] = -\sum_{i=0}^{n-1} x^i r' x^{n-i-1} \equiv -r' x^{n-1} \pmod{[\ldp r' \rdp^2] + [\ldp r' \rdp, \Pi_Q]}.
\end{equation}
Let us return to the case of general extended Dynkin $Q^0$.  
\begin{rem}
Because of Proposition \ref{ppp}, or by the same argument, one
deduces that \eqref{nrpz} remains true after replacing $\Pi_Q$ by
$P_{\dzq}$, or more generally by $\Pi_{Q,J}$, for $Q \supset Q^0$ and
either $Q \neq Q^0$ or $J \neq \emptyset$; we then set $r' = -\sum_{a
  \in Q_1^0} (aa^*-a^*a) \in \Pi_{Q,J}$.

Now, replace $\Pi_Q$ by $P_{\dzq}$. By \eqref{mipe}, as in the proof
of Proposition \ref{mip2}.(i), one deduces that $P_{\dzq}/ [\ldp r'
\rdp, f]$ and $P_{\dzq}/([\ldp r' \rdp, f] + \ldp r' \rdp^2)$ are
torsion-free. Hence, \eqref{nrpz} holds over $\Z$ (working in
$P_{\dzq}$).

We interpret \eqref{nrpz}, together with the NCCI property that
$P_\dzq \cong \Pi_{Q^0} *_R T_\Z \langle r \rangle$ (with $R = \Z^{Q_0}$),
as saying that $P_\dzq$ is a \textbf{noncommutative} or \textbf{free
  product} deformation of $\Pi_{Q^0}$, which ``quantizes'' the double
bracket $\ldb\, ,\rdb$ (more precisely, the Poisson bracket on
$HH^0(\Pi_{Q^0})$, using the following propositions). One can also say
that $P_Q$ is a noncommutative deformation of $\Pi_Q$ for any
non-Dynkin, non-extended Dynkin quiver (the NCCI property yields
``flatness''), but without a statement about Poisson bracket. For more
details and the general theory of this type of ``free product''
deformation, see \cite{GScyc}.
\end{rem}

We may now deduce the following useful result, over $\Z$:
\begin{prop} \label{ubrp} Let $Q \supsetneq Q^0$ where $Q^0$ is extended Dynkin. 
For any $z \in HH^0(\Pi_{Q^0})$ and $x \in \Pi_{Q^0}$, and any lifts $\tilde z, \tilde x$
to $\Pi_Q$, 
\begin{equation}
\tilde z \tilde x - \tilde x \tilde z \equiv [r' \psi(\{[i_0 z],[x]\})] \pmod{[\ldp \langle \dq_1 \setminus \dzqo \rangle \rdp,  \Pi_{Q}] + [\ldp \dq_1 \setminus \dzqo \rdp^3]},
\end{equation}
where $\psi: \Lambda_{Q^0} = HH_0(\Pi_{Q^0}) \onto HH^0(\Pi_{Q^0})$ is the composition
$$HH_0(\Pi_{Q^0}) \onto HH_0(\Pi_{Q^0}) / (\text{torsion}\oplus \bigoplus_{i \neq i_0} [i]) \cong i_0 \Pi_{Q^0} i_0 \cong HH^0(\Pi_{Q^0}).$$
\end{prop}
\begin{proof}
  It follows from Proposition \ref{mip2} that $\Pi_{Q} / ([\ldp \langle
  \dq_1 \setminus \dzqo \rangle \rdp, \Pi_{Q}] + \ldp \langle \dq_1
  \setminus \dzqo \rangle \rdp^3)$ is torsion-free.  Hence, it is
  enough to prove the above formula tensored over $\k$, where this
  follows from \eqref{nrpz} and Proposition \ref{izprop}.
\end{proof}
\begin{rem}
Equation \eqref{nrpz}
 also implies that, letting
\begin{equation}
\pi: P_{\dzq} \onto \Pi_{Q^0}
\end{equation}
be the projection, 
\begin{equation}
\mu \ldb P_{\dzq}, \pi^{-1}(HH^0(\Pi_{Q^0})) \rdb \subset \pi^{-1}(HH^0(\Pi_{Q^0})),
\end{equation}
where $\mu$ is the multiplication map. In other words,
$\pi^{-1}(HH^0(\Pi_{Q^0}))$ is a Loday ideal with respect to the Loday bracket
$L_{\dzq} \o P_{\dzq} \rightarrow P_{\dzq}$ of \S \ref{lbs}.
\end{rem}

\section{Quivers containing $\tilde A_{n-1}$} \label{ans}
\subsection{Bases of $\Pi_Q$ for type $A$ quivers and refinement of
Theorem \ref{mt}} 
\label{typeabasessec}
Here we describe bases of $\Pi_Q$ for extended Dynkin quivers of type
$A$ and quivers which properly contain them. The resulting Theorem
\ref{ant} proves Theorem \ref{mt} for all quivers containing a cycle. 
We work over $\Z$ throughout, i.e., with $\Z$-modules.

Let $Q^0 = \tilde A_{n-1}$, as depicted in Figure \ref{anf}, forming a polygon
which is oriented counter-clockwise (none of the results
depend on this choice; in particular the choice of orientation does not
affect the structure of $\Pi_{Q^0}, \Pi_Q$, or their zeroth Hochschild homology;
we only make this choice for convenience).
We define the following:
\begin{equation} \label{anxydefn}
x = \sum_{a \in Q_1^0} a, \quad y = \sum_{a \in Q^0} a^*.
\end{equation}
We call an arrow is \emph{counter-clockwise-oriented} or
\emph{clockwise-oriented} depending on whether it forms part of the
counter-clockwise or clockwise orientation of the polygon in Figure
\ref{anf}, i.e., whether it is in $Q^0$ or not.  Let the vertex set
$Q^0_0$ be given the natural structure of a $\Z/n$-torsor (i.e., affine
space over $\Z/n$), where adding one means moving once
counter-clockwise (if the reader prefers, one can let $Q^0_0 = \Z/n$,
choosing a fixed vertex to be labeled by zero). 
\begin{figure}[hbt]
\begin{center}
\includegraphics{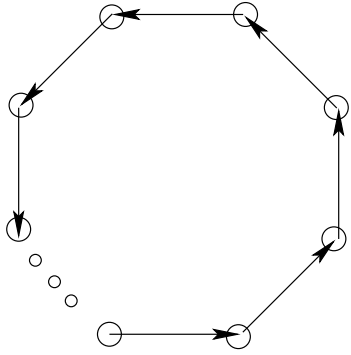}
\caption{$Q^0 = \tilde A_{n-1}$ with counter-clockwise orientation}
\label{anf}
\end{center}
\end{figure}

We then have the
\begin{thm} \label{ant}
Let $Q^0 = \tilde A_{n-1}$ with the above notation and orientation.
\begin{enumerate}
\item[(i)]
For any $i, j \in Q^0_0$, a basis of $i \Pi_{Q^0} j$ is given by $ix^a y^b j$ for
$(a-b) \equiv (j-i) \pmod n$.
\item[(ii)] A basis of $i \Pi_{Q^0} j$ is also given by the nonzero elements
$i z_{a,b} j$ \eqref{zabdefn}, which are equal to the $i x^a y^b j$ above.
\item[(iii)] $\Lambda_{Q^0}$ is a free $\Z$-module with basis given,
  for any fixed $i_0 \in Q^0_0$, by the classes
  $[i_0 x^a y^{b}] = [i_0 z_{a,b}]$ for $a, b \geq 0$ and $n \mid (b-a)$.
\item[(iv)] A basis of $HH^0(\Pi_{Q^0})$ is given by $z_{a, b}$ for
$a, b \geq 0$ and $n \mid (b-a)$.
\item[(v)] For any $Q \supsetneq Q^0$, and any fixed vertex $i_0 \in Q^0_0$,
  $W$ of Proposition
  \ref{mip2}.(ii) has the form $W = W' \oplus W_0$, where $W_0$ is the
  image of
 $[\Pi_{Q^0}, \Pi_{Q^0}]$ under \eqref{vide} and
  and $W'$ is a free $\Z$-module with basis the classes $W_{a,b}$
  given by \eqref{wabdefn1}--\eqref{wabdefn3}.  As in the case of
  Lemma \ref{mcl}, $W_{a,b}$ is a generator of the $\Z$-submodule
   $\langle [iz_{a-1,b}, x], [iz_{a,b-1}, y] \rangle_{i \in
    Q^0_0}$.
\item[(vi)] The integral span of the classes $W_{a,b}$, for $(a,b)
  \neq (p^\ell,p^\ell)$ is a saturated $\Z$-submodule of $V$. The
  order of $W_{p^\ell,p^\ell}$
 in $V$ is $p$.
\item[(vii)] The image of $\frac{1}{p} W_{p^\ell,p^\ell}$ in
  $\Lambda_Q$ is nonzero, equal to $r^{(p^\ell)}$, and every
  $p$-torsion class in $\Lambda_Q$ is an integral combination of
  these classes.
\end{enumerate}
\end{thm}
Parts (i)--(iv) are easy and their proofs are omitted. The remainder
of the theorem will be proved in \S \ref{ss:ant-proof}. As a corollary, one
deduces Theorem \ref{mt} in this case, and moreover easily obtains a
$\Z$-basis of $\Lambda_Q$ modulo torsion together with $\F_p$-bases
of all of the $p$-torsion for all $p$.  Also, note that Conjecture
\ref{rh} is a special case of $n=1$ (the case of quivers with one
vertex). We remark also that the above yields bases of $\Pi_{A_{n-1}}$ as
well, by taking the images of all paths in the above basis of $\Pi_{\tilde A_{n-1}}$ and discarding those whose image is zero.
This is true because, for
each $i,j \in Q^0_0$ and $m \geq 1$, there is at most one basis element
in $(i \Pi_{\tilde A_n} j)_m$ above that projects to a nonzero element
of $(i \Pi_{A_n} j)_m$.

\subsection{The case of $\tilde A_0$}
\label{a0pfs} Although Corollary \ref{mcl} already implies Theorem
\ref{mt} in the case $Q^0 = \tilde A_0$ (cf.~Remark \ref{mclmtr}), and
in fact Theorem \ref{ant} for $\tilde A_0$, we give a slightly
different explanation using bases, as this will be generalized to
$\tilde A_{n-1}$ for general $n$ in the next subsection.

Let $Q \supsetneq Q^0 = \tilde A_0$, i.e., $Q$ is a quiver containing
an arrow which is a loop, say $x \in Q_1$, based at the vertex $i_0 \in
Q_0$.  Let $Q^0$ be the subquiver consisting of just the vertex $i_0$
and the loop $x$.  Fix a maximal tree $G \subset \dq$ as in
Proposition \ref{bpp}: here this means that all arrows of $G$ are
oriented towards the vertex $i_0$ (one can follow oriented arrows of
$G$ to arrive at $i_0$ from any vertex).  We define $G^* := \{a^* \mid
a \in G\}$, and one evidently has $G \cap G^* = \emptyset$.

Let the reverse of the arrow $x$ be 
 $y := x^*$, and take $z_{a,b}$ defined as in \eqref{zabdefn}.

Let us make the isomorphism $\Lambda_Q \cong V/W$ from Proposition
\ref{vide} more explicit in this case.  Define
$r_{Q^0_0} = \1_{Q^0_0} r \1_{Q^0_0} = i_0 r i_0$ as in the setup of
Proposition \ref{mip2}.
Let $F := \Pi_{Q,Q^0_0}$, and set $r' := r_{Q^0_0} -xy + yx \in
\Pi_{Q\setminus Q^0, Q^0_0}$. Set $A = F/\ldp i_0 r i_0 \rdp \cong \Pi_Q$.

Then, $V := A/[A(\overline{Q}_1 \setminus \overline{Q^0_1})A,A]$ has a basis
consisting of:
\begin{enumerate}
\item Cyclic words in $\overline{Q}_1$
 containing an arrow from
  $\overline{Q \setminus Q^0}$, such that maximal subwords from
  $\overline{Q \setminus Q^0}$ are as dictated by Proposition \ref{bpp}, 
  and maximal
  subwords from $\overline{Q^0}$ are of the form $z_{a,b}$;
\item Monomials of the form $z_{a,b}$,
\end{enumerate}
and is free (cf.~Proposition \ref{bpp}).  Then,
$\Lambda_Q \cong A/[A,A] \cong V/W$ where
$W = \langle W_{a,b} \rangle_{a,b \geq 1, (a,b) \neq 1}$ is as
described in Lemma \ref{mcl}.
We have assumed that $Q \setminus Q^0$ is nonempty, so that
$r' \neq 0$. Also, since $r'$ is a sum of commutators, 
$[(r')^{p^\ell}]$ is a multiple of $p$ (as a class of $F/[F,F]$, and
hence in $V$). Then, the rest of the result follows immediately from
(the proof of) Lemma \ref{mcl}. We note that we could have
alternatively proved this result by presenting $\Pi_Q$ as a case of
Corollary \ref{mclcor} (using that $\Lambda_{Q\setminus Q^0,Q^0_0}$ is
torsion free).

This finishes the proof of Theorem \ref{ant} and hence Theorem
\ref{mt} in the $\tilde A_0$ case, which includes Conjecture \ref{rh}
as a special case.

\subsection{Proof of Theorem 
\ref{ant}}\label{ss:ant-proof}
The proof generalizes the argument of 
the previous subsection.
When there is any chance of confusion, if $i \in Q^0_0$ and
$m \in \Z$, we use underlined notation, $\underline{i+m} \in Q^0_0$, for
the result of adding $m$, i.e., moving $m$ steps in the counterclockwise
direction.
We can describe the
path algebra $P_{\overline{Q^0}}$ as generated by $Q^0_0, x,$ and $y$,
with relations/conditions: (1) $Q^0_0$ are idempotents of degree zero;
(2) $x$ and $y$ have degree $1$; and (3) $i x = x (i+1),$ $(i+1) y = y
i$, and $ixj = jyi = 0$ if $j \neq i+1$.  As before, we can define
$z_{a,b}$ by \eqref{zabdefn}. Here $x$ and $y$ are given in
\eqref{anxydefn}.
 
We define $r_{Q^0_0} := \sum_{i \in Q^0_0} i r i$ and $r' := r_{Q^0_0} -
\sum_{a \in Q_1} (aa^* - a^*a)$ analogously to the case $n=0$.
Let $F = \Pi_{Q,Q^0_0}$ and $A = F / \ldp r_{Q^0_0} \rdp = \Pi_Q$.  
We will compute $W$ and the quotient $V / W$, in the notation of
Proposition \ref{mip2}. Clearly, $i W j = i V j$ when $i \neq j$, so
it suffices to consider $\bar W := \bigoplus_{i \in Q^0_0} i W i$.

In this case, $\bar W$ is generated by reducing the elements $[z_{a-1,b},
ix], [z_{a,b-1}, iy] \in P_{\dq}$ for all $a,b \in \N$ with $n \mid
|b-a|$, and all $i \in Q^0_0$, modulo $r_{Q^0_0} = xy - yx +
r'$ and commutators of elements which include a multiple of $\langle
Q\setminus Q^0 \rangle$ (including $r'$).

One finds relations similar to \eqref{wabdefn1}--\eqref{wabdefn3},
keeping track of the idempotents $Q^0_0$, and substituting
$r_{Q^0_0}-xy+yx$ for $r$.  We will now see why our choice of the
$z_{a,b}$'s (intended originally to carry over well to the present
setting) is convenient, stemming from the property that $i xy = xy i$
for all $i \in Q^0_0$. Let us assume that $a > b$, since essentially the
same relations result in the other case (and the torsion must be the
same).  We may assume that $n \mid (b-a)$, or else the bidegree
$(a,b)$-part is zero.  Let $\eta: A \onto V$ be the projection. Since,
writing $iz_{a,b} = t_1 t_2 \cdots t_{a+b}$ where $t_j \in
\overline{Q^0_1}$, one has
\begin{equation} \label{sc1}
\sum_{m=1}^{a+b} \eta([t_{m} t_{m+1} \cdots
t_{m-2}, t_{m-1}]) = 0,
\end{equation}
it follows (using $\eta([Ar'A,A])=0$) that 
\begin{equation} \label{sc2}
b\eta([iz_{a,b-1}, y]) +
b\eta([(\underline{i+1})z_{a-1,b}, x]) + \sum_{m = 0}^{a-b-1} \eta([(\underline{i+m})
z_{a-1,b}, x]) = 0.
\end{equation}  
So, we need only compute the $\eta([i
z_{a-1,b},x])$, or equivalently, the $\eta([z_{a-1,b}, ix])$:
(recall that $Q^0_0$ is considered as a $\Z/n$-torsor and
$\underline{i+m}$ is the operation of adding $m \in \Z$ to $i \in
Q^0_0$):
\begin{equation}
\eta([z_{a-1,b}, ix]) = [(\underline{i+1} - i) z_{a,b}] + 
[\sum_{c =0}^{b-1} (\underline{i-(a-b-2)}) r'(xy)^c x^{a-b-1} (xy)^{b-1-c} x].
\end{equation}
Using only $n-1$ of the above $n$ relations for each fixed $a,b$, this
can be interpreted as eliminating the classes $[iz_{a,b}]$ for all $i
\in Q^0_0$ except any fixed vertex $i_0 \in Q^0_0$.  To add the last
relation in, we need only consider the sum of all $n$ relations, which
together with \eqref{sc2} (allowing us to divide by $\ds
\frac{b}{\gcd(\frac{a-b}{n},b)}$) gives \eqref{wabdefn1} for $a > b$
(the coefficients of $Q^0_0$ disappear as we are summing over all
translations around the cycle: every vertex of $Q^0_0$ becomes
$\1_{Q^0_0}$).  The same argument works for $a < b$, and so the proof
from \S \ref{a0pfs} shows that $\Lambda_Q$ has no torsion in these
cases. Note that the $W_{a,b}$ remain nonzero in this case since $Q
\neq Q^0$---even if there is only an additional arrow at one particular
vertex $i \in Q^0_0$, there are terms in $W_{a,b}$ where $r$ only
appears adjacent to this vertex $i$.

In the case $a=b=m$, again we see that the quotient $V/W$ is the same
as eliminating $[i z_{m,m}]$ for all $i \neq i_0$, and considering
only the relation $W_{a,a}$ from \eqref{wabdefn1}--\eqref{wabdefn3}.
For the same reasons as in \S \ref{a0pfs}, it follows that $V/W$ has
torsion $\Z/p$ in exactly those bidegrees $(p^\ell,p^\ell)$ for $p$
prime and $\ell \geq 1$, and the torsion is generated by the class
$[r^{p^\ell}]/p \in P_\dq/[P_\dq,P_\dq]$.  This implies parts
(v)--(vii) of Theorem \ref{ant}, which were all that remained to be proved.

\subsection{Hilbert series and \eqref{egid} in the $\tilde A_n$ case} 
\label{ean} In this
section, we verify the Hilbert series of $\Lambda_{\tilde A_n}$ using
our bases, and give a direct proof of the curious identity
\eqref{egid} from \cite{EG} in this case.  We provide this since it is
an easy consequence of Theorem \ref{ant} (which we just proved), and
gives a different proof from what is found elsewhere.

First, from Theorem \ref{ant} we deduce
\begin{prop}
The Hilbert series of $\Lambda_{\tilde A_{n-1}}$ and $i \Pi_{\tilde A_{n-1}} i$, for
any vertex $i$, are given by
\begin{equation}
  h(i \Pi_{\tilde A_{n-1}} i;t) = h(\Lambda_{\tilde A_{n-1}};t) = \frac{1 + t^n}{(1-t^2)(1-t^n)} = \frac{1-t^{2n}}{(1-t^2)(1-t^n)^2}.
\end{equation}
\end{prop}
It immediately follows that one has the formula
\begin{equation} \label{pfan}
h(\Lambda_{\tilde A_{n-1}};t)(1-t^2) = 1 + \frac{2t^n}{1-t^n},
\end{equation}
which we can use to verify the following formula
for Hilbert series from \cite{EG} (using \eqref{gnccif}):
\begin{gather} \label{egedf}
h(\O(\Pi);t) = \prod_{m \geq 1} \frac{1}{(1-t^m)^{a_m}}, \\ \label{egedf2}
\prod_{m \geq 1} \frac{1}{(1-t^m)^{a_m - a_{m-2}}} = \frac{1}{1-t^2} \cdot \prod_{m \geq 1} \frac{1}{\det (1 - t^m \cdot C + t^{2m} \cdot 1)},
\end{gather}
where $C$ is the adjacency matrix of $\overline{\tilde Q}$, and $a_{-1}=a_0=0$.  The element
$1 - t \cdot C + t^2 \cdot \1$ is $1/t$ times the so-called
``$t$-analogue of the Cartan matrix'', $(1 + \frac{1}{t}) \cdot \1 - C$.  For $\tilde A_{n-1}$,
one has
\begin{equation} \label{detcan}
\det(1 - t \cdot C + t^2 \cdot \1) = (1-t^n)^2.
\end{equation}

To verify \eqref{egedf2},
set $h((\Lambda_{\tilde A_{n-1}})_+;t) = \sum_m a_m t^m$; one then has from \eqref{pfan}
\begin{equation}
a_m - a_{m-2} = 2 [n \mid m], \quad m \geq 3,
\end{equation}
which implies the desired identity.

Since $h(i_0 \Pi_{\tilde A_{n-1}} i_0;t)=h(\Lambda_{\tilde A_{n-1}};t)$,
 by \eqref{egfla}, the $a_m$'s above satisfy
\begin{equation}
1 + \sum a_m t^m = \bigl(\frac{1}{1 - t \cdot C + t^2 \cdot \1}\bigr)_{i_0i_0}.
\end{equation}
That
is, $a_m = \phi_{m}(C)_{i_0i_0}$, where $\phi_m$ is the $m$-th Chebyshev
polynomial of the second type.  Since
$\phi_m - \phi_{m-2} = \varphi_m$, a Chebyshev polynomial of the first
type, our work above explicitly verifies the identity \eqref{egid} from \cite{EG}.

So, from our point of view, this identity is the fact that $\Lambda_{\tilde A_{n-1}} \cong i_0 \Pi_{\tilde A_{n-1}} i_0$, together with the similarity between the
 identity \eqref{pfan} and the formula for the determinant of
the $t$-analogue of the Cartan matrix, \eqref{detcan}.

\subsection{Poisson structure on $i_0 \Pi_Q i_0$ 
for $Q = \tilde A_{n-1}$} \label{lans} For $Q$
extended Dynkin, there is an injection $i_0 \Pi_Q i_0 \into
\Lambda_Q$, whose cokernel is isomorphic to the torsion of
$\Lambda_Q$.  Hence, the necklace Lie bracket on $\Lambda_Q$
induces a bracket on $i_0 \Pi_Q i_0$, obtained by taking the image in
$\Lambda_Q$, applying the bracket, and then using the isomorphism
$i_0 \Pi_Q i_0 \cong \Lambda_Q / \text{torsion}$. It is clear that
the result is a Lie bracket. Moreover, $i_0 \Pi_Q i_0$ is commutative.
Then, it follows from \eqref{ncleibeqn} that $i_0 \Pi_Q i_0$ is actually a Poisson
algebra, i.e., that the Leibniz identity is satisfied.



In the case $Q = \tilde A_{n-1}$, by Theorem \ref{dedz}, $\Lambda_Q$
is torsion-free, so in this case $i_0 \Pi_Q i_0 \rightarrow \Lambda_Q$
is an isomorphism.  Here we describe explicitly the resulting Poisson
structure.

The simplest way to understand $\Lambda_{\tilde A_{n-1}}$ is in terms
of the basis $[i_0 x^a y^b]$, where $x$ denotes moving clockwise
one arrow, and $y$ denotes moving counterclockwise, and $i_0$ is a
fixed vertex.  One requires that $n \mid (a-b)$, and $a, b \in \N$. 
To compute the bracket, we first compute the bracket in terms of the rational
basis $[\1 x^a y^b]$, where
$\1$ is the identity (the sum of all vertices); 
since there is no torsion in $\Lambda_{\tilde A_{n-1}}$, this suffices.
Then, it immediately follows that one can compute the bracket (and
cobracket) by summing over ways to pair letters $x,y$ that correspond
to opposite arrows.  One easily computes that $\delta$ is zero.  The
bracket is then
\begin{equation}
[ [i_0 x^a y^b], [i_0 x^c y^d] ] = i_0 \frac{ad-bc}{n} x^{a+c-1} y^{b+d-1}.
\end{equation}
In other words, the Poisson structure on $i_0 \Pi_Q i_0$ is given by
$\{i_0 z_{a,b}, i_0 z_{c,d}\} = \frac{ad-bc}{n} i_0 z_{a+c-1,b+d-1}$.

In terms of the isomorphism
$\Lambda_{\tilde A_{n-1}} \iso \Lambda_{\tilde A_{n-1}} \o 1 \subset
\Lambda_{\tilde A_{n-1}} \o \C \iso \C[x,y]^{\Z/n}$,
one has integral basis elements $x^n, y^n, xy$, with Poisson bracket
which is $\frac{1}{n}$ times the usual Poisson bracket on $\C[x,y]$,
restricted to $\C[x,y]^{\Z/n}$.

Summarizing, we can give an explicit presentation of $i_0 \Pi_Q i_0$ as a
graded Poisson algebra over $\Z$ as follows:
\begin{prop} \label{apbp} The following is an explicit presentation
of $i_0 \Pi_Q i_0$ for $Q = \tilde A_{n-1}$:
\begin{gather}
X := [i_0 x^n], \quad Y := [i_0 y^n], \quad Z := [i_0 xy], \\
i_0 \Pi_Q i_0 \liso \Z[X,Y,Z] / \ldp XY-Z^n \rdp,  \label{anap}\\
\{X,Y\} = n Z^{n-1}, \quad \{X,Z\} = X, \quad \{X, Y\} = - Y, \\
|X| = n, \quad |Y| = n, \quad |Z| = 2.
\end{gather}
\end{prop}

We remark that $i_0 \Pi_Q i_0$ can be thought of as ``$\Z[x,y]^\Gamma$'' = 
$\C[x,y]^\Gamma \cap \Z[x,y]$ (using $\frac{1}{n}$ times the standard bracket), 
and this also allows one to make sense of ``$\F_p[x,y]^\Gamma$'' 
for primes $p \mid n$. The same comment will apply for other extended Dynkin
quivers $Q$, when we explicitly compute $i_0 \Pi_Q i_0$ as a Poisson algebra
in those cases as well.

\section{Quivers containing $\tilde D_n$}
\label{dns}
\subsection{Bases of $\Pi_Q$ for type $D$ quivers and refinement
of Theorem \ref{mt}}
Here we describe bases of $\Pi_Q$ for extended Dynkin quivers of type
$D$ and quivers properly containing them, leading to Theorem \ref{dnt}
which implies Theorem \ref{mt} in the case of quivers containing a
$\tilde D_n$ quiver, i.e., containing either multiple nodes (vertices
of valence $\geq 3$) or a node of valence $4$.  Together with the type
$\tilde A$ case in the previous section, this proves Theorem \ref{mt}
in all cases except star-shaped quivers with three branches.

We will need the following notation.  Suppose that $Q^0 = \tilde D_n$
is drawn and oriented as follows:

\begin{figure}[hbt]
\begin{center}
\includegraphics{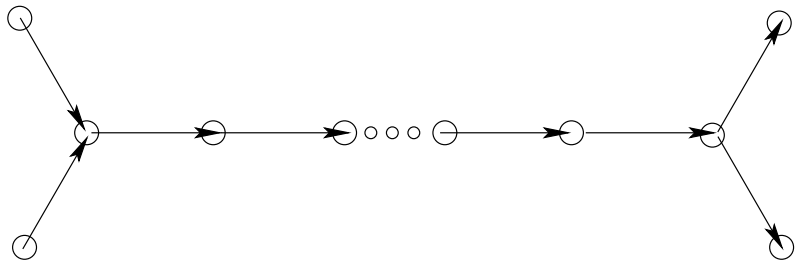}
\caption{$Q^0 = \tilde D_{n}$ with the preferred orientation}
\label{dnf}
\end{center}
\end{figure}
As in the figure, we let $i_{LU}, i_{LD}, i_{RU}, i_{RD}$ denote
the four external vertices ($L, R, U, D$ stand for ``left, right, up, down'', respectively).  Furthermore, we set $i_L := i_{LU} + i_{LD}$ and $i_R :=
i_{RU} + i_{RD}$, the sum of the leftmost and rightmost external vertices. We then define
\begin{equation}
\ino := \sum_{i \text{ internal}} i,
\end{equation}
the sum of internal vertices.

Next, we define
\begin{equation}
R := \sum_{a \in \dzqo \mid a \text{ is oriented rightward}} a, \quad
L := \sum_{a \in \dzqo \mid a \text{ is oriented leftward}} a.
\end{equation}
Here, ``leftward'' means an arrow which is either oriented from right to left in the
diagram, or one whose terminal endpoint is one of the two leftmost vertices, $i_{LU}$ and $i_{LD}$;
similarly we define ``rightward'' to be all other arrows.
One can think of $R$ and $L$ 
 as analogous to $x$ and $y$ in
the $\tilde A_{n-1}$ case (although the behavior is not exactly the same).
The choice
of orientation above is not essential, since none of the results
(essentially) depend on it; in particular, the structures of
$\Pi_{Q^0}$ and $\Pi_Q$ remain the same.

We will consider a slightly larger set of paths, called ``generalized
paths,'' $GP \subset P_{\dzq}$ (we will also use the term for the image
of such paths in $\Pi_{Q^0}$ or $\Pi_{Q}$) 
as follows: Elements of $GP_m$ are products of the form
\begin{equation} \label{gpfla}
i_1 X_1 i_2 X_2 \cdots i_m X_m i_{m+1}, \quad i_j \in Q^0_0 \cup \{i_L, i_R\}, \quad X_j \in \{L,R\}.
\end{equation}
In other words, we allow products of not merely arrows, but also the
sums of two external arrows which are pointed in the same left-right
direction and share either the same initial vertex or share the same
terminal vertex (note that the shared vertex is necessarily an internal
vertex).  Generalized paths have well-defined ``generalized
endpoints'', which means either a vertex of $Q^0_0$ or one of the elements
$i_{L}, i_R$ (the sum of the two left or right endpoints).  In
\eqref{gpfla} these are given by $i_1$ and $i_{m+1}$.

It will be convenient to define ``$F$,'' for ``forward,'' as follows: if
$Y$ is any generalized path, then 
\begin{equation} \label{fdefn}
Y F := \begin{cases} Y R, & \text{if $Y = Y i_L$ or $Y = Y' R i$ where $i \in Q^0_0$ is internal and $Y' \in GP$}, \\
Y L, & \text{if $Y = Y i_R$ or $Y = Y' L i$ where $i \in Q^0_0$ is internal and
$Y' \in GP$}.
\end{cases}
\end{equation}

\`A priori, in order for symbols $F$ to become elements of $P_\dzq$,
they must be multiplied on the left by a generalized path $Y$. It is
easy to check, however, that one obtains a well-defined product
$(P_\dzq)_+ \times \{F\} \rightarrow P_\dzq$ by linearity (i.e., it is
enough to multiply on the left by any positively-graded element of the
path algebra).

We
call an arrow ``upward'' if, in Figure \ref{dnf}, it goes diagonally up
and to the left or diagonally up and to the right, i.e., it has either
initiates at $i_{LD}$ or $i_{RD}$, or else terminates at
$i_{LU}$ or $i_{RU}$.  Similarly define ``downward'' to be an arrow
going diagonally down and to the left or down and to the right. 

Then,
we define $L_U$ and $L_D$ by:
\begin{equation}
L_U := \sum_{a \in \dq_1 \mid \text{$a$ is leftward and not downward}} a,
 \quad L_D := \sum_{a \in \dq_1 \mid \text{$a$ is leftward and not upward}} a,
\end{equation}
where the notation is chosen because, if $P$ is a path, then $P L_U$,
if nonzero, is a path of length one more which follows $P$ by a
leftward arrow which is not downward, i.e., either an arrow diagonally
up and to the left, or a leftward arrow whose incident vertices are
both internal; the similar statement holds for $L_D$.

We similarly define $R_U$ and $R_D$ to be the sum of all rightward
arrows which are not downward and the sum of all rightward arrows which
are not upward, respectively.  Finally, define $F_U$ by \eqref{fdefn},
replacing $L$ and $R$ by $L_U$ and $R_U$, and similarly define $F_D$.

We now define our proposed basis elements,
for any $c, C \geq 0$, and any choice of initial vertex $i$, terminal
vertex $j$, and initial direction $R$ or $L$:
\begin{gather}
z_{c,C,R,i,j} := \begin{cases} 
                          i (RL)^c j, & \text{if $C=0$,} \\
                          i (RL)^c R j, & \text{if $C=1$,} \\
                          i (RL)^c R_U F_U^{C-2} F j, & \text{if $C \geq 2$}.
\end{cases} \\
z_{c,C,L,i,j} := \begin{cases} 
                          i (LR)^c j, & \text{if $C=0$,} \\
                          i (LR)^c L j, & \text{if $C=1$,} \\
                          i (LR)^c L_U F_U^{C-2} F j, & \text{if $C \geq 2$}.
\end{cases}
\end{gather}
\begin{rem}
  The indices $c, C$ in the $\tilde D_n$ case are analogous to the
  quantities $\min(a,b)$ and $|a-b|$ in the $\tilde A_{n-1}$ case:
  $c$ gives the number of short cycles ($RL$ or $LR$) and
  $\frac{1}{2(n-2)} \cdot C$ gives the number of long cycles (a power of
  $f_U$ or $f_D$). We cannot keep the same notation in both cases
  because in the $\tilde A_{n-1}$ case, the winding number is a
  meaningful quantity (which is $\frac{1}{n} \cdot (a-b)$), whereas in
  the $\tilde D_n$ case, the only meaningful quantity is the number of
  short and long cycles: the starting direction $R$ or $L$ is only
  meaningful if beginning at an internal vertex, and the meaning is
  lost modulo commutators (or when passing to the center).
\end{rem}
The main result of this section is:
\begin{thm} Let $Q^0 = \tilde D_n$. We use the notation of Proposition \ref{mip2}. 
\label{dnt}
\begin{enumerate}
\item[(i)] A basis for $i \Pi_{Q^0} j$ for any $i,j \in Q^0_0$ consists of
all nonzero elements having the form 
\begin{enumerate}
\item[(a)] $z_{c,C,R,i,j}$, for $c, C \geq 0$ and any $i,j$, or 
\item[(b)] $z_{c,C,L,i,j}$, for $c,C \geq 0, i,j \in Q^0_0$ such that
 either $i \in \{i_{RU}, i_{RD}\}$ or $C \geq 1$.
\end{enumerate}
In particular, for $i_0 := i_{LU}$, a basis of $i_0 \Pi_{Q^0} i_0$
  is given by $z_{c,C,R,i_0,i_0}$ for $c, C \geq 0$.
\item[(ii)] For any quiver $Q \supsetneq Q^0$, $W = W_1 \oplus W_2$,
  where $W_1 \subset W$ projects isomorphically under $V \onto
  \Pi_{Q^0} \onto \Pi_{Q^0} / \langle z_{c,C,R,i_{LU},i_{LU}},
  z_{c,0,L,i_{RU},i_{RU}} \rangle_{c,C\geq 0}$ onto the (saturated)
    submodule
\begin{equation}
\bigoplus_{i \neq j} i \Pi_{Q^0} j \oplus \bigoplus_{i \in Q^0_0 \mid i \notin \{i_{LU},i_{RU}\}} 
i \Pi_{Q^0} i \oplus \langle z_{c,C,L,i_{RU},i_{RU}}  \rangle_{c \geq 0, C \geq 1},
\end{equation}
and $W_2$ has a basis of classes $W_{c,C}$, for $C = 2(n-2) \cdot C'$
and $c, C' \geq 0$, having the form
\begin{gather}
W_{c,C} := \sum_{c_\cdot, C_\cdot, i_\cdot, X_\cdot \mid (*)} 
\frac{\gcd(c,C')}{\rep(c_\cdot, C_\cdot)}
[\prod_{m=1}^{|\{c_\cdot\}|} r'z_{c_m, C_m, X_m, i_m, i_{m+1}}]   \quad (C > 0), 
\intertext{where $(*)$ is the condition that $X_m$ always begin going forward:}
\prod_{m'=1}^{m-1} z_{0,C_{m'},X_{m'},i_{m'},i_{m'+1}} F = 
\prod_{m'=1}^{m-1} z_{0,C_{m'},X_{m'},i_{m'},i_{m'+1}} X_m, \text{ and} \\
c = |\{c_\cdot\}| + \sum_m c_m, \quad C = \sum_m C_m; \\
\intertext{finally,}
W_{c,0} := (RL)^c - (LR)^c = i_L (RL)^c - i_R (LR)^c + \ino ((RL)^c - (RL + r')^c) \ino.\label{wc0rel} 
\end{gather}
\item[(iii)] The integral span of $W_{c,C}$ for $C > 0$ or $c \neq
  p^\ell$ for any prime $p$ and $\ell > 1$ is saturated, and the order
  of $W_{p^\ell,0}$ is $p$ (modulo the other $W_{c,C}$'s or
  otherwise).
\item[(iv)] The image of $\frac{1}{p} W_{p^\ell,0}$ is $r^{(p^\ell)}$
  and these classes generate the torsion of $\Lambda_Q$ (ranging
  over all $p$ and $\ell$). 
\end{enumerate}
\end{thm}
  Although we could have
included a description of $\Lambda_{Q^0}$ above, we relegate this to
Theorem \ref{dedz}.

Note that, as in the $A_n$ case, one way to obtain bases for
$\Pi_{D_n}$ is from the above basis.  One can set
$i_0 := i_{LD}$ and $D_n := \tilde D_n \setminus \{i_0\}$, and use
those basis elements $z_{c,C,X,i,j}$ which are nonzero under the
quotient $\Pi_{\tilde D_n} \onto \Pi_{D_n}$, and which do not pass
through $i_{LU}$ (but may begin or end at $i_{LU}$).

\subsection{Proof of Theorem \ref{dnt}.(i)}
Define a filtration on $\Pi_{Q^0}$ by powers of the two-sided ideal
$\ldp \ino RL \ino, \ino LR \ino \rdp$.  We will show that $\gr \Pi_{Q^0}$ is a
free $\Z$-module with the given basis (more precisely, its image in
$\gr \Pi_{Q^0}$), which shows that the same is true for $\Pi_{Q^0}$.

We begin by showing that the given elements integrally span $\Pi_{Q^0}$
and $\gr \Pi_{Q^0}$.  We first show that $\Pi_{Q^0}$ and $\gr \Pi_{Q^0}$ are
integrally spanned by a larger set: elements of the form $i (RL)^k
Y$ or $i (LR)^k Y$ where $Y$ is a path that satisfies the following
property: The paths change from going to the left to right or
vice-versa only at the four endpoints of $\tilde D_{n}$.  This differs
from the proposed basis only by allowing $F_U$'s to change to $F_D$'s,
and replacing the first $F_U$ or $F$ by any of $L_U, L_D, R_U$,or
$R_D$.

First, $\Pi_{Q^0}$, as well as $\gr \Pi_{Q^0}$, is obviously generated by
elements $i (RL)^k Y$ or $i (LR)^k Y$ where $Y$ is a generalized path.
We can assume that $Y$ does not begin with $LR$, $RL$, or any
expression equal to one of these, since otherwise we could absorb this
into the product $(LR)^k$ or $(RL)^k$: if $Y$ begins at an internal
vertex $i$, then $i LR = i RL$, whereas if $Y$ begins at an external
vertex, only one of $i LR$ and $i RL$ is nonzero.  Next, we show that,
if $Y$ includes a change of left-right direction at a vertex other
than an external one, then it is equivalent to an element that makes such a
change of direction strictly earlier, up to elements of $\ldp \ino RL
\ino, \ino LR \ino\rdp^{k+1}$.  Since we assume the change of direction
cannot happen at the beginning, this will prove the desired result
(obviously, there are only finite number of possible expressions
$(LR)^k Y, (RL)^k Y$ of a fixed length to consider).

Suppose that the first change of direction at an internal vertex (say $j$)
is of the form $i L j R$ where $i$ is also an internal vertex.
Then, it must follow $L$, so that $L i L j R = L i LR = L i RL$,
giving the desired element which changes direction at an internal
vertex earlier.  Similarly, we can handle $i R j L$ where $i,j$ are
internal.

Next, suppose that the first change of direction at an internal vertex
(say $j$) is of the form $i L j R$ where $i$ is external.  Then it
must be preceded by $R_X$ for the appropriate $X \in \{U,D\}$, so we
get $R_X L j R = R_X L R_\oX = RL R_\oX$ for $X \neq \oX \in \{U,D\}$.
If these three arrows are the first of $Y$, then we can absorb an
additional $RL$ (or $LR$) into the initial power of $RL$ or $LR$. In this case,
the element $i(RL)^kY$ or $i(LR)^kY$ lies in 
$\ldp \ino RL \ino, \ino LR \ino\rdp^{k+1}$, so we can discard it.
Otherwise, these three arrows must be preceded by $R$,
since we assumed this was the first change of direction.  Then, we get
$R i R L R_\oX = R i LR R_\oX$, again yielding an expression that
changes direction at an internal vertex earlier.

This proves the desired claim.  Now, to show that $\Pi_{Q^0}$, as well as
$\gr \Pi_{Q^0}$, is integrally spanned by the elements stated in the
theorem, we make two observations. First of all, if we have a nonzero
expression $i (RL)^k Y$ for $k \geq 1$, where $Y$ does not begin with
$R_U$ or $R_D$, then $i$ must be an internal vertex and $Y$ must begin
with $L_U$ or $L_D$. In this case, because $i (LR)^k = i (RL)^k$, this
expression is already equal to $i(LR)^k Y$, so we don't need $i(RL)^k
Y$ to integrally span $\Pi_{Q^0}$.

Next, take an expression of the form $i (RL)^k Y$ where $k \geq 1$ and
$Y$ begins with $R_U$ or $R_D$. If $Y$ includes $L_D$ anywhere except
at the very end, then it precedes $R$, and we may freely replace $L_D
R$ with $LR - L_U R$.  Since $Y$ does not begin with $L_D$, the
occurrence $L_D R$ in $Y$ is preceded by $L$.  Thus, we replace the
resulting $LL_DR$ in $Y$ by $LLR - L L_UR = LRL - LL_U R$.  The
resulting $LRL$ term involves changing direction at an internal
vertex, which we showed above must be in $\ldp \ino RL \ino, \ino LR
\ino\rdp^{k+1}$. In other words, in $\gr \Pi_{Q^0}$, one can replace the
$L_D$ in $Y$ by $-L_U$, and obtain the same element.  Hence, the
elements stated in the theorem integrally span $\gr \Pi_{Q^0}$, as well
as $\Pi_{Q^0}$.



It remains to show linear independence of the given elements over $\Z$
(i.e., that they form a basis for the free $\Z$-module $\Pi_{Q^0}$).  We
do this by mirroring the above, but using the Diamond Lemma and a bit
more careful analysis.  For any generalized path $Y \in GP$, we can
construct an alternating sequence
\begin{equation}
E(X) := (t_1, t_2, t_3, \ldots, t_m)_{t_\ell \in \{L,R\}},
\end{equation} which we call ``the ends
of $X$'', as follows: Start with the empty sequence, unless
the initial vertex of $X$ is an external one, in which case we start with
the side $L$ for left or $R$ for right.
Every time we hit the left end (a left external vertex or the
superposition of both) we add an $L$ to the sequence $E(X)$, unless $E(X)$ is
already nonempty with last term equal to $L$.  Every time we hit
the right end, we add an $R$, unless the sequence is nonempty with
last term $R$.  That is, the sequence records the order in which the
path hits left and right endpoints, throwing out multiple hits of one
side before next hitting the opposite side.  Put another way, it
records which is the first side the path reaches, and the number of
times it alternates from one side to another.

Now, for any
$X \in GP \cap (P_\dzq)_k$ of length $k$, with
$E(X) = (t_1, t_2, \ldots, t_m)$, we can write
$X = X_1 i_{t_1} X_2 i_{t_2} \ldots i_{t_m} X_{m+1}$ in the unique way such that
$X_1 \in P_\dzq$ has minimal possible length, and for this value of $X_1$,
$X_2$ has minimal possible length, etc.

Finally, let us consider the places $\ell \in \{1,2,\ldots,k\}$ where
$X$ is ``not going in the correct direction.''  By ``not the correct
direction at $\ell$'', we mean that the path is headed away from the
next endpoint $i_{t_{m}}$ (=towards the previous endpoint $i_{t_{m-1}}$)
appearing in the path.  Precisely, suppose that
$X = a_1 a_2 \cdots a_k$ and $X$ hits the endpoint $i_{t_m}$, at the
time corresponding to when we add $t_m$ to the sequence,
between $a_{\ell_i-1}$ and $a_{\ell_i}$ 
(i.e., if the last endpoint $t_m$ is hit at the
very end, set $\ell_m := k+1$).  Then, for all
$\ell_i \leq \ell < \ell_{i+1}$, we say that $\ell$ is a place where
$X$ heads in the ``wrong direction'' if $a_\ell \neq t_{i+1}$.  Now,
for any $X \in GP \cap (P_\dzq)_k$ of the form $X=iXj$ for some
$i,j \in Q^0_0$, let $WD(X) \subset \{1,2,\ldots,k\}$ be the subset of
places where $X$ goes in the wrong direction.  One easily sees that
the size of $WD(X)$ is the difference between $k$ and the shortest path
with the same endpoints as $X$ (this makes sense since we assumed $X$
had definite endpoints). 

Now, we may define $X \prec Y$ if either $|WD(X)| > |WD(Y)|$, or else
$|WD(X)| = |WD(Y)|$ and there is an order-preserving bijection $\phi:
WD(X) \rightarrow WD(Y)$ such that $\phi(c) \leq c$ for all $c \in
C(X)$.  That is, the places where $X$ goes in the ``wrong direction''
occur strictly before the corresponding places in $Y$.

Now, we finally have the ordering we need such that the relations on
$\langle GP \cap (P_\dzq)_k \rangle$ modulo which one gets $(\Pi_{Q^0})_k$
are confluent and give the desired elements as the normal form (one
reduces as in the first part of this proof).  So, the desired elements
are linearly independent and form a basis of $\Pi_{Q^0}$. Also, since
their images integrally span $\gr \Pi_{Q^0}$, they form a basis of $\gr
\Pi_{Q^0}$ as well, which is therefore also a free $\Z$-module.

Similarly to the $\tilde A_{n-1}$ case, it is not difficult to deduce
from the above procedure that $\Pi_{Q^0}$ is an NCCI (that one has a
unique reduction of $P_\dzq$ to words in the above basis and the
relations $i r i$). Also, it is clear from the above proof that for any
of the basis elements, we can feel free to replace any of the $F_U$'s
in the $z_{c,C,X,i,j}$ by $F_D$'s, and we will still be left with a
basis.

\subsection{Proof of the remainder of Theorem \ref{dnt}}
\label{dnpfs} 
Suppose $Q$ is any quiver with a proper subquiver $Q^0 \subsetneq Q$
with $\overline{Q^0} \cong \overline{\tilde D_{n}}$.
Define $r_{Q^0_0} :=
\sum_{i \in Q^0_0} i r i$.  Fix a forest $G \subset \overline{Q}_1
  \setminus \overline{Q_1^0}$ as in Proposition \ref{bpp}.

  As in the $\tilde A$ case, and using the same notation $V$, $W$, and
  $A$, it suffices to compute $\bar W := \bigoplus_{i \in Q^0_0} i W i
  \subset V$, which is integrally spanned by $\eta([X,e])$ for $X$
  listed in Theorem \ref{dnt}.(i) and $a \in \overline{Q_1^0}$ such
  that $Xa \in \bar \Pi_Q := \sum_{i \in Q_0} i \Pi_Q i$
  (cf.~Proposition \ref{mip2}.(iii)).  We can instead let $a$ be one
  of $L_U, L_D, R_U,$ or $R_D$.

First, we need to eliminate some of the commutators $\eta([X,a])$.  We
prefer to eliminate those $X = z_{c,C,X,i,j}$ with smaller $c$, so as
to maximize the $c$ of the remaining commutators.  Essentially, this
means continuing to use the filtration on $A$ by powers of the ideal
$(LR, RL)$.

First note the basic
\begin{lemma} \label{udsl}
If $X$ is obtained from
$z_{c,C,X,i,j}$ by changing $\alpha$ instances of $R_U L$ to $R_D L$
or $L_U R$ to $L_D R$, then $X \equiv (-1)^\alpha
z_{c,C,X,i,j}$ modulo $\langle
z_{c',C',X,i,j}\rangle_{c' > c}$.
\end{lemma}
\begin{proof}
This follows from the fact that $R_U L + R_D L = RL$, using the ideas
in the proof of Theorem \ref{dnt}.(i) (in the last subsection).
\end{proof}

We use this to write some relations which are deduced similarly to
\eqref{sc1}, \eqref{sc2}.  These will allow us to express certain
commutators in terms of commutators involving basis elements living in
a smaller power of $\ldp \ino RL \ino, \ino LR \ino \rdp$.

\begin{ntn}\label{mhpn} We say that an equality holds 
  ``modulo commutators with higher powers of \\
  $\ldp \ino RL \ino, \ino LR \ino \rdp$'' if the equation is true up to
  $\eta$ of commutators with elements $z_{c',C',X',i',j'}$, where $c'$
  is greater than all the indices $c$ which appear in the equation.
\end{ntn}

Take $C$ such that $2(n-2) \mid C$ and $C > 0$.  This is equivalent to
the condition that a basis element of the form $z_{c,C,X,i,i}$ has
either (1) $i$ is internal and $z_{c,C,X,i,i} = z_{c,C-1,X,i,j} X$ for
some $j$ (i.e., it ends in $X$); or (2) $i$ is external.  Pick any
choice of $i, X$ such that $z_{c,C,X,i,i}$ is a basis element; let $c
> 0$. Let $\tilde X$ be defined to be $X$ when $i$ is internal (first
case), and $\oX$ otherwise (second case), where here and below $\oX
\in \{L,R\}$ denotes the opposite direction from $X \in \{L,R\}$. That
is, $\tilde X$ is the last direction in $z_{c,C,X,i,i}$. One then
computes, using \eqref{sc1}, Lemma \ref{udsl}, and Notation
\ref{mhpn}, for $c \geq 1$,
\begin{multline} \label{zcomm1}
\sum_j c\eta([z_{c-1,C+1,X,i,j},j\overline{X}i]) + c
\eta([z_{c,C-1,\tilde X,j,i}, iXj]) \\ + \frac{C}{2(n-2)} \sum_{i',j'
\in Q^0_0 \setminus\{i_{rd}, i_{\ell d}\}, Y \in \{L,R\}}
(-1)^{c \delta_{\oX,Y}} \eta([z_{c,C-1,Y,i',j'},j'(L+R)i']) = 0 \\
\text{ modulo commutators with higher powers of $\ldp \ino RL \ino, \ino LR \ino \rdp$},
\end{multline}
which shows that we can eliminate the relations $\sum_j
\eta([z_{c-1,C+1,X,i,j},j\overline{X}i])$. In terms of commutators of
basis elements with arrows, we do the following: In the case that $i$
is not adjacent to external vertices by an arrow moving in the $X$
direction (beginning with $i$), this sum is only over a single value
of $j$, so we eliminate the relation
$\eta([z_{c-1,C+1,X,i,j},j\overline{X}i])$.  Otherwise, the first sum
is over two possible values of $j$, then we can choose to eliminate
the relation $\eta([z_{c-1,C+1,X,i,i_{XD}},i_{XD} \overline{X} i])$.
Note in the above that the sums over all $i', j'$ are effectively only
over adjacent pairs $i', j'$, and the $j'(L+R)i'$ is shorthand for the
unique arrow in $\overline{Q^0}$ from $j$ to $i$.

Also, the relation $\eta([z_{c-1,C+1,X,i,i_{XU}},i_{XU} \overline{X}
i])$ we left above can be interpreted simply as expressing
$z_{c-1,C+2,X,i,i}$ as $z_{c,C,X,i_{XU},i_{XU}}$ plus some multiple of
$r' = r_{Q^0_0} - \sum_{a \in Q_1^0} (a a^* - a^* a)$ in the quotient
$V/W$.  So we can also eliminate this relation if we also eliminate
the generator $z_{c-1,C+2,X,i,i}$ from $V$.  Actually, this paragraph
is still true if $C=0$, so we can also eliminate the relation
$\eta([z_{c-1,1,X,i,i_{XU}},i_{XU} \overline{X} i])$ and the generator
$z_{c-1,2,X,i,i}$ (for $i$ adjacent to external vertices by arrows of the form
$i X_U j$ or $i X_D j$).

Next, consider a basis element $z_{c,C,X,i,i}$ such that
$2(n-2) \nmid C$ and $2(n-2) \nmid (C-2)$ (we dealt with the other
cases in the last paragraph).  In particular, $i$ is internal and not
adjacent to external vertices by moving in the $X$ direction.
Let $g=g(X,i)$ be defined to be
the distance from $i$ to the $X$-end. One has
$g = \frac{C}{2} - (n-2) \lfloor \frac{C}{2(n-2)} \rfloor$ since the
$F_U^C$-portion of the element $z_{c,C,X,i,i}$ goes forward $g$ units,
cycles   $\lfloor \frac{C}{2(n-2)}\rfloor$ times around the long distance
of $\tilde D_{n}$, and then moves $g$ units forward back to $i$.
We will show that (1) 
$\eta([z_{c,C-1,X,i,j},j\overline{X}i]) \equiv -
\eta([z_{c,C-1,\oX,j,i},iXj])$
modulo commutators with higher powers (in particular, modulo terms that begin
and end closer to the external vertices on the $X$-side than $i$), and
that (2) either of these (sums of) commutators can be taken to express
$z_{c,C,X,i,i}$ in terms of $z_{c+m,C-2m,X,j',j'}$ for $m > 0$ (and
$g(X,j') < g(X,i)$); eventually this will reduce us to basis elements beginning
and ending at an external vertex.

(1) follows from the identity
\begin{equation} \label{zcomm2}
 [z_{c,C-1,X,i,j}, j\overline{X}i] + [j(\oX X)^c z_{0,C-1,X,j,i},i X j] 
=  [j \overline{X} X j, z_{c,C-2,X,j,j}], 
\end{equation}
and (2) follows because $\eta([z_{c,C-1,X,i,j},j\overline{X}i]) -
(z_{c,C,X,i,i} - z_{c+1,C-2,X,j,j}) \in [r' A]$.

Similarly to \eqref{zcomm2}, for $C > 2$ we can also show that
$\eta([z_{c,C-1,\oX,i,i_{XD}}, i_{XD} \overline{X} i])$ is in the 
integral span
of other commutators (note that $2(n-2) \mid C$ for this to be
nonzero):
\begin{multline}
[i (X \oX)^c z_{0,C-1,\oX,i,i_{XD}}, i_{XD} \overline{X} i] + [i (X \oX)^c z_{0,C-1,\oX,i,i_{XU}}, i_{XU} \overline{X} i] + [z_{c,C-1,\oX,i_{XD},i},i X i_{XD}] \\ + [z_{c,C-1,\oX,i_{XU},i},i X i_{XU}] = [(X \oX)^c z_{0,C-2,\oX,i,i}, X \oX],
\end{multline}
which shows that
\begin{multline}
\eta([z_{c,C-1,\oX,i,i_{XD}},i_{XD} \overline{X} i]) +
\eta([z_{c,C-1,\oX,i,i_{XU}},i_{XU} \overline{X} i ]) +
\eta([z_{c,C-1,\oX,i_{XD},i},i X i_{XD}]) \\  +
\eta([z_{c,C-1,\oX,i_{XU},i},i X i_{XU}]) \in
\eta([\ldp \ino RL \ino, \ino LR \ino \rdp^{c+1},\Pi_Q]),
\end{multline}
so we can indeed throw out any one of these commutators: we choose
$\eta([z_{c,C-1,\oX,i,i_{XD}},i_{XD} \overline{X} i])$.

The only other commutators of the form $\eta([X,e])$, with $X$ as in
Theorem \ref{dnt}.(i) such that $Xe$ is nonzero and in $\bar \Pi_Q$,
that we have not yet mentioned, are those of the form
$\eta([z_{c,1,X,i,j},j\overline{X} i])$ where $j$ is an internal
vertex.  Here we can make use of the identity, similar to the above:
if $i$ and $j$ are internal, then
\begin{equation}
\eta([z_{c,1,X,i,j}, j \overline{X} i]) = \eta([z_{c,1,\oX,j,i}, i X j]),
\end{equation}
and otherwise,
\begin{equation}
\sum_{Y \in \{U,D\}} \eta([z_{c,1,\oX,i_{X Y}, j}, j X i]) = \sum_{Y \in \{U,D\}} \eta([z_{c,1,X,j,i_{\oX Y}}, i_{\oX Y} \overline{X} j]).
\end{equation}
Thus, we can eliminate the commutators $\eta([z_{c,1,L,i,j}, j R i])$
for $i \neq i_{RD}, j \neq i_{L D}$.  Furthermore, in the cases $i =
i_{RD}$ and $j = i_{LD}$, these commutators allow one to express
$z_{c,0,L,i_{R D}, i_{RD}}$ in terms of $z_{c, 0, L, i_{RU}, i_{RU}}$,
and $z_{c,0,R,i_{LD}, i_{LD}}$ in terms of $z_{c,0,R,i_{L U},
  i_{LU}}$, so we can eliminate all commutators $\eta([z_{c,1,L,i,j},
j R i])$ along with the generators $z_{c, 0, \oX, i_{X D}, i_{X D}}$
for $X \in \{L,R\}$.  To summarize the result of the computation (with
parenthetical English versions):
\begin{enumerate}
\item We can eliminate $z_{c,C,X,i,i}$
  from the basis if $i$
  is internal and $2(n-2)
  \nmid
  C$, thus passing to a submodule \(V_0 \subset V\)
  (that is, we consider only cyclic paths that are a combination of
  length-two cycles at a vertex and long cycles around the length of $\tilde D_{n}$);
\item The only commutators we then need to consider, in order to
  integrally span $W \cap V_0$ (up to dividing certain commutators
  by certain integers), and thus present $\Lambda_Q \cong V_0 / (W
  \cap V_0)$, are
\begin{enumerate}
\item $\eta([z_{c,C-1,X,i,j}, j (L+R) i])$ for $2(n-2) \mid C$ and
  $j \notin \{i_{LD}, i_{RD}\}$ (we only need consider commutators
  obtained from the previous elements by separating the last arrow off
  and writing the commutator), and
\item $\eta([z_{c,1,R,i,j},jLi])$.
\end{enumerate}
\end{enumerate}

In computing the remaining commutators, we use the following basic
identities, which are similar to those for $\tilde A_0, \tilde A_n$:
First, note that, by the choice of orientation of $\tilde D_{n} = Q^0$
in the previous subsection, in $A = \Pi_{Q}$, one has $\ino (LR - RL)
\ino = \ino (LR - RL + r_{Q^0_0}) \ino \in \langle \overline{Q}_1 \setminus
  \overline{Q^0_1} \rangle^2$.  So, letting $r':= r_{Q^0_0} - \sum_{a \in Q_1^0} (a
a^* - a^* a)$ (as before),
\begin{multline}
R_X (LR) i_R = (R_X LR) i_R = (RLR_\oX + R_X L R_X - R_\oX L R_\oX) i_R
\\ = ((RL) R_\oX  - R_\oX r') i_{R\oX} + R_X r' i_{RX}; \label{fdr}
\end{multline}
\begin{equation}
R (RL) \ino =
((RL)R - R r') \ino; \label{idr}
\end{equation}
and similarly swapping left with right, and the element $r' \in \Pi_Q$
with $-r'$.  We can use these to move an $R$ (or $L$) past a power of
$RL$ (or $LR$), thus allowing one to compute $\eta
([z_{c,C-1,X,i,j},jRi])$ for $X = R$ and $i$ internal, or for $X=L$
and $i$ a right external vertex (and similarly swapping left with
right).  For $i$ internal and $j \neq i_{L D}$ a vertex adjacent to
$i$ on the left, one obtains
\begin{equation} \label{zcomm3}
\eta([z_{c,C-1,R,i,j},jRi]) = z_{c,C,R,i,i} - z_{c,C,R,j,j}
+ \sum_{0 \leq c' < c} [r' z_{c',C-1,R,i,j} z_{c-c'-1,1,R,j,i}].
\end{equation}
(We already noted that we can discard the relation in the case $j =
i_{L D}$.)

For $i \in \{i_{RU}, i_{RD}\}$ right external, one obtains the
following. We use the simplified notation
$i^{(c)} := \begin{cases} i, \text{if $c$ is even}, \\ \overline{i},
  \text{if $c$ is odd}, \end{cases}$ where $\overline{i}$ is the other
right external vertex from $i$.  We'll need the same definition for
left external vertices later.
\begin{multline} \label{zcomm4}
\eta([z_{c,C-1,L,i,j},jRi]) = z_{c,C,L,i,i} - \eta(z_{c,1,R,j,i^{(c)}} 
z_{0,C-1,L,i^{(c)},j})
\\ + \sum_{c' < c, \epsilon \in \{0,1\}} 
(-1)^{1+c'+\epsilon} 
[r' z_{c',C-1,L,i^{(\epsilon)},j} \cdot z_{c-c'-1,1,R,j,i^{(\epsilon)}}];
\end{multline}
We simplify the RHS by noting that 
\begin{equation}\label{zfix}
\eta(z_{c,1,R,j,i^{(c)}} z_{0,C-1,L,i^{(c)},j}) = \begin{cases}
z_{c,C,R,j,j}, & \text{if $i^{(c)} = i_{RU}$}, \\
-z_{c,C,R,j,j}+ Y, & \text{if $i^{(c)} = i_{RD}$},
\end{cases}
\end{equation}
where $Y \in \langle z_{c+m(n-2),C-2m(n-2),R,j,j} \rangle_{m \geq 1}$ (that is, $Y$ consists of ``higher powers of $\ldp \ino RL \ino, \ino LR \ino \rdp$'').

We can thus consider the above relations \eqref{zcomm3},\eqref{zcomm4}
as eliminating for each $a,b$ the basis elements $z_{c,C,X,i,i}$
(where $i$ is not lower external) for all but a single pair
$(X_0,i_0)$ where $i_0 \notin \{i_{RD}, i_{LD}\}$, leaving only a
single relation.  We can conveniently choose this relation to be the
sum of \eqref{zcomm3} and \eqref{zcomm4} over all $i \notin \{i_{RD},
i_{LD}\}$ (with $j \neq i_{LD}$ for \eqref{zcomm3}), together with
$(-1)^c$ times the sum of the commutators obtained from these by
swapping left and right.  This is zero if $c = 0$, so we can assume $c
\geq 1$. By \eqref{zcomm1}, to include all of the commutators we
discarded, we need only divide the resulting relation by $\ds
\frac{c}{\gcd(c,\frac{C}{2(n-2)})}$.

We thus get the single relation (for all choices of $c, C \geq 1$ such
that $2(n-2) \mid C$):
\begin{multline}
  \frac{\gcd(c, \frac{C}{2(n-2)})}{c} \sum_{i,j \notin \{i_{L D},
    i_{RD}\}} \eta([z_{c,C-1,R,i,j},jLi] +(-1)^c
  [z_{c,C-1,L,i,j},jRi]) = \\ \frac{\gcd(c, \frac{C}{2(n-2)})}{c}
  \sum_{c',i,X; j \notin \{i_{LD}, i_{RD}\}} (-1)^{c \delta_{X,L} +
    (c-c')(\delta_{i,i_{\oX D}} + \delta_{i, i_{\oX U}})+
    \delta_{i,i_{\oX D}}} [r' z_{c',C-1,X,i,j} z_{c-c'-1,1,\oX,j,i}] +
  Y,
\end{multline}
where $Y$ is in the integral span of products of terms
$z_{c',C',X,i',j'}$ and $r'$, such that the sum of $C'$-degrees is
strictly less than $b$. That is, $Y$ is the image of a greater power
of the ideal $\ldp \ino RL \ino, \ino LR \ino \rdp$ than $c$ (as is
$z_{c,C,X,i,j}$ for any $X,i,j$).  In other words, if we pass to the
associated graded algebra of $\Pi_Q$ with respect to the filtration by
powers of the ideal $\ldp \ino RL \ino, \ino LR \ino \rdp$ considered
earlier, and the associated graded $\Z$-module with respect to the
image of this filtration in $\bar \Pi_Q/([\bar \Pi_Q(\overline{Q}_1
\setminus \overline{Q^0_1})\bar \Pi_Q,\bar \Pi_Q] \cap \bar \Pi_Q)$,
then we can eliminate the $Y$ term and have an equality above.  In any
case, by induction on filtration degree, we will see that one can
neglect the $Y$ portion.

We then expand the RHS in terms of the basis of $V$ previously
described: (1) the elements $z_{c,C,X,i,j}$ from Theorem
\ref{dnt}.(i); and (2) cyclic alternating products of elements
$z_{c,C,X,i,j}$ and monomials in $\overline{Q}_1 \setminus \overline{Q^0_1}$ not
containing $a a^*$ for any $a \in G$ (the forest we picked as in
Proposition \ref{bpp}). This results in the expression
\begin{equation}
  \frac{\gcd(c, \frac{C}{2(n-2)})}{c} 
  \sum_{(c_\cdot, C_\cdot, X_\cdot, i_\cdot, j_\cdot) \mid (*)} 
  (c_1+1) s(c_\cdot,C_\cdot,i_\cdot,X_\cdot) [\prod_{m=1}^{|\{c_\cdot\}|} r' z_{c_m,C_m,X_m,i_m,i_{m+1}}] + Y,
\end{equation}
\begin{multline} 
  s(c_\cdot,C_\cdot,i_\cdot,X_\cdot):= (-1)^{c \delta_{X_1,L}}
  \prod_{m=1}^{|\{c_\cdot\}|}
  (-1)^{(\delta_{i_m,i_{RD}}+\delta_{i_m,i_{L D}}) + c_m(\sum_{k \geq m}B_k)}, \\
  \text{where }B_k = \text{\# of times that an arrow of the form $R_Z
    i_{R Z}$ or $L_Z i_{L Z}$} \\ \text{appears in the factor
    $z_{c_k,C_k,X_k,i_k,i_{k+1}}$, and } Y \in \ldp \ino RL \ino, \ino
  LR \ino \rdp^{c+1}.
\end{multline}
where the $(*)$ in the first sum indicates that we sum only over
distinct $|\{c_\cdot\}|$-tuples of elements of $\N \times \N \times
\{L,R\} \times Q^0_0 \times Q^0_0$ such that $\sum c_m = c-|\{c_\cdot\}|,
\sum C_m = C, C_m \geq 1, \forall m$, and $X_{m+1}$ is always a
forward direction that follows $z_{c_m,C_m,X_m,i_m,i_{m+1}}$ (which
can be upward or downward).  We take indices modulo the length of the
tuples (e.g.~$i_{|\{c_\cdot\}|+1}=i_1$).

Since $\sum_{m=1}^{|\{c_\cdot\}|} B_m$ is always even (because $2(n-m)
\mid C$), when we pass to summing only over distinct \textbf{cyclic}
$|\{c_\cdot\}|$-tuples, all of the contributing terms have the same
sign, leaving us with
\begin{equation} \label{finaldnreln}
 \sum_{(c_\cdot, C_\cdot, X_\cdot, i_\cdot, j_\cdot) \mid (*)} 
\pm \frac{\gcd(c, \frac{C}{2(n-2)})}{\rep(c_\cdot,C_\cdot,X_\cdot,i_\cdot,j_\cdot)}
[\prod_{m \in \Z/|\{c_\cdot\}|} r' z_{c_m,C_m,X_m,i_m,i_{m+1}}] + Y,
\end{equation}
which now so closely resembles \eqref{wabdefn1}--\eqref{wabdefn3} that
it is straightforward to conclude that there is no torsion in the
portion of $\Lambda_Q$ corresponding to $C \geq 1$. To be precise,
consider the grading on $\Pi_Q$ and $\Lambda_Q$ by length of paths, so
that $(\Pi_Q)_m, (\Lambda_Q)_m$ denote the $\Z$-submodules integrally
spanned by paths of length $m$.  Consider the filtration by powers of
the ideal $\ldp \ino RL \ino, \ino LR \ino \rdp$ in $\Pi_Q$, and the
filtration by the image of these powers in $\Lambda_Q$. Then we have
the associated graded $\Z$-modules, $\gr \Pi_Q = \bigoplus_g \gr_g
\Pi_Q$ and $\gr \Lambda_Q = \bigoplus_g \gr_g \Lambda_Q$.  Thus,
$\gr_g \Pi_Q =
\ldp \ino RL \ino, \ino LR \ino \rdp^g / \ldp \ino RL \ino, \ino LR
\ino \rdp^{g+1}$, and $\gr_g \Lambda_Q$ is its quotient by the
integral span of commutators.

Then, the precise statement we infer from the above is that there is
no torsion in $(\gr_g \Lambda_Q)_m$ except possibly when $2g =
m$. This is the maximum possible value of $g$, so that $(\gr_g
\Lambda_Q)_{2g} = [ \ldp \ino LR \ino, \ino RL \ino \rdp^g]_{2g}$ is
actually a submodule of $(\Lambda_Q)_{2g}$.  Therefore it has torsion
if and only if $(\Lambda_Q)_{2g}$ itself does.

It remains to see if the module $[\ldp \ino LR \ino, \ino RL \ino \rdp^g]_{2g}$
has torsion. We can present the  module as follows.  It is integrally
spanned by $[z_{c,0,R,i,i}], [z_{c,0,L,i_{RY},i_{RY}}]$, and multiples of
$\overline{Q}_1 \setminus \overline{Q^0_1}$ (with each $z_{c_t,0,X_t,i,i}$ having
$X_t = R$ unless $i$ is right external, in which case each $X_t = L$).
The relations are spanned by $\eta([z_{c-1,1,R,i,j}, jLi])$ for $j
\neq i_{R D}, i \neq i_{L D}$ (we can ignore other commutators by our
initial arguments), and by $(\sum_{i \in Q_0 \setminus Q^0_0} i r i)$.
When $i$ and $j$ are internal vertices, we can view this as
eliminating $z_{c,0,R,i,i}$ for all internal vertices except one, and
expressing the other in terms of that one and multiples of $r'$.  This
leaves us with only two relations.  Let $j$ be the internal vertex
adjacent to the left end. For convenience, let us replace
$\eta([z_{c-1,1,R,i_{LU},j},jLi_{LU}])$ by
$\eta([z_{c-1,1,R,i_{LU},j},jLi_{LU}]) +
\eta([z_{c-1,1,R,i_{LD},j},jLi_{LD}])$.  (The added term is exactly
what is used to eliminate the $z_{c,2,L,j,i_{LY}}$ we would otherwise
need as a generator, cf.~\eqref{zcomm1} and the following paragraphs.)
Then, this relation can be viewed as expressing $z_{c,0,R,j,j}$ in
terms of $z_{c,0,R,i_{L U}, i_{LU}}$, since we already expressed
$z_{c,0,R,i_{LD}, i_{LD}}$ in terms of the latter.  This eliminates
the generator $z_{c,0,R,i,i}$ for the last remaining internal vertex
$i$, since we may take $i=j$.

To express the final relation in terms of remaining basis elements, we
take the sum over all adjacent $i,j$ of $\eta([z_{c-1,1,R,i,j},jLi])$.
The result is
\begin{equation} \label{dnsreln} \sum_{i \in Q^0_0}
  \eta(i((RL)^{c}-(LR)^{c})i) = \sum_{i \in Q^0_0 \mid \text{$i$ is
      internal}} [i((RL)^{c} - (RL + r')^{c})i] + \sum_{X \in \{U,D\}}
  (i_{LX} (RL)^c i_{LX} - i_{RX} (LR)^c i_{RX}),
\end{equation}
where we used the non-basis elements $z_{c,0,R,i_{LD}, i_{LD}}$ and
$z_{c, 0, L, i_{RD}, i_{RD}}$ for convenience, which can be replaced
by basis elements with only a slight modification.

The result is nonzero, since there exists a vertex $i \in Q^0_0$ such
that $i$ is adjacent to an arrow of $\overline{Q \setminus Q^0}$; in
this case, for every proper factor $1 \neq m \mid c$, one has the term
$\ds [i((RL)^{\frac{c}{m}-1} r')^{m} i]$, nonzero and independent of
other terms in the expansion, with coefficient $m$.  Hence, the gcd of
all coefficients is one unless $c$ is a prime power. We conclude that
there is only torsion in $\Lambda_Q$ in degrees $2p^k$ for $p$ prime
and $k \geq 1$, where it is generated by the single relation
\eqref{dnsreln}.  For much the same reason as in the $\tilde
A_{n-1}$-case (by the similarity of \eqref{dnsreln} and
\eqref{wabdefn1}--\eqref{wabdefn3}; see the end of the proof for
$\tilde A_{n-1}$), the torsion is a single copy of $\Z/p$, generated
by $\frac{1}{p}[r^{p^k}]$.

This completes the proof of Theorem \ref{dnt}.
\subsection{Hilbert series for $\tilde D_n$ and \eqref{egid}}\label{edn}
From Theorem \ref{dnt}, we deduce the
\begin{prop}
  The Hilbert series of $i_0 \Pi_{\tilde D_{n}} i_0$ over any field, for $i_0$ the extending vertex
  (i.e.~an external vertex) of $\tilde D_n$, and the Hilbert series
  of $\Lambda_{\tilde D_n}$ over characteristic zero, are given by
\begin{equation}
  h(i_0 \Pi_{\tilde D_{n-1}} i_0;t) = h(\Lambda_{\tilde D_{n}};t) = \frac{1}{(1-t^{2n-4})(1-t^2)} -
\frac{t^2}{1-t^4} = \frac{1+t^{2n-2}}{(1-t^4)(1-t^{2n-4})}.
\end{equation}
\end{prop}
It immediately follows from the above that one has the formula
\begin{equation} \label{pfdn}
h(\Lambda_{\tilde D_{n}};t)(1-t^2) = 1 + \frac{t^{2n-4}}{1-t^{2n-4}} + \frac{2t^4}{1-t^4} - \frac{t^2}{1-t^2},
\end{equation}
which we can use to verify the formulas \eqref{egedf},\eqref{egedf2}.
Namely, letting $C$ be the adjacency matrix of $\overline{\tilde D_n}$, one has
\begin{equation} 
\det(1 - t \cdot C + t^2 \cdot \1) = \frac{(1-t^{2n-4})(1-t^4)^2}{1-t^2}.
\end{equation}

One then has from \eqref{pfdn}
\begin{equation}
a_m - a_{m-2} = [(2n-4) \mid m] + 2[4 \mid m] - [2 \mid m], \quad m \geq 3,
\end{equation}
which implies the desired identity \eqref{egedf2} and verifies \eqref{egid} for
$\tilde D_n$.
\subsection{The Poisson algebra $i_0 \Pi_Q i_0$
  for $Q = \tilde D_n$} \label{dans} As in \S \ref{lans}, the Lie
structure on $\Lambda_Q$ can be described in terms of a Poisson
algebra structure on $i_0 \Pi_Q i_0$. In particular, by Theorem
\ref{dedz} (which does not rely on the results of this subsection) we
can define a Lie bracket on $i_0 \Pi_Q i_0$ using the projection
$q: (\Lambda_Q)_+ \onto (i_0 \Pi_Q i_0)_+$ and inclusion $j: i_0 \Pi_Q
i_0 \into \Lambda_Q$ by the formula
\begin{equation}\label{e:i0pq-br}
\{f,g\} = q\{j(f),j(g)\}.
\end{equation}
It follows as before from \eqref{ncleibeqn} that this endows $i_0
\Pi_Q i_0$ with the structure of a Poisson algebra. In view of
Proposition \ref{rkp} and Remark \ref{r:rkp}, in fact the Poisson
structure on $i_0 \Pi_Q i_0$ entirely captures the Lie algebra
structure on $\Lambda_Q$.

We then compute the following:
\begin{prop} \label{dpbp}
Let $Q = \tilde D_n$, and let $i_0 := i_{LU}$ be the upper-left external vertex.
Set $X := i_0 z_{2,0} i_0, Y := i_0 z_{0,2(n-2)} i_0$, and $Z := i_0 z_{1,2(n-2)} i_0$.
Then, one has
\begin{gather} \label{i0dnpres}
i_0 \Pi_Q i_0 \liso \Z[X,Y,Z] / \ldp Z^2 + XY^2 - \delta_{2 \mid n} X^{\frac{n}{2}} Y - \delta_{2 \nmid n} X^{\frac{n-1}{2}} Z \rdp, \\
\{X,Y\}= 2 Z - \delta_{2 \nmid n} X^{\frac{n-1}{2}}, 
\quad \{X,Z\} = 2 X Y - \delta_{2 \mid n} X^{\frac{n}{2}},
\\ \{Y,Z\} = Y^2-\delta_{2 \mid n} \frac{n}{2} X^{\frac{n-2}{2}} Y - \delta_{2 \nmid n}\frac{n-1}{2}
X^{\frac{n-3}{2}} Z, \\
|X| = 2, \quad |Y| = 2(n-2), \quad |Z| = 2(n-1).
\end{gather}
\end{prop}
Notice that, over $\Z[\frac{1}{2}]$,  $i_0 \Pi_Q i_0 \o \Z[\frac{1}{2}]$ is generated as a Poisson
algebra by $X$ and $Y$.
\begin{proof}
This can all be computed using the results of the preceding sections.  To see that the given relation is
the only relation (i.e., the map in \eqref{i0dnpres} is injective), we can compare Hilbert series.
The rest can all be computed using the basis.  To compute $\{Y,Z\}$, it is easiest to compute everything
else first, and then compute $\{Y,Z^2\}$ two ways: either as $2Z\{Y,Z\}$, or using the formula
\eqref{i0dnpres} for $Z^2$, and then use that $i_0 \Pi_Q i_0$ is an integral domain (which is
easy to see from the filtration and bases).
\end{proof}

\section{(Partial) Preprojective algebras of star-shaped quivers}\label{ssqsec}
The purpose of this section is to establish some basic results about
the structure of (partial) preprojective algebras of star-shaped
quivers, and more generally to study what happens to preprojective
algebras when one adds line segments or loops to the quiver. We will
begin in \S \ref{gasec} with geometric motivation, relating
preprojective algebras to Riemann surfaces.  The reader not interested
in this purely motivational material may safely skip \S \ref{gasec}.

\subsection{A geometric analogy} \label{gasec}
There is a known analogy between the preprojective algebra of
star-shaped quivers with branches of lengths $d_1, \ldots, d_n$ and
the Riemann spheres with orbifold points of orders
$d_1+1, \ldots, d_n+1$, wherein the preprojective algebra is the
``additive'' version and the fundamental group of the associated
orbifold is the ``multiplicative'' version.  Namely, the fundamental
group of the latter orbifold is the quotient of the free group on $x_1, \ldots, x_n$
by the relations
\begin{equation}
x_1^{d_1+1}, \ldots, x_n^{d_n+1}, x_1 x_2 \cdots x_n,
\end{equation}
whereas, for $i_s$ the special vertex of the preprojective algebra of the associated quiver $Q$, and now letting $x_i := e_i e_i^*$ where $e_i$ is the arrow in $\dq$ in the $i$-th branch which begins at $i_s$ (which we will assume
for simplicity lies in $Q$), 
\begin{equation}
i_s \Pi_Q i_s \cong \Z \langle x_1, \ldots, x_n \rangle / \ldp x_1^{d_1+1}, \ldots, x_n^{d_n+1}, x_1 + x_2 + \cdots + x_n \rdp.
\end{equation}
Next, suppose that $j \in Q_0$ is the endpoint of a branch, say the first branch. Then, if we take the partial preprojective algebra with respect to $\{j\}$, this changes the $i_s$-part by eliminating the relation $x_1^{d_1+1}$, i.e.:
\begin{equation}
i_s \Pi_{Q,\{j\}} i_s \cong \Z \langle x_1, \ldots, x_n \rangle / \ldp x_2^{d_2+1}, \ldots, x_n^{d_n+1}, x_1 + x_2 + \cdots + x_n \rdp.
\end{equation}
This is the same as letting $d_1 \rightarrow \infty$, and corresponds
geometrically to replacing the first orbifold point by a puncture
point.

\subsection{Preprojective algebras of quivers with line segments added}
The last comment suggests the following interpretation of the partial
preprojective algebra: Take the underlying quiver and adjoin, at each
white vertex $j \in J$, an infinite ray based at $j$, to form an
extended quiver $\hat Q_{J,\infty}$.  Then, the projection $\1_{Q_0}
\Pi_{\hat Q_{J,\infty}} \1_{Q_0}$ of the preprojective algebra of this
extended quiver to the span of paths beginning
and ending at $Q_0$ (see Remark \ref{rem:prep-infinite} below)
is the partial preprojective algebra
$\Pi_{Q,J}$.  Precisely, for any quiver $Q$ and any
subset $J \subseteq Q_0$, we make the following definition:
\begin{defn} \label{qhatdfn}
  For any map $f: J \rightarrow \N$, let $\hat Q_{J,f}$ be the
  quiver obtained from $Q$ by attaching a segment of length $f(j)$ to
  each $j \in J$ (i.e., the segment has $f(j)$ arrows and $f(j)$
  vertices), oriented outward from $j$.
\end{defn}
The choice of orientation in the definition is not important; we make
it only for definiteness.

Of course, we could have taken $J = Q_0$ in the above definition
without loss of generality, but it will be convenient for the limit we
consider below to sometimes have $J$ be a proper subset (i.e., to
compare $\Pi_{Q,J}$ with a limit $\ds
\lim\limits_{\underset{f}{\leftarrow}} \Pi_{\hat Q_{J, f}}$).
\begin{defn}
For any two functions $f, g: J \rightarrow \N$, 
we say that $f \leq g$ if $f(j) \leq g(j)$ for all $j \in J$.
\end{defn}
Note that, for any $f \geq g$, one has a surjection $\Pi_{\hat Q_{J,
    f}} \onto \Pi_{\hat Q_{J, g}}$.  One may consider the inverse
limit of \emph{graded algebras}, $\ds \lim\limits_{\underset{f}{\leftarrow}}
\Pi_{\hat Q_{J, f}}$.  Here, the fact that we are taking the limit
as graded algebras means that the limit is defined separately in each
degree, using sequences of homogeneous elements.

We may think of this limit as $\Pi_{\hat Q_{J, \infty}}$, where
$\hat Q_{J, \infty}$ is the ``infinite quiver'' described above, using
the following remark:
\begin{rem}\label{rem:prep-infinite}
  The definition of preprojective algebra makes sense even if the
  quiver has infinitely many arrows, as long as there are only
  finitely many arrows incident to each vertex.  Call such a quiver a
  locally finite quiver.  Then the preprojective algebra of a locally
  finite quiver $Q$ on vertex set $Q_0$ is
\begin{equation}\label{e:prep-locfin}
\Pi_Q := P_{\dq}/(\sum_{a \in Q_1:  a_s = i} (a a^* - a^* a))_{i \in Q_0},
\end{equation}
i.e., even though the element $r$ may no longer exist, we can still
define the elements $r_i := \sum_{a \in Q_1 : a_s = i} (a a^* - a^* a)$,
and take the ideal generated by these elements. In the case of a
finite quiver, $r_i = iri$ and $\ldp r\rdp = \ldp iri\rdp_{i \in Q_0}$,
so this indeed recovers the original definition.
\end{rem}

Note that Hochschild homology does
not commute with the above inverse limit. One has instead 
the following result:
\begin{prop} \label{ppp} 
We have the following natural isomorphisms:
\begin{enumerate}
\item[(i)] $\Pi_{Q, J} \cong \lim\limits_{\underset{f}{\leftarrow}} \1_{Q_0} \Pi_{\hat Q_{J, f}} \1_{Q_0}$.
\item[(ii)] $\lim\limits_{\underset{f}{\leftarrow}} \Lambda_{\hat Q_{J, f}} \cong
\Lambda_{Q, J}/ \langle j r^\ell j \rangle_{j \in J, \ell \geq 1}$.
\end{enumerate}
\end{prop}
\begin{proof} 
(i) This follows from the formula 
\begin{equation} \label{piqhatf}
\1_{Q_0} \Pi_{\hat Q_{J, f}} \1_{Q_0} \cong \Pi_{Q,J} / \ldp j r^{f(j)} j \rdp_{j \in J}.
\end{equation}
Note that the limit is of graded algebras, and in each degree the RHS
stabilizes to the LHS; in particular all classes in the RHS are
represented by finite linear combinations of paths in $\hat Q_{J,\infty}$.

(ii) Let $I$ be the vertex set of $\hat Q_{J, f}$.  By the following
Proposition \ref{attlsp}, 
the inclusion
 $\1_{Q_0}
\Pi_{\hat Q_{J, f}} \1_{Q_0} \into \Pi_{\hat Q_{J, f}}$
induces
an isomorphism $(\1_{Q_0}
\Pi_{\hat Q_{J, f}} \1_{Q_0})_{\cyc} / \langle [a_s^{\ell}], [i] \rangle_{s
  \in J, \ell \geq 1, i \in I \setminus Q_0} \iso
\Lambda_{\hat Q_{J, f}}$, where, for each $s \in J$,  
$a_s$ is the cycle of length two
obtained by beginning at $s$,
traversing one arrow in $\overline{\hat Q_{J, f}} \setminus \dq$,
and then its reverse back to $i_s$. If we now assume that $f$ is such
that $f(i) \geq N$ for all $i \in J$, then in degrees $m < 2N$, we
evidently have $((\1_{Q_0} \Pi_{\hat Q_{J, f}} \1_{Q_0})_{\cyc})_m \cong
(\Lambda_{Q, J})_m$, using the isomorphism \eqref{piqhatf}.
Under this isomorphism, the image of $a_s^\ell$ is $i_s r^\ell i_s$,
for every $i_s \in J$.  Now taking the limit $N \rightarrow
\infty$ gives the result.
\end{proof}
In the above proposition, we used the following  result for 
for (still more general) quivers $\hat Q$ obtained from 
$Q$ by attaching any number of line
segments to its vertices (not limiting to one line segment per
vertex).  In particular, this explains why adding infinite rays to
a vertex is like adding ``punctures'' to the corresponding ``surface.''
\begin{defn} Beginning with a quiver $Q$, a (not
  necessarily distinct) collection of vertices $i_1, \ldots, i_m \in
  Q_0$, and some positive integers $d_1, \ldots, d_m$, let $\hat Q :=
  \hat Q_{i_1,\ldots,i_m, d_1, \ldots, d_m}$ be the quiver obtained
  from $Q$ by attaching line segments $L_s, s \in \{1, \ldots, m\}$ of
  lengths $d_s$ to vertices $i_s \in Q_0$, oriented outward from the
  vertex. That is, $L_s$ is a quiver whose arrows are disjoint from $Q$,
  sharing only the vertex $i_s$. Then
  $(L_s)_1 \cap Q_1 = \emptyset$ and $(L_s)_0 \cap Q_0 = \{i_s\}$.
  Let $L_s^* \subseteq \hat Q^*$ be the oppositely oriented
  quiver and $\overline{L_s}:=L_s \cup L_s^* \subseteq \overline{\hat Q}$ the
  double.
\end{defn}
The choice of orientation in the definition of the segments $L_s$ is
not important; we made an arbitrary choice.
We will also make use of the following notation:
\begin{ntn}\label{n:pijdefn}
  For any branch $L_s$ based at vertex $i_s$ and any vertex $i \in
  (L_s)_0$, let $p_{i_s i}$ and $p_{i i_s}$ denote the straight-line paths
  in $\hat Q$ from $i_s$ to $i$ and from $i$ to $i_s$ inside
  $\overline{L_s}$, respectively. Let $d_{i_s,i}=d_{i,i_s}$ denote the
  lengths of these paths. For each line segment $L_s$, let $C_s$ be
  the cycle in $\bar{\hat Q}$ of length two which begins at $s$,
  travels one arrow along the segment $L_s$, and then takes the reverse
  arrow back to $s$.
\end{ntn}
\begin{prop} \label{attlsp}
\begin{enumerate}
\item[(i)] The inclusion
$\1_{Q_0} \Pi_{\hat Q} \1_{Q_0} \into \Pi_{\hat Q}$ induces a surjection
\begin{equation} \label{hh0hats}
(\1_{Q_0} \Pi_{\hat Q} \1_{Q_0})_{\cyc} \onto \Lambda_{\hat Q}
\end{equation}
with kernel $\langle [C_s^\ell] \rangle_{s \in \{1,\ldots,m\}, \ell \geq 1}$.
As a consequence, the inclusion induces an isomorphism
\begin{equation} \label{hh0hati}
\Lambda_{\hat Q} \liso (\1_{Q_0} \Pi_{\hat Q} \1_{Q_0})_{\cyc} / \langle [C_s^\ell] \rangle_{s \in \{1,\ldots, m\}, \ell \geq 1}.
\end{equation}
\item[(ii)] One obtains isomorphisms, for $J = \{i_1,\ldots,i_m\}$,
\begin{gather}
\Lambda_{\hat Q} \liso (\Pi_{Q,J} *_{\Z^{Q_0}} B)_{\cyc} /\langle [C_s^\ell] \rangle_{s \in \{1,\ldots,m\}, \ell \geq 1}, \text{ where}\\ B =
\bigoplus_{j \in J} B_j, \quad B_j = \Z \langle C_s \mid i_s = j \rangle / 
\ldp C_s^{d_s}, \sum_{s' \mid i_{s'} = j} C_{s'} \rdp_{s \mid i_s = j}.
\end{gather}
\end{enumerate}
\end{prop}
Note that it is easy to combine Propositions \ref{ppp} and
\ref{attlsp} to describe the case of a quiver with some finite and
some infinite line segments added. In this direction, we only state
the promised
\begin{cor} \label{quivsfcc}
  Let $Q$ be the quiver obtained by beginning with one vertex $i_s$,
  and attaching $g$ loops, and $m$ line segments of lengths
  $d_1, \ldots, d_m$ (allowing for $d_j = \infty$).  
  Then, defining $\Lambda_Q = (\Pi_Q)_{\cyc}$ where
  $\Pi_Q = P_{\dq} / \ldp i r i\rdp_{i \in Q_0}$, one obtains formulas
\begin{gather} \label{addansfc1} i_s \Pi_Q i_s \cong A := \Z \langle
  x_1, \ldots, x_g, y_1, \ldots, y_g, p_1, \ldots, p_m \rangle / \ldp
  \sum_{i=1}^g [x_i, y_i] + \sum_{i=1}^m p_i, p_j^{d_j+1} \rdp_{j
    \in \{1,\ldots,m\}},
  \\ \label{addansfc2} \Lambda_Q \cong A/([A,A] +\langle p_j^\ell
  \rangle_{j \in \{1,\ldots,m\}, \ell \geq 1}),
\end{gather}
where we set by definition $p_j^\infty := 0$ (for any $j$).  Thus,
$\Lambda_Q$ is the ``additive analogue of the orbifold surface of
genus $g$ with orbifold/puncture points $p_1, \ldots, p_m$ of orders
$d_1+1, \ldots, d_m+1$.''
\end{cor}
The corollary easily follows from part (ii) of Proposition \ref{attlsp} using
the argument of the proof of Proposition \ref{ppp}.

The rest of this section is devoted to the proof of Proposition
\ref{attlsp}. 
\begin{proof}[Proof of Proposition \ref{attlsp}]
(i) First, we show surjectivity in \eqref{hh0hats}.  
Any path in $P_{\overline{\hat Q}}$
which lies entirely in one of the line segments $\overline{L_s}$ (for any fixed $s$)
maps to zero in
$\Lambda_{\hat Q}$.  Then, any other path projects to the
same class as a path that begins and ends at a vertex in $Q^0_0$.
This proves the surjectivity. Also, it is clear that $[a_s^\ell]$ is in the
kernel of \eqref{hh0hats} for any $\ell \geq 1$ and $s \in \{1,\ldots,m\}$, so
it remains to show that the resulting map \eqref{hh0hati} is injective (and
hence an isomorphism).

To do this, we construct an explicit inverse. For this, we need to
characterize the $\Z$-module $j_1 \Pi j_2$ in terms of $\1_{Q_0} \Pi_{\hat
  Q} \1_{Q_0}$ (in cases when at least one of $j_1, j_2$ are not in $Q_0$).
We use the following lemma.  
\begin{lemma} \label{qhatijl} In the situation of Proposition
  \ref{attlsp}, let $\hat Q_0$ be the vertex set of $\hat Q$, and let
  $j_1, j_2 \in \hat Q_0 \setminus Q_0$, which are on line segments
  $\overline{L_{s_1}}$ and $\overline{L_{s_2}}$, respectively.  
Then, there is an exact sequence of $\Z$-modules,
\begin{equation}
0 \rightarrow a_{s_1}^{d_{s_1} - d_{i_{s_1},j_1} + 1} \Pi_{\hat Q} +
\Pi_{\hat Q} a_{s_2}^{d_{s_2} - d_{i_{s_2},j_2} + 1} \into i_{s_1} \Pi_{\hat Q}
i_{s_2} \mathop{-\!\!\!-\!\!\!\onto}^{x \mapsto  p_{j_1 s_1} x p_{s_2 j_2}} 
j_1 \Pi_{\hat Q} j_2 \rightarrow 0, \quad \text{if $s_1 \neq s_2$},
\end{equation}
\begin{multline}
0 \rightarrow a_{s_1}^{d_{s_1} - d_{i_{s_1},j_1} + 1} \Pi_{\hat Q} +
\Pi_{\hat Q} a_{s_2}^{d_{s_2} - d_{i_{s_2},j_2} + 1} \into i_{s_1} \Pi_{\hat Q}
i_{s_2} \traa^{x \mapsto  p_{j_1 s_1} x p_{s_2 j_2}} 
j_1 \Pi_{\hat Q} j_2 \\ \rightarrow \sum_{\ell=0}^{\min(d_{i_{s_1} j_1}, d_{i_{s_2} j_2})-1} p_{j_1 j_2} (a')^\ell p_{j_2 j_1} \rightarrow 0, \quad \text{if $s_1 = s_2$},
\end{multline}
where $a'$ is any cycle of length two inside $\overline{L_{s_2}}$, 
beginning and ending at $j_2$.

Furthermore, if one of $j_1, j_2$ is in $Q_0$ and the other is in $\hat Q_0 \setminus Q_0$,
we have the exact sequence 
\begin{gather}
0 \rightarrow a_{s_1}^{d_{s_1} - d_{i_{s_1},j_1} + 1} \Pi_{\hat Q}
 \into i_{s_1} \Pi_{\hat Q}
j_2 \mathop{-\!\!\!-\!\!\!\onto}^{x \mapsto  p_{j_1 s_1} x} 
j_1 \Pi_{\hat Q} j_2 \rightarrow 0, 
\quad \text{if } j_1 \in (L_{s_1})_0, j_2 \in Q_0, \text{ or}\\
0 \rightarrow
\Pi_{\hat Q} a_{s_2}^{d_{s_2} - d_{i_{s_2},j_2} + 1} \into i_{s_1} \Pi_{\hat Q}
i_{s_2} \mathop{-\!\!\!-\!\!\!\onto}^{x \mapsto  x p_{s_2 j_2}} 
j_1 \Pi_{\hat Q} j_2 \rightarrow 0, \quad \text{if } j_2 \in (L_{s_2})_0, j_1 \in Q_0.
\end{gather}
\end{lemma}
\begin{proof}
The image of any path in $\Pi_{\hat Q}$ which lies strictly inside a line segment 
$\overline{L_{s}}$ only depends on its endpoints and length.  This and the fact that any path from $j_1$ to $j_2$ must pass through $\dq$ if $i_{s_1} \neq i_{s_2}$ shows the
exactness at $j_1 \Pi_{\hat Q} j_2$ above.  Exactness (injectivity) at $a_{s_1}^{d_{s_1} - d_{i_{s_1},j_1}} \Pi_{\hat Q} +
\Pi_{\hat Q} a_{s_2}^{d_{s_2} - d_{i_{s_2},j_2}}$ and at $\sum_{\ell=0}^{\min(d_{i_{s_1} j_1}, d_{i_{s_2} j_2})-1} p_{j_1 j_2} (a')^\ell p_{j_2 j_1}$ (surjectivity) is obvious. It remains to show exactness at $i_{s_1} \Pi_{\hat Q} i_{s_2}$.  That is, it remains to
compute the kernel of $x \mapsto p_{j_1 s_1} x p_{s_2 j_2}$.  To do this, let
us consider
\begin{equation}
 j_1 P_{\bar{\hat Q}} j_2 \onto j_1 \Pi_{\hat Q} j_2.
\end{equation}
The kernel of this is $j_1 \ldp r \rdp j_2$, which we may rewrite as follows.
For all $s$, let $\1_{L_s}$ be the sum of all vertices on $L_s$ except 
$i_s$ (so, the vertices from $\hat Q_0 \setminus Q_0$ on $L_s$).  Then,
\begin{equation}
j_1 \ldp r \rdp j_2 = j_1 \ldp (\1_{s_1} + \1_{s_2})r \rdp j_2 + p_{j_1 i_{s_1}} \ldp r \rdp p_{i_{s_2} j_2}.
\end{equation}
We deduce (using the RHS and the observations at the beginning of the
proof) that
\begin{equation}
p_{j_1 i_{s_1}} P_{\bar{\hat Q}} p_{i_{s_2} j_2} \cap 
j_1 \ldp r \rdp j_2  = p_{j_1 i_{s_1}} a_{s_1}^{d_{s_1} - d_{i_{s_1},j_1} + 1} \Pi_{\hat Q} p_{i_{s_2} j_2} +
p_{j_1 i_{s_1}} \Pi_{\hat Q} a_{s_2}^{d_{s_2} p_{i_{s_2} j_2} - d_{i_{s_2},j_2} + 1} + p_{j_1 i_{s_1}} \ldp r \rdp p_{i_{s_2} j_2}.
\end{equation}
This yields the desired result.
\end{proof}

Now, we define the map
$\Lambda_{\hat Q} \rightarrow (\1_{Q_0} \Pi_{\hat Q} \1_{Q_0})_{\cyc} / \langle
[C_s^\ell] \rangle_{s \in \{1,\ldots, m\}, \ell \geq 1}$
  as follows. First, let us define a map 
\begin{equation}\label{premap}
\Pi_{\hat Q} \rightarrow
(\1_{Q_0} \Pi_{\hat Q} \1_{Q_0})_{\cyc}.
\end{equation}
First, the map sends $i x j$ to zero
if $i \neq j$ for vertices $i,j \in \hat Q_0$.  Next, on $\1_{Q_0} \Pi_{\hat Q} \1_{Q_0}$,
the map is the tautological one. Then,
for any $j \in \hat Q_0 \setminus Q_0$, with $j$ in the line
  segment $L_s$ ($j \in (L_s)_0$), we set
  $p_{j i_{s}} x p_{i_s j} \mapsto [x (p_{i_s j} p_{j i_s})] = [x
  C_s^{d_{i_s j}}]$.
  To see that this is well-defined, by the lemma it suffices to show
  that if
  $x \in C_s^{d_s-d_{i_s j}+1} \Pi_{\hat Q} + \Pi_{\hat Q} C_s^{d_s-d_{i_s
      j}+1}$,
  then
  $[x C_s^{d_{i_s j}}] = 0 \in (\1_{Q_0} \Pi_{\hat Q} \1_{Q_0})_{\cyc}$.
  However, this follows from the fact that
  $[C_s^{d_s - d_{i_s j} + 1} \Pi_{\hat Q} C_s^{d_{i_s j}}] +
  [\Pi_{\hat Q} C_s^{d_{i_s j}} C_s^{d_s - d_{i_s j} + 1}] \subset
  [\1_{Q_0} \Pi_{\hat Q} \1_{Q_0}, \1_{Q_0} \Pi_{\hat Q} \1_{Q_0}] + [C_s^{d_s + 1}
  \1_{Q_0} \Pi_{\hat Q} \1_{Q_0}]$.
  Then, the only elements of $j \Pi_{\hat Q} j$ which can not be
  written in the form $p_{j i_{s}} x p_{i_s j}$ are (by the Lemma)
  those classes represented by paths lying entirely in $L_s$, of
  total length less than $2d_{j i_s}$.  Let us define such paths to
  map to zero. We thus get a well-defined map \eqref{premap}.

Next, if we post-compose the map with the quotient
\[
(\1_{Q_0} \Pi_{\hat Q} \1_{Q_0})_{\cyc} \onto (\1_{Q_0} \Pi_{\hat Q} \1_{Q_0})_{\cyc} /
\langle [C_s^\ell] \rangle_{s \in \{1,\ldots, m\}, \ell \geq 1},
\] then
all elements of $\Pi$ represented by paths which lie entirely in $L_s$ (of any length) map to zero. It remains
to show that the composite map kills $[\Pi_{\hat Q}, \Pi_{\hat Q}]$.  For this, we
need to show that
\begin{enumerate}
\item $[p_{j i_s} f p_{i_s j}, p_{j i_s} f' p_{i_s j}]$ maps to zero
for any $f, f' \in i_s \Pi_{\hat Q} i_s$ and $j \in (L_s)_0$, 
\item $[p_{j i_s} f p_{i_s j}, a']$ maps to zero for any $f \in i_s \Pi_{\hat Q} i_s$, $j \in (L_s)_0$ ($j \neq i_s$), and where $a'$ is a path of length two beginning
and ending at $j$.
\end{enumerate}
The image of the class in part (1) is $[f C_s^{d_{i_s j}} f' C_s^{d_{i_s j}}]
-[f' C_s^{d_{i_s j}} f C_s^{d_{i_s j}}] = 0$.  The image of the class in part 
(2) is $[f C_s^{d_{i_s j}+1}] - [C_s f C_s^{d_{i_s j}}] = 0$.  This proves that
our map descends to a map
\begin{equation}
\Lambda_{\hat Q} \rightarrow (\1_{Q_0} \Pi_{\hat Q} \1_{Q_0})_{\cyc} / \langle
C_s^\ell \rangle_{s \in \{1,\ldots, m\}, \ell \geq 1}.
\end{equation}
By definition, this map inverts the map in
\eqref{hh0hati} induced by the inclusion.

(ii) This easily follows from (i) using the formula
\begin{equation}
\1_{Q_0} \Pi_{\hat Q} \1_{Q_0} \cong \Pi_{Q,Q^0_0} *_{\Z^{Q_0}} B. \qedhere
\end{equation}
\end{proof}
\subsection{Presentation of $\Pi_Q$ and $\Lambda_Q$ for star-shaped quivers}
We will need the following notation, which makes sense not just for star-shaped quivers but for
tree-shaped quivers (i.e., quivers such that, forgetting the orientations of arrows, one obtains a tree):
\begin{ntn} For two vertices $i,j$ in a tree-shaped quiver $Q$, let
  $d_{i,j}$ denote their distance, i.e., the minimum number of arrows
  that must be traversed in $\dq$ to go from $i$ to $j$. Let
  $p_{i,j} \in P_{\dq}$ be the unique path of length $d_{i,j}$ from $i$ to
  $j$.
\end{ntn}
Note that, in the case that $i$ and $j$ lie in an external line
segment of the tree $Q$, then the above notation is consistent with
Notation \ref{n:pijdefn}.

From Corollary \ref{quivsfcc}
and Lemma \ref{qhatijl} we immediately deduce the following general result:
\begin{prop} \label{starprop}
Let $Q$ be a star-shaped quiver with branches $L_{k}$ of lengths $d_{k}$, for
$k \in \{1,\ldots,m\}$, and special vertex $i_s \in Q_0$. Then:
\begin{enumerate}
\item[(i)] $i_s \Pi_Q i_s \cong A := \Z\langle x_1,\ldots,x_m \rangle / \ldp x_1+\cdots+x_m,x_1^{d_1+1},\ldots,x_m^{d_m+1} \rdp$.
\item[(ii)] For every pair of vertices $j_1, j_2$ in branches
  $L_{k_1}, L_{k_2}$, respectively,
  $i \Pi_Q j \cong A / \langle x_{1}^{d_{k_1} - d_{i_s,j_1}+1} A + A
  x_{2}^{d_{k_2} - d_{i_s,j_2}+1} \rangle$.
\item[(iii)] $\Lambda_Q \cong A/([A,A] + \langle x_{k}^{q} \rangle_{1 \leq k \leq m, q \geq 1})$.
\end{enumerate}
\end{prop}
Note that the torsion structure of $\Lambda_Q$, as a special case of
all quivers $Q$, is the main result of this article, but this cannot
be obviously deduced for star-shaped quivers using only the above
proposition.

\section{Type $\tilde E_n$ quivers} \label{ens}


Specializing Proposition \ref{starprop} to the case of quivers of type
$\tilde E_n$, and using noncommutative Gr\"obner generating sets
(cf.~Appendix \ref{gbs} and Proposition \ref{gbp} therein), we obtain
the following bases of $\Pi_Q$:
\begin{prop} \label{enprop}
Let $Q$ be an extended Dynkin quiver of type $E$.
Let $A$ be defined as in Proposition \ref{starprop}.(i).  For
readability, set $x := x_1, y := x_2, z := x_3$, assume $d_1 \geq d_2
\geq d_3$, and set $d := d_1$.
\begin{enumerate}
\item[(i)] A basis for $A \cong i_s \Pi_Q i_s$ (as a free $\Z$-module) is given by the elements, for all $0 \leq \ell_1 < d$, $\ell_2 \geq 0$, and $\ell_3 \geq 0$, 
\begin{gather}
x^{\ell_1} (y x^{d})^{\ell_2} (y x^{d-1})^{\ell_3} \cdot Y, \quad Y \in \{yx^{d-2} y, yx^{\ell_4}, 1\}_{0 \leq \ell_4 \leq d-2}.  \\
\notag \text{ (i.e., $Y$ is an initial subword of $yx^{d-2} y$).}
\end{gather}
\item[(ii)] A basis for $i_0 \Pi_Q i_0$, (as a free $\Z$-module), 
via Proposition \ref{starprop}.(ii), is given by
\begin{gather}
i_0, \\
p_{i_0 i_s} \bigl((y x^d)^{\ell} y \bigr) p_{i_s i_0} \quad(\ell \geq 0), \\
p_{i_0 i_s} \bigl((y x^d)^{\ell_1} (y x^{d-1})^{\ell_2} y x^{d-2} y\bigr) p_{i_s i_0} \quad(\ell_1,\ell_2 \geq 0).
\end{gather}
\end{enumerate}
\end{prop}
This proposition follows from a computation of Gr\"obner generating
sets that we performed with Magma for the $\tilde E_6, \tilde E_7$,
and $\tilde E_8$ cases separately. We give the details in the
following four subsections.

\subsection{Type $\tilde E_6$} \label{pie6s} We first compute a basis
for the ring $A := \Z\langle x,y,z \rangle / \ldp x^3,y^3,z^3,x+y+z
\rdp \cong i_s \Pi i_s$ (cf.~Proposition \ref{starprop}).  We do this
by computing a Gr\"obner generating set for the ideal $\ldp
x^3,y^3,z^3,x+y+z\rdp \subset \Z\langle x,y,z \rangle$, which we can
do by computer using Buchberger's algorithm (we used Magma).  It is
also straightforward to verify by hand that the given elements form a
Gr\"obner generating set.
\begin{prop}
In the graded lexicographical order with $x \prec y \prec z$,  the Gr\"obner generating set for
the ideal $\ldp x^3,y^3,z^3,x+y+z \rdp \subset \Z\langle x,y,z \rangle$ is
\begin{gather} \label{e6mag}
    yxyx^2 - xyxyx + x^2yxy + x^2yx^2, \\
    y^3, \\
    y^2x + yxy + yx^2 + xy^2 + xyx + x^2y, \\
    x^3, \\
    z + y + x. 
\end{gather}
\end{prop}
By definition of Gr\"obner generating sets we immediately deduce Proposition
\ref{enprop}.(i).  We will prove Proposition \ref{enprop}.(ii) for
$E_6, E_7$, and $E_8$ simultaneously in \S \ref{ss:pf-enprop2}.

As a result of Proposition \ref{enprop}.(ii), we deduce, where
$\Phi_m(t)$ denotes the cyclotomic polynomial whose roots are the
primitive $m$-th roots of unity:
\begin{cor}
\begin{equation}
h(i_0 \Pi i_0;t) = \frac{\Phi_{24}(t)}{(1-t^4)(1-t^6)}=\frac{1-t^4+t^8}{(1-t^4)(1-t^6)} =
\frac{(1-t^{24})}{(1-t^6)(1-t^8)(1-t^{12})}.
\end{equation}
Furthermore, one has the following partial fraction decomposition:
\begin{equation} \label{pfd}
h(i_0 \Pi i_0) (1-t^2) = 1 + \frac{2t^6}{1-t^6} + \frac{t^4}{1-t^4} - \frac{t^2}{1-t^2}.
\end{equation}
\end{cor}
\begin{proof} For the first part, we note that our basis above shows
  that $h(i_0
  \Pi i_0;t) = 1 + \frac{t^6}{(1-t^6)} +
  \frac{t^8}{(1-t^6)(1-t^4)}.$ Putting this over the common
  denominator $(1-t^4)(1-t^6)$,
  we get a numerator of $1-t^6-t^4+t^6-t^{10}+t^{10}+t^8=1-t^4+t^8
  = \Phi_{24}(t)$.

  For the second part, one may explicitly verify the identity.  Note
  that, since $(1-t^6)(1-t^4) =
  (1+t^2+t^4)(1+t^2)(1-t^2)^2$ is a decomposition into relatively
  prime factors, one sees that $h(i_0 \Pi i_0)
  (1-t^2)$ must have a partial fraction decomposition with
  denominators $1-t^6, 1-t^4, 1-t^2$, and the above is one such.
\end{proof}

The meaning of the partial-fraction decomposition \eqref{pfd} is
again the identity \eqref{egid} (cf.~\S \ref{ean}): 
setting $h(i_0 \Pi i_0) = \sum a_m t^m$, 
\begin{equation} \label{pfide6}
a_m - a_{m-2} = 2 [6 \mid m] + [4 \mid m] - [2 \mid m], \quad m \geq 2.
\end{equation}
This bears similarity to the determinant of the $t$-analogue of the Cartan matrix:
\begin{equation}
\text{det}(\1 - t \cdot C + t^2 \cdot \1) = \frac{(1-t^6)^2(1-t^4)}{1-t^2}.
\end{equation}
Indeed, \eqref{pfide6} says that
\begin{equation}
\prod_{m \geq 1} \frac{1}{(1-t^{m})^{a_m-a_{m-2}}} = \prod_{m \geq 1} \frac{1-t^2}{(1-t^6)^2(1-t^4)},
\end{equation}
Then, as in \S \ref{ean}, we verify \eqref{egid} in this case.

\subsection{Type $\tilde E_7$} \label{pie7s}
We first compute a Gr\"obner generating set for the ideal $\ldp x^4,y^4,z^2,x+y+z \rdp
\subset \Z\langle x,y,z
\rangle$ (cf.~Proposition \ref{starprop}), which we can do with Magma.  It is also
straightforward to verify by hand that the given elements form a Gr\"obner generating set. 
\begin{prop}
In the graded lexicographical order with $a \prec b \prec c$,  the Gr\"obner generating set for
the ideal $\ldp x^4,y^4,z^2,x+y+z \rdp \subset \Z\langle x,y,z \rangle$ is
\begin{gather} \label{e7mag}
    yx^2yx^3 - xyx^2yx^2 + x^2yx^2yx - x^3yx^2y, \\
    yxyx + yx^2y + yx^3 + xyxy + x^3y, \\
    x^4, \\
    y^2 + yx + xy + y^2, \\
    z + y + x.
\end{gather}
\end{prop}
We immediately deduce Proposition \ref{enprop}.(i) for $E_7$. See the
next subsection for the proof of (ii).  As a result of (ii) we deduce the
\begin{cor}
One has the formula
\begin{equation}
h(i_0\Pi_{\tilde E_7}i_0;t) = \frac{\Phi_{36}(t)=1-t^{6}+t^{12}}{(1-t^6)(1-t^8)} = \frac{1-t^{36}}{(1-t^8)(1-t^{12})(1-t^{18})}.
\end{equation}
Additionally, one has the partial fraction decomposition
\begin{equation}
h(i_0 \Pi_{\tilde E_7}i_0;t) (1-t^2) = 
1 + \frac{t^4}{1-t^4} + \frac{t^6}{1-t^6} + \frac{t^8}{1-t^8} - \frac{t^2}{1-t^2}.
\end{equation}
Since, letting $C$ be the adjacency matrix of $\overline{\tilde E_7}$, one has the formula
\begin{equation}
\det (1 - t \cdot C + t^2 \cdot \1) = \frac{(1-t^4)(1-t^6)(1-t^8)}{1-t^2},
\end{equation}
the identity \eqref{egid} is verified.
\end{cor}

\subsection{Type $\tilde E_8$} \label{pie8s}
We first compute a Gr\"obner generating set for the ideal $\ldp x^6,y^3,z^2,x+y+z \rdp
\subset \Z\langle x,y,z
\rangle$ (cf.~Proposition \ref{starprop}), which we can do with Magma.  It is also
straightforward to verify by hand that the given elements form a Gr\"obner generating set. 
\begin{prop}
In the graded lexicographical order with $x \prec y \prec z$,  the Gr\"obner generating set for
the ideal $\ldp x^6,y^3,z^2,x+y+z \rdp \subset \Z\langle x,y,z \rangle$ is
\begin{gather} \label{e8mag}
    yx^4yx^5 - xyx^4yx^4 + x^2yx^4yx^3 - x^3yx^4yx^2 + x^4yx^4yx - x^5yx^4y, \\
    yx^3yx + yx^4y + yx^5 + xyx^3y - x^2yx^3 - x^3yx^2 + x^5y, \\
    x^6, \\
    yx^2y + yx^3 + xyx^2 + x^2yx + x^3y + 2x^4, \\
    yxy - xyx - x^3, \\
    y^2 + yx + xy + x^2, \\
    z + y + x.
\end{gather}
\end{prop}
We immediately deduce Proposition \ref{enprop}.(i) for the $E_8$ case,
which completes the proof of that part. 

\begin{cor}
The Hilbert series of $i_0 \Pi_{\tilde E_8} i_0$ is given by
\begin{equation}
h(i_0 \Pi_{\tilde E_8} i_0) = \frac{\Phi_{12}(t)\Phi_{60}(t) = 1 - t^{10} + t^{20}}{(1-t^{10})(1-t^{12})}=
\frac{(1-t^{60})}{(1-t^{12})(1-t^{20})(1-t^{30})}.
\end{equation}
Additionally, one has the partial fraction decomposition
\begin{equation}
h(i_0 \Pi_{\tilde E_8} i_0) (1-t^2) = 
1 + \frac{t^4}{1-t^4} + \frac{t^6}{1-t^6} + \frac{t^{10}}{1-t^{10}} - \frac{t^2}{1-t^2}.
\end{equation}
Since, letting $C$ be the adjacency matrix for $\overline{\tilde E_8}$, one has the formula
\begin{equation}
\det (1 - t \cdot C + t^2 \cdot \1) = \frac{(1-t^4)(1-t^6)(1-t^{10})}{1-t^2},
\end{equation}
the identity \eqref{egid} is verified.
\end{cor}

\subsection{Proof of Proposition \ref{enprop}.(ii)}\label{ss:pf-enprop2}
We prove this proposition simultaneously for all cases $\tilde E_6, \tilde E_7$, and $\tilde E_8$.

By Lemma \ref{qhatijl}, the map $i_s \Pi
i_s \rightarrow i_0 \Pi i_0, f \mapsto p_{i_0 i_s} f p_{i_s i_0}$ has
image $(i_0 \Pi i_0)_+$ and kernel integrally spanned by elements of
the form $xg, gx$, $g \in i_s \Pi i_s$.  First, the basis elements
that begin or end in $x$ are killed. Conversely, for any basis element
$g$, expressing $xg$ in terms of basis elements consists only of terms
beginning with $x$.  On the other hand, $gx$ might not consist only of
terms beginning or ending with $x$: this can happen in the case that
$g = (yx^d)^{\ell_1} (yx^{d-1})^{\ell_2} yx^{d-2} y$.  In this case,
$gx$ is a sum of basis elements beginning or ending in $x$, and the
element $(yx^d)^{\ell_1} (yx^{d-1})^{\ell_2+1}$.  Hence, we can simply
eliminate the latter element from our basis, and we obtain the stated
result and hence the proposition.

\subsection{Poisson structure on $i_0 \Pi_Q i_0$} \label{eans}
As in \SS \ref{lans} and \ref{dans}, Theorem \ref{dedz} (which does
not rely on any results of this subsection) implies that
$i_0 \Pi_Q i_0$ can be given a Poisson algebra structure for $Q$ of
types $\tilde E_6$, $\tilde E_7$, or $\tilde E_8$, by
\eqref{e:i0pq-br}.  Here we give a description of this Poisson
structure for
$Q = \tilde E_n$ over $\Z$ using the bases of
Proposition \ref{enprop}.  Note that, by Remark \ref{r:rkp}, this also
entirely describes the Lie structure on $\Lambda_Q$.
\begin{prop} \label{epbp} 
The Poisson structure of $i_0 \Pi_Q i_0$ is given as follows for $Q = \tilde E_n$.
Let us use the notation of Proposition \ref{enprop}, and define
\begin{equation}\label{e:XYZ-en}
X := p_{i_0 i_s} y p_{i_s i_0}, \quad Y :=  p_{i_0 i_s} yx^{d-2} y p_{i_s i_0}, \quad Z := p_{i_0 i_s}
y x^{d-1} y x^{d-2} y p_{i_s i_0}.
\end{equation}
Finally, let us assume that all arrows of the quiver $Q$ are oriented towards the special
vertex.
\begin{enumerate}
\item[(i)] For $Q = \tilde E_6$,  $i_0 \Pi_Q i_0 \liso \Z[X,Y,Z] / \ldp Z^2 + Y^3 + ZX^2 \rdp$, and
\begin{equation} \label{e6pb}
\{X,Y\} = -2Z - X^2, \quad \{X,Z\} = 3Y^2, \quad \{Y,Z\} = -2XZ.
\end{equation}
\item[(ii)] For $Q = \tilde E_7$,  $i_0 \Pi_Q i_0 \liso \Z[X,Y,Z] / \ldp Z^2-X^3 Y + Y^3 \rdp$, and
\begin{equation}
\{X,Y\} = -2 Z, \quad \{X,Z\} = 3Y^2 - X^3, \quad \{Y,Z\} = 3 X^2 Y.
\end{equation}
\item[(iii)] For $Q = \tilde E_8$,  $i_0 \Pi_Q i_0 \liso \Z[X,Y,Z] / \ldp Z^2 + X^5 + Y^3 \rdp$, and
\begin{equation}
\{X, Y\} = - 2 Z, \quad \{X,Z\} = 3 Y^2, \quad \{Y,Z\} = - 5 X^4.
\end{equation}
\end{enumerate}
In particular, the homogeneous elements $X, Y, Z$ generate $i_0 \Pi_Q i_0$ as a graded algebra, and after tensoring
over any commutative ring containing $\frac{1}{2}$, $X, Y$ generate
as a Poisson algebra.  The degrees are $|X| = 2(d+1), |Y| = 4d$, and $|Z| = 6d$.
\end{prop}
\begin{proof}
  This can all be verified by an explicit computation. To see that the
  given relation (e.g., $Z^2 + Y^3 + ZX^2$ for $\tilde E_6$) is the
  only relation, we can also use the Hilbert series computed in
  \cite{EG} (cf.~the first formula of \eqref{egfla}, which holds for
  extended Dynkin quivers). To compute the relation and the necklace
  bracket formulas, we used Magma, but it is also tractable by hand
  (and only requires computations in fairly low degrees).
\end{proof}
We note that, in the above, the presentation in (i) of $\tilde E_6$ is somewhat more inconvenient
than the presentations of $\tilde E_7, \tilde E_8$ since the latter have the form $Z^2 = P(X,Y)$ for
polynomials $P$ in $X,Y$ (i.e., $i_0 \Pi_Q i_0$ is a direct sum of the part with odd degree in $Z$
and the part with even degree).  Over $\Z[\frac{1}{2}]$, one can fix this:
\begin{gather}
Z' := Z + \frac{1}{2} X^2, \quad i_0 \Pi_{\tilde E_6} i_0[\frac{1}{2}] \liso \Z[\frac{1}{2}][X,Y,Z'] / \ldp 
(Z')^2 + Y^3 + \frac{1}{4} X^4 \rdp, \\
\{X,Z'\} = 3 Y^2, \quad \{Y,Z'\} = X^3, \quad \{X,Y\} = -2 Z'.
\end{gather}

Finally, we explicitly compute the zeroth Poisson homology of $i_0 \Pi_{Q^0} i_0 \o \F_p$ where
$p$ is a bad prime for $Q^0$, which will be needed to finish the proof of the main theorem. Let $M(m)$ denote the shift of a graded $\Z$-module $M$ by degree $m$, i.e., $M(m)_n=M_{n+m}$.
\begin{prop} \label{ezphp} 
Let $(Q^0,p)$ be one of the seven exceptional cases $(\tilde E_6, 2)$, $(\tilde E_6, 3)$,
$(\tilde E_7, 2)$, $(\tilde E_7, 3)$, $(\tilde E_8, 2)$, $(\tilde E_8, 3)$, $(\tilde E_8, 5)$.  
Let $A := i_0 \Pi_{Q^0} i_0$ be
the Poisson algebra, and set $A_p := A \o \F_p$ and $A_\Q := A \o \Q$. 
If $(Q^0, p) \neq (\tilde E_8, 2)$, then the zeroth Poisson homology
$HP_0(A_p) = A_p / \{A_p, A_p\}$ is given, as a graded $A_p^{(p)} := \langle f^p \rangle_{f \in A_p}$-module, by
\begin{equation}
HP_0(A_p)  \cong (A_p^{(p)})_+(2-2p) \oplus HP_0(A_p)',
\end{equation}
where $HP_0(A_p)'$ is an $A_p$-module of the form $U \otimes \F_p$,
where $\F_p$ is the augmentation module, and $U$ is a graded
$\F_p$-module with Hilbert series $\leq h(HP_0(A_\Q);t)$.  Moreover, a
basis of $HP_0(A_p)'$ can be taken to be the image of a subset of
$A$ which projects to linearly independent classes in $HP_0(A_\Q)$.

If $(Q^0, p) = (\tilde E_8, 2)$, then we have a correction, related to
the torsion class in $(\Lambda_{Q^0})_{28}$:
\begin{equation}
HP_0(A_2) \cong (A_2^{(2)})_+/\langle X^2 \rangle
\oplus HP_0(A_2)', \quad X = p_{i_0 i_s} y p_{i_s i_0}.
\end{equation}
\end{prop}
\begin{proof}
This is done on a case-by-case basis. In all cases, we can assume that the quiver is oriented
as in Proposition \ref{epbp}, because for any $a \in Q_1^0$, letting $(Q^0)_a$ be the quiver with the
arrow $a$ reversed, there is an isomorphism
$\Pi_{Q^0} \iso \Pi_{(Q^0)_a}$ sending $a$ to $-a$ and fixing all other arrows and all 
vertices, which induces a Poisson isomorphism $i_0 \Pi_{Q^0} i_0 \iso i_0 \Pi_{(Q^0)_a}$.

Now, we show how to prove the proposition for the case $(\tilde E_6, 2)$, and omit the other six 
cases (which are all similar). Note first  that, by Proposition \ref{epbp}.(i), a basis of $i_0 \Pi_{Q^0} i_0$
is given by $\{ X^a Y^b,  X^a Y^b Z\}_{a, b \geq 0}$. We wish to compute the image of the Poisson
bracket explicitly.  As in the commutator case (in Lemma \ref{mcl}), we may use the formula
\begin{equation}
\{fg, h\} = f\{g, h\} + g\{f, h\} = \{f, gh\} + \{g, fh\}
\end{equation}
to reduce to computing Poisson brackets of the form $\{X, f\}, \{Y, f\},$ and $\{Z, f\}$. 
Next, by \eqref{e6pb}, we may compute, now over
$\F_2$:
\begin{gather}
\{x, x^a y^b\} = b x^{a+2} y^{b-1}, \quad \{y, x^a y^b\} = a x^{a+1} y^b, \quad \{z, x^a y^b\} = a x^{a-1} y^{b+2}, \\
\{x, x^a y^b z\} = b x^{a+2} y^{b-1} z + x^a y^{b+2}, \quad \{y, x^a y^b z\} = a x^{a+1} y^b z, \quad
\{z, x^a y^b z\} = a x^{a-1} y^{b+2} z.
\end{gather}
Now, it follows from this that the image of the Poisson bracket is integrally spanned by the elements
\begin{gather}
x^a y^b, x^a y^b z, \quad \text{where $a$ or $b$ is even, and $a \geq 2, b \geq 0$}, \\
x^{2a+1} y^{2b} (y^3 + x^2 z), \quad a,b \geq 0,  
x y^{2b+2}, y^{2b+2}, y^{2b+2} z, \quad a,b \geq 0.
\end{gather}
In other words, the following gives a basis of $HP_0(A_2)$:
\begin{equation}
xy(x^{2a} y^{2b}), xyz(x^{2a} y^{2b}), xz(y^{2b}), yz, z, y, x, 1.
\end{equation}
It is easy to see that the last five elements are killed by multiplication by $A_2^2$. The
first three elements form a $A_2^2$-submodule isomorphic to $((A_2)_+)^2(-2)$, which can
be seen by formally adjoining an element $r'$ in degree two and setting $r'x^2=xy, r'y^2=xz,$ and
$r'z^2 = xyz$.

Finally, we note that $HP_0(A_\Q) \cong \langle xy,yz,z,y,x,1
\rangle$, which is known and can be computed directly.\footnote{As
  pointed out by P.~Etingof, by a spectral sequence argument, whenever
  $V$ is a finite-dimensional complex symplectic vector space and $G$
  a finite subgroup of symplectic automorphisms, the dimension of
  $\mathit{HP}_{0}(\mathbb{C}[V]^G)$ is at least the dimension of the
  zeroth Hochschild homology of the Weyl algebra of $V$ smashed
  with~$G$, which by~\cite{AFLS} is the number of conjugacy classes of
  $G$ which do not have $1$ as an eigenvalue. In our case, this is one
  less than the number of vertices of the quiver, and one in fact has
  equality. The equality does not hold in general, but in the appendix
  to~\cite{BEG}, $\mathit{HP}_{0}(\mathbb{C}[V]^G)$ is at least shown to always
  be finite-dimensional.}
\end{proof}

\section{Proof of Theorem \ref{wdescpt} and hence Theorem \ref{mt}}\label{s:mt-fp}
In this section we complete the proof of the main results.  First we
prove a result, Theorem \ref{dedz}, which explicitly describes the
torsion of $\Lambda_Q$ in the Dynkin and extended Dynkin cases,
which is interesting in its own right.  Note that, in the Dynkin case,
$\Lambda_Q$ is entirely torsion (and finite), so this actually
describes all of $\Lambda_Q$ in that case.

\subsection{The torsion of
$\Lambda_Q$ in the extended Dynkin and Dynkin cases}
This subsection is devoted to the following result:
\begin{thm}\label{dedz} For any Dynkin quiver $Q$ with
corresponding extended Dynkin quiver $\tilde Q \supset Q$, with
extending vertex $i_0 \in \tilde Q_0$, one
has the split exact sequence
\begin{equation}\label{des}
i_0 \Pi_{\tilde Q} i_0 \into \Lambda_{\tilde Q} \onto \Lambda_Q, \quad
\Lambda_{\tilde Q} \cong i_0 \Pi_{\tilde Q} i_0 \oplus \Lambda_Q,
\end{equation}
using the natural maps.  Furthermore, $i_0 \Pi_{\tilde Q} i_0$ is a free $\Z$-module, and $\Lambda_Q$ is finite, given as follows:
\begin{enumerate}
\item $(\Lambda_{A_n})_+ = 0$.
\item $(\Lambda_{D_n})_m = 0$ for all $m \geq 1$ except when
  $4 \mid m$ and $m \leq 2(n-2)$, in which case
  $\Lambda_m \cong \Z/2$, and is integrally spanned by the class
  $[i_s(xy)^{m/4}]$, where $i_s$ is the special vertex and $x,y$ are
  the length-two cycles along the branches of length one (to the
  endpoint and back). This class equals $[r^{(m/2)}]$ when $m$ is a
  power of $2$.

  In $\Lambda_{\tilde D_n}$, the above class lifts to the torsion
  class $[i_{RU}(LR)^{m/2}] + [i_{LU}(RL)^{m/2}]$, which may also be
  described as
  $X^{m/4-1} \cup r^{(2)} := \Bigl[X^{m/4-1}
  \widetilde{r^{(2)}}\Bigr]$
  for $X \in HH^0(\Pi_{\tilde D_n})$ as defined in Proposition
  \ref{dpbp}, and $\widetilde{r^{(2)}} \in \Pi_{\tilde D_n}$ any
  class such that $\Bigl[\widetilde{r^{(2)}}\Bigr] = r^{(2)}$
  (cf.~Proposition \ref{epbp}).
\item  $(\Lambda_{E_6})_m = 0$ for all $m \geq 1$ except for $m \in \{4,6\}$,
where $(\Lambda_{E_6})_4 \cong \Z/2$ and $(\Lambda_{E_6})_6 \cong \Z/3$, 
generated by the classes
$r^{(2)}$ and $r^{(3)}$.
\item $(\Lambda)_{E_7})_{m} = 0$ for all $m \geq 1$ except when $m$ is
  $4,6,8$, or $16$, where we get $\Z/2, \Z/3, \Z/2$, and $\Z/2$,
  respectively, generated by the classes $r^{(2)}, r^{(3)}, r^{(4)},$
  and $r^{(8)}$.
\item $(\Lambda_{E_8})_{m} = 0$ for all $m \geq 1$ except when $m$ is
  $4,6,8,10$, $16$, $18$, or $28$, where we get $\Z/2,\Z/3,\Z/2,\Z/5,$
  $\Z/2$, $\Z/3$, and $\Z/2$, respectively, generated by the classes
  $r^{(2)}, r^{(3)}$, $r^{(4)}, r^{(5)},$ $r^{(8)}$, $r^{(9)}$, and
  $[i_s x^4 y x^4 y x^3 y]$, where $x$ is a length-two cycle in the
  direction of the longest branch, and $y$ is a length-two cycle in
  the direction of the second-longest branch, and $i_s$ is the special
  vertex.

  In $(\Lambda_{\tilde E_8})_{28}$, the last class above lifts to the
  class $X^2 \cup r^{(2)} = \Bigl[X^2 \widetilde{r^{(2)}}\Bigr]$,
  where $X$ is defined in \eqref{e:XYZ-en}.
\end{enumerate}
The classes above also lift to classes which generate the torsion
of $\Lambda_{\tilde Q}$, which is isomorphic to $\Lambda_Q$  under the second
map of \eqref{des}.   Explicitly, these classes are the
corresponding $r^{(p^\ell)}$, and in the $\Lambda_{\tilde D_n}$ and 
$(\Lambda_{\tilde E_8})_{28}$ cases, they are as noted above.
\end{thm}
\begin{proof}[Proof of Theorem \ref{dedz}]
  Since the cokernel of \eqref{des}, $\Lambda_Q$, is always torsion
  (i.e., $\Lambda_Q \o \Q = 0$ \cite{MOV}, which is also easy to check
  using Gr\"obner generating sets), and the free rank of the first two terms in
  \eqref{des}, $i_0 \Pi_{\tilde Q} i_0$ and $\Lambda_{\tilde Q}$, are
  the same in each degree \eqref{edisos}, the map $i_0 \Pi_{\tilde Q}
  i_0 \rightarrow \Lambda_{\tilde Q}$ is always a monomorphism.  It
  remains only to compute $\Lambda_Q$ where $Q$ is Dynkin, and check
  that the above descriptions indeed give a splitting $\Lambda_Q \into
  \Lambda_{\tilde Q}$.  We first explain the second step (the
  splitting).  We define the splitting on the torsion classes
  $r^{(p^\ell)}$ by $r^{(p^\ell)} \mapsto r^{(p^\ell)}$. This
  obviously is well-defined.  Next, in the $D_n$ case, it is easy to check
  that, in $\Lambda_{\tilde Q}$, and hence also in $\Lambda_{Q}$, ($[i_{RU}(LR)^{m/2}] +
  [i_{LU}(RL)^{m/2}]$) is two-torsion using \eqref{wc0rel} for $c =
  \frac{m}{2}$. Thus we can define the splitting to send the corresponding class in
 $\Lambda_Q$ to that in $\Lambda_{\tilde Q}$.
  For both the $D_n$ and $E_8$ cases, the cup-product
  formula for the remaining elements of $\Lambda_Q$ can be checked using Magma; the same formula defines
  two-torsion elements of $\Lambda_{\tilde Q}$, so we can define the splitting similarly to send the corresponding
  element of $\Lambda_Q$ to that of $\Lambda_{\tilde Q}$.

  It remains to verify the formulas for the torsion of $\Lambda_Q$
  where $Q$ is Dynkin.  This is a finite computation, since $\Pi_Q$ is
  finite-dimensional. We do this directly as follows: (1) We note that
  here $i \Pi_{A_n} i$ is one-dimensional for each $i$, corresponding
  to the path of length zero. Thus $\Lambda_{A_n}$ is torsion-free
  (and isomorphic to $i \Pi_{A_n} i$ as a graded $\Z$-module). 

  (2) Consider the quiver $D_n$ obtained from $\tilde D_n$ by cutting
  the vertex $i_{LD}$.  In this case, $\Lambda_{D_n}$ is generated by
  closed paths that sit entirely on the two rightmost external arrows
  (we can apply relations of $\Pi$ and cyclic rotations to move any
  cycle of length two there).  Let $i$ be the rightmost internal
  vertex. Then, such classes are zero if they are not a power of the
  length-four cycle of the form $iR_U L R_D L$.  Furthermore, the
  additional relations that one obtains are: (i) any closed path that
  does not touch the rightmost external arrows is zero; (ii) $0 =
  [i(RL)^m] = [i((R_U+R_D)L)^m]$ for all $i$ and $m$; and (iii)
  $i(RL)^{n-3} = 0$ (the latter is an algebra relation from
  $\Pi_{D_n}$). The last condition says that any closed cycle along
  the rightmost external arrows of length $\geq 2(n-2)$ is zero.  The
  second condition says that $[(R_U L R_D L)^m] + [(R_D L R_U L)^m] =
  0$, which combined with the fact that these two are equal (which was
  already true in $\Lambda_{\tilde D_n}$), shows that $[(L_U R L_D
  R)^m]$ generates a copy of $\Z/2$.  Finally, it is straightforward
  to check that the class $r^{(2m)}$ reduces to this class when
  $m$ is a power of two.

  (3),(4),(5) Using Proposition \ref{starprop}, we may give an
  explicit presentation of $i_s \Pi_{E_n} i_s$ and $\Lambda_{E_n}$ for
  $n \in \{6,7,8\}$; this immediately cuts off the degree at $\leq 10$
  for $\tilde E_6$, $\leq 16$ for $\tilde E_7$, and $\leq 28$ for
  $\tilde E_8$ (these numbers may also be computed using the Coxeter
  numbers of the corresponding root systems: see e.g.~\cite{MOV,
    EE2}).  We completed the remaining finite computation by hand with
  the help of Magma: we computed the torsion abstractly by hand using
  Gr\"obner generating sets, and checked with Magma that the given
  classes generate the torsion. (We also double-checked the Hilbert
  series over finite fields including $\F_2, \F_3, \F_5$ in Magma.)
\end{proof}
Theorem \ref{dedz}, together with the explicit bases we already gave for
$i_0 \Pi i_0$ of an extended Dynkin quiver, gives an explicit basis of
$\Lambda$ for all extended Dynkin quivers.

\subsection{Proof of Theorem \ref{wdescpt} (completing the proof of Theorem
\ref{mt})} \label{wdescptpfs}
\begin{proof}[Proof of Theorem \ref{wdescpt}] Recall that we are working over $\Z_{\ldp p \rdp}$; e.g.,
when we write $\Pi_Q$ we mean $\Pi_Q \otimes \Z_{\ldp p \rdp}$.

First, in the cases that $(Q^0, p)$ is a good pair ($p \nmid |\Gamma|$
where $\Gamma \subset SL_2(\C)$ is the finite group associated to
$Q^0$), the result follows from Theorem \ref{mclgam}: it remains only
to verify the statements concerning $p$-th powers and image of Poisson
bracket.  The statements about $p$-th powers follow from
the form of the relations $W_{a,b}$ appearing in Theorem~\ref{mclgam} (these elements should not be confused with the submodules of $W$ appearing Theorem~\ref{wdescpt}); in particular, it follows from \eqref{wabdefn1}--\eqref{wabdefn3} that:
\begin{equation}
W_{pa,pb} \equiv W_{a,b}^p \pmod p, \quad \frac{1}{p}W_{p^\ell,p^\ell} \equiv (\frac{1}{p}
W_{p,p})^{p^{\ell-1}} \pmod p.
\end{equation}
The final statement about Poisson bracket (iv) then follows
immediately from Proposition \ref{ubrp}.

Next, in the cases that $Q^0 = \tilde A_n$ or $Q^0 = \tilde D_n$, the results follow from
Theorems \ref{ant} and \ref{dnt}, by a similar analysis of the given relations.

In the seven remaining cases $\{(\tilde E_6,2),(\tilde E_6,3),(\tilde E_7,2),(\tilde E_7,3),
(\tilde E_8,2),(\tilde E_8,3),(\tilde E_8,5)\}$, we make use of the following claim, which
follows from Proposition \ref{ezphp} (as we will explain):
\begin{claim}
If $M_p$ is the image of the Poisson bracket on $i_0 \Pi_Q i_0 \otimes \F_p$, then
\begin{equation}
t^2 h(i_0 \Pi_Q i_0; t) = \sum_{\ell \geq 0} t^{2p^\ell} (h(M_p;t^{p^\ell}) + F(t^{p^\ell})),
\end{equation}
where $F(t)= h(HP_0(i_0 \Pi_Q i_0 \o \F_p)'; t) \leq h(HP_0(i_0 \Pi_Q i_0 \o \Q);t)$.   
\end{claim}
Let us first use this claim. By Proposition \ref{ubrp}, 
$r' \cdot M$ 
and
$[HH^0(\Pi_{Q^0}), i_0 \Pi_{Q^0} i_0] = [i_0 \Pi_{Q^0} i_0, i_0
\Pi_{Q^0} i_0]$ 
are equivalent mod $[\ldp \langle \dq_1 \setminus \dzqo \rangle \rdp^3]
+ [\ldp \langle \dq_1 \setminus \dzqo \rangle \rdp, \Pi_{Q}] + \ldp p \rdp$. Moreover,
this image lies in the image of $r' HH^0(\Pi_{Q^0})$, which we call
$[r' HH^0(\Pi_{Q^0})] \otimes \F_p$.
So, there exists a saturated free submodule
$\tilde M \subset W'$ which projects 
to $[r' M]$.  We will find a (saturated) $Z_{\ldp p \rdp}$-submodule $U'_p \subset
HH^0(\Pi_{Q^0})$ with Hilbert series $h(U'_p;t) = F(t) - t^{2p^m-2}$
(where $m$ is the smallest positive integer such that $r^{(p^m)}$ is
zero in $\Pi_{Q^0}$), such that the image of $U'_p$ and $M$ in
$HH^0(\Pi_{Q^0}) \otimes \F_p$ are linearly independent (over
$\F_p$).  We will then find a submodule $\tilde U'_p \subset W'$ which
projects isomorphically mod $[\ldp \langle \dq_1 \setminus \dzqo \rangle
\rdp^3] + [\ldp \langle \dq_1 \setminus \dzqo \rangle \rdp, \Pi_{Q}]$ to
$[r' U'_p]$.  In view of the observations before the statement of the
theorem (given any element of $W \otimes \F_p$, its $p$-th power must
also be in $W \otimes \F_p$), it then follows that $W' \otimes \F_p$
contains the direct sum of all $p$-th powers of $[r' U'_p] \otimes
\F_p$ and $[r' M] \otimes \F_p$.  We can set $W'_s := \tilde M \oplus \tilde U'_p$,
which by the claim  satisfies the conditions of (iv).

Once we also
describe $W_0$ and $W_r$ as claimed, the sum $U:=W_0\oplus (W_s' \oplus W_r')$
is
a graded
submodule $U \subseteq W$ which has the correct Hilbert series, i.e.,
such that $W/U$ is torsion.  Moreover, we obtain that $V/U$ is torsion
and generated by the classes $r^{(p^\ell)}$.  Since we know that these
classes are nonzero by Proposition \ref{rpnz} (the easy direction of
Theorem \ref{mt}), we conclude that $U = W$, which proves the theorem.

It remains to find the submodules $U'_p$ and $W_r$, and to describe $W_0$.
We begin with the description of $W_0$. This is easy: using Theorem
\ref{dedz}, it already has
the desired form, except in the case of $(\tilde E_8, 2)$, where it is
not immediately clear that $(W_0)_{28}$ is saturated.  However, this can
be verified explicitly (with the help of Magma): $[x^4 y, x^4 y x^3
y]$ projects to $2 x^4 y x^4 y x^3 y - x^5 y x^5 y = 2 x^4 y x^4 y x^3
y - [p_{i_s i_0} r' p_{i_0 i_s} y p_{i_s i_0} r' p_{i_0 i_s} y]$ in
$V$, and this is not a multiple of $2$ (working over $\Z$, the above
relation is in fact not a multiple of any positive integer).

Next, we find $U'_p$. For Hilbert series reasons, $U'_p$ should be
$\Z_{\ldp p \rdp}$-linearly spanned by elements of $i_0 \Pi_Q i_0$
which project to a subset of the generators of the zeroth Poisson
homology of $i_0 \Pi_Q i_0 \o \Q$ (note that $i_0 \Pi_Q i_0 \o \C
\cong \C[x,y]^\Gamma$).  The degrees of elements of $U'_p$ are $\leq
20$ for $\tilde E_6$, $\leq 32$ for $\tilde E_7$, and $\leq 56$ for
$\tilde E_8$. For these low degrees, the needed $\tilde U'_p \subset
W'$ can be found (or its existence verified) using Magma.

Finally, let us find $W_r$.  For each $\ell \geq 1$, define $f_\ell =
p \cdot \bar f_\ell$, such that $\bar f_\ell$ is any noncommutative
polynomial in $x,y$ which projects, modulo commutators, to
$\frac{1}{p} [(x+y)^{p^\ell} - x^{p^\ell} - y^{p^\ell}]$.  Our choice
of $f_\ell = p \bar f_\ell$ makes it clear that, in fact,
$r^{(p^\ell)}$ and $\bar f_\ell$ have the same image in $V/W$.  In
particular, $p [\bar f_\ell] = [f_\ell] \in W$ and $[f_\ell] \in W'$
for $\ell \geq m$. Thus, it suffices to show that $\bar f_\ell$ is
independent of $W_s' \o \F_p$ for all $\ell \geq m$; to do this it
suffices to show it for $\ell = m$.  In these cases, $W'_{|\bar f_m|}
\o \Q \cong HH^0(\Pi_{Q^0} \o \Q)(2)_{|\bar f_m|}$ is one-dimensional;
thus, it suffices to show that $\bar f_m \notin pV + [\ldp r'
\rdp^2]$.  This we can explicitly verify with Magma.

It remains only to prove the claim.  In general, $G(t) = \sum_{\ell
  \geq 0} H(t^{p^\ell})$ if and only if $H(t) = G(t) - G(t^p)$.  This
observation, together with Proposition \ref{ezphp}, proves the desired
result.
\end{proof}

\appendix
\section{The Diamond Lemma for modules} \label{dla} In this section,
we prove a generalization of the Diamond Lemma that applies to free
modules, rather than to free algebras (as in \cite{B}).  As a
consequence, we deduce results on Gr\"obner generating sets for free
algebras (which are probably known).  In the case of fields, one
recovers usual Gr\"obner bases and the Diamond Lemma of \cite{B}.  The
arguments used are essentially the same.

We first recall Gr\"obner generating sets, since they are simpler,
then proceed with the generalized Diamond Lemma,
which implies the result on Gr\"obner generating sets.

\subsection{Gr\"obner generating sets} \label{gbs}
First, suppose that $F = \k \langle x_1, \ldots, x_n \rangle$ is the free
noncommutative algebra over the commutative ring $\k$ generated by indeterminates
$x_1, \ldots, x_n$.  Consider the
graded lexicographical ordering, which means
that $M_1 \prec M_2$ if either $|M_1| < |M_2|$ (where $|M|$ denotes the
length of a monomial $M$), or $|M_1| = |M_2|$ and $M_1 \ll M_2$ with
respect to the lexicographical ordering $\ll$ on $x_1, \ldots, x_n$ where
$x_1 < x_2 < \ldots < x_n$.\footnote{We can generalize this to replace
lexicographical ordering by any total ordering on monomials of a given
degree, and generalize the degree $|\ |$ to a weighted degree where each $x_i$
is assigned a positive integer, not necessarily $1$.}

Given a set of elements $P_i \in F$, another polynomial $P$ is said to
be reducible with respect to the $P_i$ if the leading monomial (with
respect to $\prec$) $LM(P)$ of $P$ contains as a subword the leading
monomial of one of the $P_i$'s.  Otherwise, $P$ is said to be
irreducible with respect to the $P_i$'s.  If $P$ is reducible, then a
reduction of $P$ is an element of the form $P - \lambda X P_i Y$ where
$X$ and $Y$ are monomials, $\lambda \in \k$, and $\lambda X \cdot
LM(P_i) \cdot Y$ is the leading term of $P$ (i.e., $X \cdot
LM(P_i)\cdot Y$ is the leading monomial of $P$, which appears with
coefficient $\lambda$).

Let us call an \textbf{ideal} generating set $\{P_i\}$ for an ideal $I
= \ldp P_i \rdp$
a \textbf{Gr\"obner generating set} if 
the leading monomial of each $P_i$ has coefficient which is a unit in
$\k$, 
and 
any polynomial
has a unique reduction to an irreducible element with respect to the $P_i$.  In other
words, the irreducible monomials form a basis of the
quotient $F/I$ as a free $\k$-module.

The following criterion is well-known (and is the basis for the
Buchberger algorithm for computing Gr\"obner generating sets):
\begin{prop}\label{gbp} A set $(P_i)$ forms a Gr\"obner generating set for
  $I = \ldp P_i \rdp$ if and only if, for all elements $P_i$ and $P_j$
  with leading terms $\lambda_i M_i$ and $\lambda_j M_j$, and all
  monomials $M$ such that $M = M_i X = Y M_j$ for some monomials $X$
  and $Y$, 
  one can reduce the element
  $\lambda_j P_i X - \lambda_i Y P_j$ to zero, meaning that there is a
  sequence of reductions, using the $P_{\ell}$, taking this element to
  zero.
\end{prop}
\begin{proof} 
This follows from the Diamond Lemma in the next subsection. 
\end{proof}


\subsection{The Diamond Lemma for modules} \label{dlss} Here, we
formulate and prove a version of the Diamond Lemma for free modules
over a arbitrary commutative ring $\k$, which we haven't seen in the
literature.  We then specialize to the free algebra case.

Let $\k$ be a commutative ring.  Given any free module $V$ over $\k$
with a fixed basis $(v_i)_{i \in I}$ labeled by a partially ordered
set $(I,\prec)$ which satisfies the descending chain condition
(meaning every strictly descending sequence in $I$ is finite), and
given any submodule $W \subseteq V$, the Diamond Lemma gives a
criterion for a set $S = (w_j)_{j \in J} \subseteq W$ (for some index
set $J$) to be a \emph{confluent spanning set}, which is a
generalization of the notion of Gr\"obner generating set, essentially meaning
that applying reductions in any order yields the same result.  In the
case $S$ is confluent, every element of $V$ has a unique reduction to
certain $\k$-linear combinations of the irreducible monomials
$(v_i)_{i \in I' \subseteq I}$ for a certain subset $I' \subseteq I$
determined by $S$. More precisely, for each $i \in I'$, there exists a
subset $R_i \subseteq \k$ so that every monomial has a unique
reduction to a finite sum of the form $\sum_{i \in I'} \lambda_i v_i$,
with $\lambda_i \in R_i$ for all $i$. (The subset $R_i \subseteq \k$
will be an arbitrary choice of representatives of the quotient $\k /
Y_i$ for some ideal $Y_i$.)

Let $i\preceq i'$ mean $i = i'$ or $i \prec i'$.  An element $w_j \in V$
defines a ``reduction'' if, writing
$w_j = \sum_{i \in I} \lambda_{ji} v_i$ (all but finitely many
$\lambda_{ji}$ are nonzero for each $j$), there exists a unique
$\psi(j) \in I$ such that $\lambda_{j \psi(j)} \neq 0$ and
$\lambda_{j i} \neq 0$ implies $i \preceq \psi(j)$.  Then, the
``reduction'' associated to $w_j$ sends
$\lambda_{j \psi(j)} v_{\psi(j)}$ to
$-\sum_{i \prec \psi(j)} \lambda_{ji} v_i$.

For every collection $S = (w_j)_{j \in J} \subseteq W$ of elements
defining reductions
and every $i \in I$, let $S_i := \{w_j \mid j \in J, \psi(j) = i\}$ be
the set of $w_j$ which reduce scalar multiples of $v_i$ (i.e.,
elements of the form $\lambda v_i$ for $\lambda \in \k$).
Also, let $V_{\prec i} := \Span_{\ell \prec i}(v_\ell)$, where here and below
span means the $\k$-linear span.

\begin{defn} A \emph{confluent set} is a collection $S = (w_j)_{j \in
    J}$ of elements which define reductions such that, for all $i \in
  I$, $\Span(S_i) \cap V_{\prec i} \subseteq \Span(S_\ell)_{\ell \prec
    i}$.  A \emph{confluent spanning set} for $W$ is a confluent set
  $S \subseteq W$ which $\k$-linearly spans $W$.
\end{defn}
\begin{rem}
  Unlike the definition of Gr\"obner generating sets, we have included no
  minimality condition in the definition of confluent spanning sets.
  For example, if $S$ is a confluent spanning set, then any superset
  of $S$ obtained by adjoining scalar multiples of some elements of $S$
  with the same leading monomials (this last condition would be
  automatic if $\k$ were an integral domain) is still a confluent set.
  We have no need for minimality (this would be useful if we wanted minimal
  confluent spanning sets to be unique, but we do not need this).
\end{rem}
Given $S = (w_j)_{j \in J}$, define for all $i \in I$
the ideal $Y_i \subseteq \k$ generated by all $\lambda \in \k$ such that
there is an element $w_j \in S$ of the form
  $w_j = \lambda v_i + \sum_{i' \prec i} \mu_{i'} v_{i'}$ (in particular, when
$w_j$ defines a reduction,  $\psi(j)=i$).

\begin{prop}\label{dl1} (The Diamond Lemma I) For any confluent set
  $(w_j)$, and
  any choices of representatives $R_i \subseteq \k$ of $\k/Y_i$
  (i.e.~$R_i$ maps bijectively to $\k/Y_i$) such that $0 \in
  R_i$ for all (but finitely many) $i$, 
 every class in $V/W$ has a
  unique expression as a linear combination $\sum_i \mu_i v_i$ where
  $\mu_i \in R_i$.
 \end{prop}
\begin{prop}\label{dl2}
  (The Diamond Lemma II) For any confluent spanning set $(w_j)$, one
  may obtain another confluent spanning set $(u_{i,r})$ for $W$ by
  choosing, for each $i \in I$ such that $Y_i \neq (0)$, arbitrary
  elements
  $u_{i,r} \in \Span\ W_i$ such that
  $u_{i,r} = \lambda_{i,r} v_i + \sum_{\ell \prec i} \mu_{i,\ell,r}
  v_\ell$
  such that the $\lambda_{i,r}$ satisfy
  $\ldp \lambda_{i,r} \rdp_{r} = Y_i$.
\end{prop}
\begin{prop} \label{dl3} (The Diamond Lemma III) Again suppose
  $S = (w_j)_{j \in J}$ is a confluent spanning set. Let
  $I_{\neq 1} \subseteq I$ be the subset such that $i \in I_{\neq 1}$ if
  and only if $Y_i \neq (1)$.  Let $I_{\neq 1,0}$ be the subset such
  that $i \in I_{\neq 1,0}$ if and only if $Y_i \neq (1)$ and
  $Y_i \neq (0)$.  Then $V/W$ may be presented as
  $\Span_{i \in I_{\neq 1}}(v_i) / \Span_{i \in I_{\neq 1,0}}(T_i)$,
  where $T_i = \{u_{i,r}\}$ is a set of elements as in Proposition
  \ref{dl2}.
\end{prop}
We only prove the first version, since the other two follow fairly
easily.
\begin{proof}[Proof of Proposition \ref{dl1}]
  We claim that it is enough to assume that $I$ is finite. In general,
  $I$ is an inductive limit of finite subsets $I' \subseteq I$ which are
  downward-closed, i.e., if $i < i'$, then $i' \in I'$ implies
  $i \in I'$. For such $I' \subseteq I$, 
  $W_{I'} := \Span_{i \in I'}(S_{i}) \subseteq V_{I'} := \Span_{i \in
    I'}(v_i)$,
  and then $V$ is the union of such $V_{I'}$, and $W$ is the union of
  such $W_{I'}$.  Hence, it is enough to prove the statement for every
  $(V_{I'}, W_{I'})$.

  So, assume $I$ is finite. We prove the statement by induction on the
  size of $I$. If $|I|=1$, the statement is clear.  Inductively,
  suppose that $i_0 \in I$ is a maximal element and that the result is
  true for the submodules $V' \subsetneq V, W' \subseteq W$ where $V' =
  \Span_{i \neq i_0}(v_i)$ and $W' = \Span_{i \prec i_0}(S_{i})$.
  Then, the result also holds for the pair $(V, W')$.  Next, write
  $V/W = (V/W') / \Span(S_{i_0})$. It is evident that every element of
  $V/W$ is the class of an element of the form $\sum_{i \in I} \mu_i
  v_i$ (for $\mu_i \in R_i$), since for every $y \in Y_{i_0}$, $y
  \cdot v_{i_0} \in (V' + W)$.  It remains to show that every class is
  uniquely represented in this way.  For this, it suffices to prove
  the following claim: If $\sum_{i \in I} \mu_i v_i \in W$ for some
  $\mu_i$, then if $i_1 \in I$ is maximal such that $\mu_{i_1} \neq
  0$, then $\mu_{i_1} \in Y_{i_1}$.  Let us prove the claim by
  induction on $|I|$: that is, let us strengthen the inductive
  hypothesis to include this (it is clear when $|I|=1$).

  The claim is evident for $i_1 = i_0$, since for every element of
  $W$, the coefficient of $v_{i_0}$ lies in $Y_{i_0}$.  More
  generally, if $i_1 \not \prec i_0$, then this follows by the
  inductive hypothesis, since any $\k$-linear combination of elements
  of $S_{i_0}$ has nonzero coefficients only on $v_{i}$ with $i
  \preceq i_0$. It remains to prove the claim for $i_1 \prec i_0$.
  Given any element $\sum_{i \in I} \mu_i v_i$, we may subtract, for
  every maximal $i'$ with $i' \not \prec i_0$, some element of
  $\Span(S_{i'})$, so that the new element $\sum_{i \in I} \mu'_{i}
  v_i$ has $\mu'_{i'} = 0$ for all such $i'$, by the case $i_1 \not
  \prec i_0$ of the claim above. This will not change the set of
  maximal $i_1$ such that $\mu_{i_1} \neq 0$ which instead have the
  property $i_1 \prec i_0$.  Thus, it suffices to assume that, for all
  maximal $i_1$ with $\mu_{i_1} \neq 0$, $i_1 \prec i_0$.  In this
  case, $\sum_{i \in I} \mu_i v_i \in W \cap V_{\prec i_0}$, and by
  the confluence condition, it follows that $\sum_{i \in I} \mu_i v_i
  \in \Span(S_{\prec i_0})$. Now, the claim follows by induction.
\end{proof}

Note that we did \textbf{not} actually describe the abstract module
structure of $V/W$ above: for example, for $\k = \Z$, one could have
$V = \langle v_1, v_2 \rangle$ and
$W = \langle 3v_1 - v_2, 3v_2 \rangle$, where Proposition \ref{dl1}
says that the quotient is set-theoretically the same as sums
$\lambda_1 v_1 + \lambda_2 v_2$ for $\lambda_i \in \{0,1,2\}$. The
abstract $\Z$-module structure, however, is $\Z/9$.

In the situation of Gr\"obner generating sets, as we will explain below, $W$ as
above could be taken to be multiples of a Gr\"obner generating set by monomials
on either side; in particular, $W$ is much larger than an actual
Gr\"obner generating set.

\begin{rem} In the case that $\k$ is a field, then the Diamond Lemma
  says that a basis of $V/W$ is given by $(\bar v_i)_{i \in I'}$ where
  $I' = I_{\neq 1}$ is the set of indices such that $v_i$ does
  not appear as the leading term of any of the relations $(w_j)$ in
  the confluent set.
\end{rem}

\begin{rem} We could have instead taken $\prec$ to be a well-ordering
  (e.g.~a labeling by $\N$ if the module is countably generated) with
  no loss of generality. This is because we may convert any partial
  ordering satisfying the descending chain condition into an arbitrary
  well ordering that preserves all relations $x \prec y$ from the
  partial order (i.e.~such that there is a map of partially ordered
  sets from the original set to the totally ordered one), without
  changing any of the above objects.  (Note that this requires the
  Axiom of Choice in general, and it works because the descending
  chain condition is equivalent to saying that any subset has a
  minimal element.) This would yield exactly the same results and
  proof. We state it in the generality of partially ordered sets
  because that setting is sometimes more convenient.
\end{rem}

Now, let us specialize to the free algebra case.  Let $A = \k\langle
x_1, \ldots, x_m \rangle$ be a free noncommutative algebra generated
by indeterminates $x_i$.  Let $\prec$ be a partial order on the
monomials in the $x_i$'s satisfying the descending chain condition,
such that $f \prec g$ implies $h_1 f h_2 \prec h_1 g h_2$ for any
monomials $h_1, h_2$ (for instance, this is satisfied by the graded
lexicographical ordering). Suppose $B \subseteq A$ is an ideal.  Then we
can define a set $(b_i)$ to be a confluent ideal generating set of $B$
if the elements $(f b_i g)$, for $f,g$ ranging over all monomials in
the $x_i$, form a confluent spanning set for $B$ as an $\k$-module
(with basis the monomials and partial order $\prec$).  To understand
what this means, call the ``leading monomial'' $LM(P)$ of an element
$P \in A$ the highest monomial which appears with nonzero coefficient
with respect to $\prec$, if such a monomial exists and exceeds all
other monomials which have nonzero coefficient.  In order to be
confluent, we first require that each $b_i$ have a leading monomial.
Then, elements of the form $f b_i g$ define reductions, which reduce
certain multiples of $f \cdot LM(b_i) \cdot g$ into linear
combinations of smaller monomials. Then, the confluence condition
says: if $h \in A$ admits multiple reductions, then every linear
combination of elements defining such reductions which has a lower
leading monomial than $h$ is itself a linear combination of elements
defining reductions whose leading monomials are all less than $h$.

Then, in this case, one concludes all of the Diamond Lemma versions. In
the case $\k$ is a field, for example, one finds that $A/B$ has a
basis, as a vector space, given by those monomials in ``normal form'',
which means that they do not contain the leading monomial of any of the
$b_i$'s as a subset (as a word).

Letting $\prec$ be the graded lexicographical ordering (or a variant
as discussed in \S \ref{gbs}), we find that a minimal confluent ideal
generating set (with leading monomials having coefficient one) is a
Gr\"obner generating set for $B$, and recover Proposition \ref{gbp}.  In fact,
at the cost of relinquishing uniqueness, we recover a version of
Gr\"obner generating sets over arbitrary commutative rings (which 
should perhaps be called Gr\"obner generating sets).

\small{
\bibliography{theorem}

\providecommand{\bysame}{\leavevmode\hbox to3em{\hrulefill}\thinspace}
\providecommand{\MR}{\relax\ifhmode\unskip\space\fi MR }
\providecommand{\MRhref}[2]{%
  \href{http://www.ams.org/mathscinet-getitem?mr=#1}{#2}
}
\providecommand{\href}[2]{#2}
\begin{thebibliography}{CBEG07}

\bibitem[AFLS00]{AFLS}
J.~Alev, M.~A. Farinati, T.~Lambre, and A.~L. Solotar, \emph{Homologie des
  invariants d'une alg{\`e}bre de {W}eyl sous l'action dun groupe fini}, J.
  Algebra \textbf{232} (2000), 564--577.

\bibitem[Ani82]{An1}
D.~J. Anick, \emph{Noncommutative graded algebras and their {H}ilbert series},
  J. Algebra \textbf{78} (1982), 120--140.

\bibitem[Ani86]{An2}
\bysame, \emph{On the homology of associative algebras}, Trans. Amer. Math.
  Soc. \textbf{296} (1986), 641--659.

\bibitem[BEG04]{BEG}
Yuri Berest, Pavel Etingof, and Victor Ginzburg, \emph{Morita equivalence of
  {C}herednik algebras}, J. Reine Angew. Math. \textbf{568} (2004), 81--98,
  arXiv:math/0207295.

\bibitem[Ber78]{B}
George~M. Bergman, \emph{The diamond lemma for ring theory}, Adv. in Math.
  \textbf{29} (1978), no.~2, 178--218.

\bibitem[BLB02]{BLB}
Raf Bocklandt and Lieven Le~Bruyn, \emph{Necklace {L}ie algebras and
  noncommutative symplectic geometry}, Math. Z. \textbf{240} (2002), no.~1,
  141--167.

\bibitem[CBEG07]{CBEG}
William Crawley-Boevey, Pavel Etingof, and Victor Ginzburg,
  \emph{Noncommutative geometry and quiver algebras}, Adv. Math. \textbf{209}
  (2007), 274--336, math.AG/0502301.

\bibitem[CBH98]{CBH}
William Crawley-Boevey and Martin~P. Holland, \emph{Noncommutative deformations
  of {K}leinian singularities}, Duke Math. J. \textbf{92} (1998), no.~3,
  605--635.

\bibitem[EE05]{EE}
Pavel Etingof and Ching-Hwa Eu, \emph{Koszulity and the {H}ilbert series of
  preprojective algebras}, math.RA/0512287, 2005.

\bibitem[EE07]{EE2}
\bysame, \emph{Hochschild and cyclic homology of preprojective algebras of
  {ADE} quivers}, arXiv:math/0609006, 2007.

\bibitem[EG06]{EG}
Pavel Etingof and Victor Ginzburg, \emph{Noncommutative complete intersections
  and matrix integrals}, math.AG/0603272, 2006.

\bibitem[EK96]{EK}
Pavel Etingof and David Kazhdan, \emph{Quantization of {L}ie bialgebras {I}},
  Selecta Math (N.S.) \textbf{2} (1996), no.~1, 1--41.

\bibitem[Gin01]{G}
Victor Ginzburg, \emph{Non-commutative symplectic geometry, quiver varieties,
  and operads}, Math. Res. Lett. \textbf{8} (2001), no.~3, 377--400.

\bibitem[GP79]{GP}
I.M. Gelfand and V.A. Ponomarev, \emph{Model algebras and representations of
  graphs}, Func. Anal. and Applic. \textbf{13} (1979), no.~3, 1--12.

\bibitem[GS10]{GSncbv}
V.~Ginzburg and T.~Schedler, \emph{Differential operators and {BV} structures
  in noncommutative geometry}, Selecta Math. (N.S.) \textbf{16} (2010), no.~4,
  673--730, arXiv:0710.3392. \MR{2734329 (2012k:18021)}

\bibitem[GS12]{GScyc}
Victor Ginzburg and Travis Schedler, \emph{Free products, cyclic homology, and
  the {G}auss--{M}anin connection}, Adv. Math. \textbf{231} (2012), no.~3-4,
  2352--2389, arXiv:0803.3655. \MR{2964640}

\bibitem[Hes97]{H}
Lars Hesselholt, \emph{{W}itt vectors of non-commutative rings and topological
  cyclic homology}, Acta Mathematica \textbf{178} (1997), 109--141.

\bibitem[Hes05]{H2}
\bysame, \emph{Correction to ``{W}itt vectors of non-commutative rings and
  topological cyclic homology''}, Acta Mathematica \textbf{195} (2005), 55--60.

\bibitem[MOV06]{MOV}
Anton Malkin, Viktor Ostrik, and Maxim Vybornov, \emph{Quiver varieties and
  {L}usztig's algebra}, Adv. Math. \textbf{203} (2006), no.~2, 514--536,
  arXiv:math/0403222.

\bibitem[Sch05]{S}
Travis Schedler, \emph{A {H}opf algebra quantizing a necklace {L}ie algebra
  canonically associated to a quiver}, Int. Math. Res. Not. (2005), no.~12,
  725--760, IMRN/14217.

\bibitem[Sch07]{Sarx}
T.~Schedler, \emph{Hochschild homology of preprojective algebras over the
  integers}, arXiv:0704.3278v1, 2007.

\bibitem[VdB08]{VdB}
M.~Van~den Bergh, \emph{Double {P}oisson algebras}, Trans. Amer. Math Soc.
  \textbf{360} (2008), no.~11, 5711--5769, arXiv:math.QA/0410528.

\bibitem[Wit37]{W}
E.~Witt, \emph{Zyklische {K\"o}rper und {A}lgebren der {C}harakteristik {$p$}
  vom grad {$p^n$}}, J. Reine Angew. Math. \textbf{176} (1937), 126--140.

\end{thebibliography}
\bibliographystyle{amsalpha}
}
\end{document}